\documentclass[11 pt, reqno]{article}
\usepackage{amsmath} 
\usepackage{amssymb}
\usepackage{amsthm}
\usepackage{cite}
\usepackage{graphicx, color}
\usepackage{xspace}
\usepackage[a4paper, left=2cm, right=2cm, top=2 cm, bottom= 2 cm]{geometry}
\usepackage[colorlinks=true,hyperindex=true]{hyperref}
\usepackage{comment}

\newtheorem{ccounter}{ccounter}[section]
\newtheorem{thm}[ccounter]{Theorem}
\newtheorem{lem}[ccounter]{Lemma}
\newtheorem{cor}[ccounter]{Corollary}
\newtheorem{defn}[ccounter]{Definition}
\newtheorem{prop}[ccounter]{Proposition}
\newtheorem{ass}[ccounter]{Assumption}
\newtheorem{ex}[ccounter]{Example}

\numberwithin{equation}{section}

\def\cH{{\mathcal H}}

\def\cG{{\mathcal G}}

\newcommand{\La}{\Lambda}
\newcommand{\eps}{\varepsilon}

\newcommand{\be}{\begin{equation}}
\newcommand{\ee}{\end{equation}}
\newcommand{\ben}{\be\nonumber}
\newcommand{\ind}[1]{\textbf{1}_{#1}}

\newcommand{\ga}{{\gamma}}

\newcommand{\la}{\lambda}
\newcommand{\Om}{{\Omega}}
\newcommand{\om}{{\omega}}

\renewcommand{\th}{\theta}

\renewcommand{\P}{\mathbb{P}}
\newcommand{\E}{\mathbb{E}}
\newcommand{\R}{\mathbb{R}}
\newcommand{\C}{\mathbb{C}}
\newcommand{\N}{\mathbb{N}}

\def\fb{\mathfrak{b}}
\def \fm{\mathfrak{m}}
\def \fw{\mathfrak{w}}
\def\frb{\mathfrak{b}}
\def\frc{\mathfrak{c}}
\def\gXY{\gamma^{X \boxplus Y}}
\def\hatgam{\hat{\gamma}}
\def\gamhat{\hat{\gamma}}
\def\rhosc{\rho_{\mathrm{sc}}}
\def\Gkap{\mathcal{G}_\kappa}
\def\F{\mathcal{F}}
\def\bbE{\mathbb{E}}

\newcommand{\abs}[1]{\lvert #1 \rvert}
\newcommand{\absa}[1]{\left\lvert #1 \right\rvert}
\newcommand{\norm}[1]{\lVert #1 \rVert}
\newcommand{\norma}[1]{\left\lVert #1 \right\rVert}

\DeclareMathOperator{\diag}{diag}
\DeclareMathOperator{\tr}{Tr}
\DeclareMathOperator{\var}{Var}
\DeclareMathOperator{\supp}{supp}

\DeclareMathOperator{\re}{Re}
\DeclareMathOperator{\im}{Im}

\DeclareMathOperator{\dd}{d}
\DeclareMathOperator{\ii}{i}

\def\bet{\begin{thm}}
	\def\eet{\end{thm}}
\def\bel{\begin{lem}}
	\def\eel{\end{lem}}
\def\bas{\begin{ass}}
	\def\eas{\end{ass}}
\def\bec{\begin{cor}}
	\def\eec{\end{cor}}
\def\bed{\begin{defn}}
	\def\eed{\end{defn}}
\def\bep{\begin{prop}}
	\def\eep{\end{prop}}
\def\beq{\begin{equation}}
\def\eeq{\end{equation}}
\def\bea{\begin{equation*}}
\def\eea{\end{equation*}}
\def\tr{\mathrm{tr}}
\def\bex{\begin{ex}}
	\def\eex{\end{ex}}

\def\1{\boldsymbol{1}}
\def\Im{\mathrm{Im}}
\def\Re{\mathrm{Re}}
\def\e{\mathrm{e}}
\def\i{\mathrm{i}}
\def\del{\partial}
\def\d{\mathrm{d}}
\def\eps{\varepsilon}
\def\om{\omega}
\def\fa{\mathfrak{a}}

\def\hatxi{\hat{\xi}}
\def\O{\mathcal{O}}

\def\G{\mathcal{G}}
\def\hatu{\hat{u}}
\def\hatB{\hat{B}}
\def\UB{\mathcal{U}^{\hat{B}}}

\def\A{\mathcal{A}}
\def\hatz{\hat{z}}

\def\fae{\mathfrak{a}}
\def\mfb{\mathfrak{b}}
\def \cx{\mathfrak{c}_\textsc{x}}

\title{Local spectral statistics of addition of random matrices}
\author{Ziliang Che \and Benjamin Landon}
\date{}

\begin{document}

\begin{table}
\centering
\begin{tabular}{c}
\multicolumn{1}{c}{\Large{\bf Local spectral statistics of the addition of random matrices}}\\
\\
\\
\end{tabular}
\begin{tabular}{c c  c}
Ziliang Che &   &  Benjamin Landon\\
\\
\multicolumn{3}{c}{ \small{Department of Mathematics} } \\
 \multicolumn{3}{c}{ \small{Harvard University} } \\
 \\
\small{zche@math.harvard.edu} & & \small{landon@math.harvard.edu}  \\
\\
\end{tabular}
\\
\begin{tabular}{c}
\multicolumn{1}{c}{\today}\\
\\
\end{tabular}

\begin{tabular}{p{15 cm}}
\small{{\bf Abstract:} We consider the local statistics of $H = V^* X V + U^* Y U$ where $V$ and $U$ are  independent Haar-distributed unitary matrices, and $X$ and $Y$ are deterministic real diagonal matrices.   In the bulk, we prove that the gap statistics and correlation functions coincide with the GUE in the limit when the matrix size $N \to \infty$ under mild assumptions on $X$ and $Y$.  Our method relies on running a carefully chosen diffusion on the unitary group and comparing the resulting eigenvalue process to Dyson Brownian motion.  Our method also applies to the case when $V$ and $U$ are drawn from the orthogonal group.  Our proof relies on the local law for $H$ proved in \cite{BES16,BES16a,BES16b} as well as the DBM convergence results of \cite{LY, fixed}. }
\end{tabular}

\end{table}

{
\hypersetup{linkcolor=black}
}

\newpage
\section{Introduction}
In this paper we study the spectrum of matrices formed by adding two large $N \times N$ `generic' Hermitian matrices,
\begin{align*}
A+B.
\end{align*}
A natural probabilistic way of interpreting  `generic' is to take a matrix with a general empirical eigenvalue measure and conjugate it by a random unitary matrix.  We are led to the model
\begin{align*}
H = V^* X V + U^* Y U,
\end{align*}
where $X$ and $Y$ are deterministic real diagonal $N \times N$ matrices and $U$ and $V$ are independently drawn from the Haar measure on the unitary group.


The limiting global eigenvalue density of $H=V^*XV+U^*YU$ was first obtained by Voicelescu in the influential work \cite{Vo91}, in which he proved that the normalized empirical distribution converges weakly to the free convolution of the limiting empirical laws of $X$ and $Y$. 
 This result was then subsequently  obtained via several different approaches in \cite{Bi98, Co03, PV00, Sp93}.  The first result going beyond weak convergence was that of Kargin \cite{Ka12} who showed that under suitable assumptions on $X$ and $Y$, that convergence holds not only at the global scale, but on the scale $( \log(N) )^{-\frac{1}{2}}$.   Kargin later improved this result to the scale $N^{-\frac{1}{7}}$ in \cite{Ka15}.  

  In a series of works, Bao, Erd{\"o}s and Schnelli \cite{BES16, BES16a, BES16b} established the important result that in the bulk of the spectrum, the empirical eigenvalue distribution of $H$ converges to a deterministic quantity down to the optimal scale $N^{-1 +\nu}$, using a specific decomposition of the Haar measure and a sophisticated analysis of the Green's function.  Moreover, by implementing a new fluctuation averaging mechanism, they were able to obtain the optimal error rate in \cite{BES16a}.  Results of this form in random matrix theory are known as {\it local laws}. 
  



The results of \cite{BES16,BES16a,BES16b} control the eigenvalue behaviour down to the scale $N^{-1+\nu} \gg N^{-1}$.  In the bulk of the spectrum the eigenvalue spacing is of order $N^{-1}$ and the behaviour of the eigenvalues at this scale remains unstudied.  In this paper we wish to investigate the local statistics of the eigenvalues of $H$ at the scale $N^{-1}$  by determining the limiting distribution of the eigenvalues gaps as well as the limiting correlation functions.  

One of the central tenets of random matrix theory is the Wigner-Dyson-Mehta universality conjecture.  This conjecture asserts that the local statistics for wide classes of random matrix ensembles exhibit universality, in that the local statistics do not depend on the underlying details of the matrix ensemble, but only on the symmetry class (real symmetric vs. complex Hermitian) of the ensemble.  In particular, in the limit $N \to \infty$, the local statistics should coincide with the Gaussian Orthogonal and Unitary ensembles (GOE/GUE) for which explicit formulas are known.

Wigner matrices are perhaps the most fundamental class of random matrix ensembles, and form a natural class on which to study the bulk universality conjecture.  A Wigner matrix is constructed by taking the entries to be independent (up to the Hermitian $H=H^*$ constraint) centered random variables with variance $1/N$.  Bulk universality for Wigner matrices was established in a series of works \cite{EKYY12,EKYY13, EPR10, ESY11, EY15, EYY12, wigfixed};  parallel results were established for special cases in \cite{TV10,TV11}. 

In addition to proving universality for Wigner matrices, the works \cite{EKYY12,EKYY13, EPR10, ESY11, EY15, EYY12, wigfixed} established universality for the generalized Wigner class as well as the adjacency matrices of certain classes of random graphs.  Moreover, these works established a robust three-step strategy to proving universality for random matrix ensembles.  

The success of this three-step strategy is seen in the recent progress in random matrix theory of proving universality for various matrix ensembles.  Going beyond the class of generalized Wigner matrices, universality has been proven for the adjacency matrices of sparse random graphs \cite{BHKY15,HLY15,AC15,BKY},  random matrices of general Wigner type \cite{ajanki1,ajanki2,ajanki3}, matrices with correlated entries \cite{AEK16,Ch16}, deformed Wigner ensembles \cite{schnelli1,schnelli2} as well as a class of random band matrices \cite{band}.  

Our main result is to prove bulk universality for $H$, under mild assumptions on $X$ and $Y$.  We prove that in the limit $N \to \infty$ the gap statistics and correlation functions coincide with those of the GUE when $V$ and $U$ are drawn from the unitary group, and those of the GOE when they are drawn from the orthogonal group.  

Previous works on universality have relied heavily on Dyson Brownian motion (DBM).  DBM is a stochastic process on random matrices which leaves the GOE/GUE invariant; Dyson's seminal calculation \cite{Dyson} shows that under this flow, the eigenvalues satisfy a closed system of stochastic differential equations.  DBM was first used to prove universality of Wigner matrices in the works \cite{localrelax1,ESY11}, by showing that the time to local equilibrium is $t \sim N^{-1}$ when the initial data of the process is a Wigner matix.  Later, this local equilibrium was established in the strong fixed energy sense in \cite{wigfixed} for Wigner matrices.  Recently, the works \cite{ES,LY,fixed} have gone beyond the Wigner class, and showed that the time to local equilibrium is order $N^{-1}$ for a wide class of initial data.  The role played by the works \cite{ES,LY,fixed} in the recent progress on bulk universality is that, after proving a local law for a given ensemble, the works \cite{ES,LY,fixed} immediately yield bulk universality for the original ensemble at the expense of adding a small Gaussian component.  This small Gaussian component is then removed using a perturbation argument exploiting the matrix structure based around either It{\^o}'s lemma \cite{QUE} or a Green's function comparison theorem.

In the ensemble considered here, the initial local law estimate is a consequence of the works \cite{BES16,BES16a,BES16b} and so one could attempt to proceed by applying \cite{ES,LY,fixed}.  
However, the perturbation methods usually used to remove the Gaussian component fail in the case considered here.  The matrix $H$ lacks the ``Wigner-type'' structure of previously considered models, and the addition of a Gaussian component is a singular perturbation which does not respect the structure of the ensemble.

Instead, we exploit the symmetry of our model, the translation invariance of the Haar measure.  We take $U(t)$ to be a diffusion process on the unitary group and define
\begin{align*}
H (t) := V^* X V + U (t)^* Y U(t).
\end{align*}
Note that by the translation invariance of the Haar measure, $H(t)$ has the same eigenvalue distribution as $H$ for any choice of $U(t)$ independent from $V$.  By carefully choosing the weights of the diffusion $U(t)$ in certain directions, we derive a stochastic differential equation for the eigenvalues,
\beq \label{eqn:pertdbm}
	\dd \la_i = \frac{\dd B_i}{\sqrt{N}} +\frac{\beta}{2N}\sum_{j\neq i} \frac{1+o(1)}{\la_i-\la_j}\dd t + o(1)\dd t
\eeq
which can be view as a perturbed version of the usual DBM process, starting from the same initial data,
\ben
	\dd \mu_i = \frac{\dd B_i}{\sqrt{N}}+ \frac{\beta}{2N}\sum_{j\neq i} \frac{1}{\mu_i-\mu_j} \dd t , \qquad \mu_i (0) = \lambda_i (0) .
\ee
The local law for $H$ provides a priori estimates for the error terms arising in \eqref{eqn:pertdbm}.  In particular, for bulk eigenvalues, the optimal estimates are provided by the results of \cite{BES16,BES16a,BES16b}.  However, our approach to analyzing \eqref{eqn:pertdbm} requires at least a weak global bound valid for all eigenvalues, for which the results of \cite{BES16,BES16a,BES16b} do not apply.  A crucial component of our work is establishing a local law as well as stability estimates down to a scale $N^{ -c}$ for some $c>0$, which is valid throughout the entire spectrum of $H$.

With these a priori estimates in hand we compare \eqref{eqn:pertdbm} to the usual DBM.  The work \cite{wigfixed} introduced the important idea of {\it coupling} two Dyson Brownian motions.  We use this idea and 
set the Brownian motions appearing in the two systems of SDEs equal to each other.  As observed in \cite{wigfixed}, the difference $\lambda_i - \mu_i$ then satisfies a discrete parabolic equation.  In our case there is a forcing term and we have zero initial data.  Using parabolic equation techniques, we are  able to prove that at later times $t$ the difference $\lambda_i (t) - \mu_i (t)$ is $o ( N^{-1})$.  Hence, the local eigenvalue statistics of $H$ can instead be computed from the usual DBM process started from an ensemble related to $H$.  The main result of \cite{fixed} then says that the local statistics of this DBM coincide with that of the corresponding Gaussian ensemble.

The well-posedness of Dyson Brownian motion is non-trivial and in fact $\beta=1$ is in some sense critical for this system.  For $\beta \geq 1$, the eigenvalues do not collide under the DBM flow, whereas for $\beta <1$ they do, and the system is not well-posed.  In the case $\beta =1 $ there is therefore difficulty in establishing that the $\lambda_i (t)$ coming from $H (t)$ satisfy the equation \eqref{eqn:pertdbm}, as the $o(1)$ appearing above the repulsive $1/ | \lambda_i - \lambda_j |$ interaction term in fact comes with a minus sign.  This means that effectively $\beta <1$, and therefore it is nontrivial to justify \eqref{eqn:pertdbm}.  We resolve this by adding a tiny Gaussian component to $X$ which results in level repulsion bounds, making the terms on the RHS of \eqref{eqn:pertdbm} integrable.  This allows us to prove that the $\lambda_i$ in fact are a strong solution to \eqref{eqn:pertdbm}.

We outline the rest of the paper. In Section \ref{sec: main}, we define the model, introduce the assumptions, then sketch the proof of the main theorem.  In Section \ref{sec: estimates} we prove some estimates that will be used later on.  In Section \ref{sec: DBM} we analysis the SDE of the eigenvalues within a short time and prove the main result. In Section \ref{sec: wpness} we prove the well-posedness of the SDE.

\vspace{5 pt}
\noindent{\bf Acknowledgements.} The authors would like to thank H.-T. Yau and P. Sosoe for useful and enlightening discussions. 

\section{Model and main results}\label{sec: main}

\subsection{Definition of model and assumptions}\label{subsec: ass}

We consider $H=V^*XV+U^*YU$ where $X=\diag(x_1,\cdots,x_N)$ and $Y=\diag(y_1,\cdots,y_N)$ are deterministic diagonal matrices while $V$ and $U$ are unitary matrices independently drawn from the unitary group equipped with the Haar measure. Denote by $\la_1\leq \cdots\leq \la_N$ the eigenvalues of $H$.

We assume that there is a positive universal constant $K>0$ such that 
\be
	\sup_{1\leq k\leq N}\abs{x_k}\vee\abs{y_k}\leq K\,.
\ee
Denote $\mu_{1,N}:=\frac{1}{N} \sum_k \delta_{x_k}$ and $\mu_{2,N}:= \frac{1}{N}\sum_k \delta_{y_k}$. We assume that, as $N\to\infty$, the discrete measures $\mu_{1,N}$ and $\mu_{2,N}$ converge weakly to probability measures $\mu_1$ and $\mu_2$, respectively, which are supported on the bounded interval $[-K, K]$.  We assume that $\mu_2$ has an continuous density, i.e., there is a continuous function $\rho_2$ such that
\ben
	\mu_2(\dd y)=\rho_2(y)\dd y\,.
\ee
For technical reasons, we require $\mu_2$ to behave `more or less' like the square root edge of the semicircle law. In particular, we assume that there are constants $\mathfrak{c}>0$, $\delta_0>0$ such that for any $\xi\in\supp \mu_2$ and $0 \leq h \leq \delta_0$,
\be\label{edge behavior}
\mu_2([\xi-h,\xi+h]) \geq h^{2-\mathfrak{c}}\,.
\ee
To illustrate the meaning of this assumption, a typical example is that, if $\mu_2$ is of the form $\mu_2(\dd x) \asymp \sqrt{[(a-x)(x-b)]_+}\dd x$ for some $a<b$, then $\mu_2$ satisfies the assumption with any $\mathfrak{c}\in(0,1/2)$ and small enough $\delta_0>0$. Moreover, we assume that either $\mu_1$ or $\mu_2$ has a bounded Stieltjes transform.

In addition to weak convergence, we impose stronger regularity of $\mu_Y$.  Denote the $k$-th $N$-quantile of $\mu_1$ by $x_k^\star$ and the $k$-th $N$-quantile of $\mu_2$ by $y^\star_k$, i.e.,
\be\label{def: quantile}
	x^\star_k = \inf\{s: \int_{-\infty}^{s} \mu_1(\dd x) = k/N\}\,,\quad y^\star_k = \inf\{s: \int_{-\infty}^{s} \mu_2(\dd y) = k/N\}\,.
\ee
One important observation is that the $y^\star_k$'s do not get too close to each other because of the boundedness of the density of $\mu_2$. In particular, for any $k,l\leq N$, we have
\be\label{quantile space}
	\abs{y^\star_k-y^\star_l}\geq \frac{\abs{k-l}}{\norm{\rho_2}_\infty N}\,.
\ee
We assume that $y_k$ is not too far from $y_k^\star$, in the sense that for any $c>0$, 
\be\label{reg y}
	\sup_{0\leq k\leq N} \abs{y_k-y^\star_k} \leq N^{-1+c}\,, \text{ for large enough $N$}\,.
\ee
This condition can be relaxed, for example, by allowing a small number of $y_k$'s to violate the above inequality, i.e., near the spectral edges of $Y$. However, in this paper we refrain from exploring the optimal condition, for the transparency of argument.  Condition \eqref{reg y} together with \eqref{quantile space} yields, for any $c>0$ and large enough $N$,
\be\label{y space}
	\abs{y_k-y_l}\geq \frac{\abs{k-l}-N^c}{\norm{\rho_2}_\infty N} \,.
\ee
This bound is useful when $\abs{k-l}> N^c$.

We also impose regularity of $\mu_X$, slightly stronger than weak convergence. We assume that there is a constant $\cx>0$ such that 
\beq \label{reg x}
\left| \frac{1}{N} \sum_{i=1}^N \frac{1}{ x_i - E - \i \eta } - \int \frac{ \d \mu_1 (x) } { x - E - \i \eta } \right| \leq N^{ - \cx }
\eeq
for $\eta \geq N^{ - \cx }$.


\begin{remark} The condition \eqref{reg y} is used only to prove \eqref{y space} as well as a polynomial speed of convergence of the Stieltjes transform of a matrix closely related to $Y$ for $\eta \geq N^{ - c}$ for some $c >0$ (c.f., Proposition \ref{prop: m1-fm1}).  The estimate \eqref{y space} often holds under weaker assumptions than \eqref{reg y} (i.e., near the spectral edges where eigenvalues have wider spacing) and the result of Proposition \ref{prop: m1-fm1} is easy to check in practice.  We have refrained from exploring optimal conditions on $Y$ for transparency of the argument.
\end{remark}

\subsection{Main results}
It is known that as $N\to\infty$, the empirical law $\mu_N:= \frac{1}{N} \sum_k \delta_{\la_k}$ converges to the free additive convolution of $\mu_1$ and $\mu_2$. We denote the free convolution of $\mu_1$ and $\mu_2$ by 
\be
	\mu:=\mu_1 \boxplus \mu_2\,.
\ee
A more precise definition will appear later. 
We denote its density by $\rho$ and its classical eigenvalue locations by $\gamma_i$.  We denote the $k$-point correlation functions of $H$ by $p^{(k)}_H$ and by $p^{(k)}_G$ those of the corresponding Gaussian ensemble (GOE for $\beta=1$ and GUE for $\beta=2$).  

\bet \label{thm:bu} Let $X=\diag\{x_1,\cdots,x_N\}$, $Y=\diag\{y_1,\cdots,y_N\}$ where $(x_k)$ and $(y_k)$ satisfy the assumptions in Section \ref{subsec: ass}. Let  $H=V^*XV+U^*YU$ where $V$ and $U$ are independently drawn from the Haar measure on the unitary group $ \mathcal{U}_N$ (or the orthogonal group $\mathcal{O}_N$). Let $I = (a, b)$ be an interval on which the density $\rho$ of $\mu$ is strictly positive.

Then, for each $E \in I$ we have bulk universality.  For any smooth test function $O$,
\begin{align}
\bigg| &\int O ( \alpha_1, \cdots \alpha_k ) p^{(k)}_H \left( E + \frac{ \alpha_1}{N \rho (E) } , \cdots ,  E + \frac{ \alpha_k}{N \rho (E) } \right) \d \alpha \notag\\
- &\int O ( \alpha_1, \cdots \alpha_k ) p^{(k)}_G \left( E' + \frac{ \alpha_1}{N \rhosc (E') } , \cdots ,  E' + \frac{ \alpha_k}{N \rhosc (E') } \right) \d \alpha \bigg| \leq N^{ - c}
\end{align}
for some $c>0$ and any $E' \in (-2, 2)$.
Let $i$ be an index such that $\gamma_i \in (a + \kappa, b - \kappa )$ for some fixed $\kappa >0$.  We then have gap universality at the index $i$.
\begin{align}
 \bigg|  &\bbE^{(H)} \left[ O ( N \rho (\gamma_i ) (\lambda_{i+1} - \lambda_i ), \cdots ,N \rho (\gamma_{i} ) (\lambda_{i+k} - \lambda_{i+k-1} )  \right]  \notag\\
 - &\bbE^{(G)} \left[ O ( N \rho (\gamma_{j, \mathrm{sc}} ) (\lambda_{j+1} - \lambda_j ), \cdots ,N \rho (\gamma_{j, \mathrm{sc}} ) (\lambda_{j+k} - \lambda_{j+k-1} )  \right] \bigg| \leq N^{-c}
\end{align}
where $\bbE{(H/G)}$ denotes expectation w.r.t. $H$ or the corresponding Gaussian ensemble, respectively.  Here $j$ is any index in $\kappa N \leq j \leq (1- \kappa ) N$ for $\kappa >0$ and $\gamma_{j, \mathrm{sc}}$ denote the classical eigenvalue locations of the semicircle law.
\eet

\subsection{Sketch of proof}\label{subsec: formal derivation}

Since the law of $VU^*$ is still the Haar measure and the eigenvalues of $H$ are invariant under conjugation by $U$, it is sufficient to prove Theorem \ref{thm:bu} for $U$ being any random unitary matrix independent of $V$.  The strategy is roughly as follows. We run a Brownian motion $U(t)$ on the unitary group with certain weights in different directions with initial value $U(0)=I$.   We take $U(t)$ to be independent from $V$. Then we define 
\be
	H(t):= V^*XV+U(t)^*YU(t)\,.
\ee 
We derive an SDE for the eigenvalues $(\la_1(t),\cdots,\la_N(t))$ of $H(t)$.  By a judicious choice of $U(t)$, the SDE turns out to be very similar to Dyson Brownian motion with $\beta=2$ (in the case where $V$ and $U$ are orthogonal matrices, we get an SDE similar to Dyson Brownian motion with $\beta=1$.)  A careful analysis of this SDE yields that at time 
\be\label{def: T}
	T=N^{-1+\mathfrak{b}}\,,
\ee
where $\mathfrak{b}>0$ is a small constant to be chosen (it will depend on the assumptions on $X$ and $Y$), the eigenvalue statistics of $H(T)$ coincide with those of DBM started with initial data $H(0)$ (to be more precise, it will turn out that we need to make a slight modification to the initial data $H(0)$ of the usual DBM process).   The main result of \cite{fixed} states that the local statistics of this process then coincide with the GUE and so we conclude that bulk universality holds for $H(T)$. This immediately implies the same result for $H(0)$ because the law of eigenvalues of $H(T)$ is the same as that of $H(0)$. In this subsection, we define the Brownian motion on the unitary group, then formally derive the dynamics of eigenvalues.  The well-posedness of the SDEs in concern will be handled in Section \ref{sec: wpness}.

The distribution of $(\la_1,\cdots,\la_N)$ is unaffected if we let $U$ be any random unitary matrix with a law independent from $V$. Therefore, without loss of generality, we consider $U(0)=I$ and let $U(t)$ solve the following SDE:

\be\label{eq: du}
\dd U(t) = \ii \dd W(t) U(t) -\tfrac{1}{2}A U(t)\dd t\,.
\ee
Here $W$ and $A$ are defined as follows.  $W=(W_{\alpha\beta})_{\alpha\neq \beta}$ is a family of independent complex-valued Brownian motions (up to the Hermitian constraint $W_{\alpha\beta}^* = W_{\beta\alpha}$) with quadratic variation process
 $$\langle W_{\alpha\beta}^*,W_{\alpha\beta}\rangle_t=N^{-1}\sigma_{\alpha\beta}^2 t$$ 
 where $\sigma_{\alpha\beta}$ are deterministic and to be chosen (they will later be chosen to be a function of the $y_k$'s). The matrix $A$ on the right hand side of \eqref{eq: du} is a deterministic diagonal matrix with diagonal entries given by 
\be\label{eq: A}
	A_{\alpha\alpha} = N^{-1}\sum_\beta \sigma_{\alpha\beta}^2\,.
\ee
  By standard results (e.g. Theorem H.6 in \cite{AGZ}), the SDE \eqref{eq: du} has a unique strong solution $U(t)$.  By differentiating $U(t)^*U(t)$ using It{\^ o}'s formula, one easily sees that the solution $U(t)$ stays on the unitary group.

We differentiate $H$ using It{\^o}'s formula to see

\be \label{eq: dH1}
\dd H = \ii U^* [Y, \dd W] U -U^*AYU\dd t + U^* \dd\langle W, YW\rangle_t U \,,
\ee
Here $[\cdot,\cdot]$ means the commutator of two matrices and $\langle\cdot,\cdot\rangle_t$ denotes the quadratic covariation of two processes.  Recalling that $\dd\langle W_{\alpha\beta}^*,W_{\alpha\beta}\rangle_t=N^{-1}\sigma_{\alpha\beta}^2 \dd t$, we have

\be	
\dd H = \ii U^* ( (y_\alpha-y_\beta)\dd W_{\alpha\beta}) U + U^* (\delta_{\alpha\beta}N^{-1}\sum_\gamma \sigma_{\alpha\gamma}^2(y_\gamma-y_\alpha)) U\dd t\,.
\ee 
Here we used the notation that $(a_{\alpha\beta})$ represents the $N$ by $N$ matrix whose $(\alpha,\beta)$-th entry equals $a_{\alpha\beta}$.  Denote 
\be\label{def: haty}
	\hat{y}_\alpha := N^{-1}\sum_\gamma \sigma_{\alpha\gamma}^2(y_\gamma-y_\alpha)\,,\quad\hat{Y}:= (\delta_{\alpha\beta} \hat{y}_\alpha)_{\alpha\beta}\,,\quad \th_\alpha :=\sum_\gamma \sigma_{\alpha\gamma}^2\,.
\ee 
Equation \eqref{eq: dH1} becomes

\be	\label{eq: dH2}
\dd H = \ii U^* ( (y_\alpha-y_\beta)\dd W_{\alpha\beta}) U + U^* \hat{Y} U\dd t\,.
\ee 
Let $0<\mathfrak{b}<\mathfrak{a}<1$ be two small constants such that $\mathfrak{b}\leq \mathfrak{a}/100$. Define
\ben\label{def: sig}
\sigma_{\alpha\beta}=
  \begin{cases}
  	\abs{y_\alpha-y_\beta}^{-1}, & \text{for } \abs{\alpha-\beta}\geq N^{\mathfrak{a}}\,; \\
  	0, & \text{for } \abs{\alpha-\beta}< N^{\mathfrak{a}}\,.
  \end{cases}
\ee
In view of the lower bound \eqref{y space}, we see that the deterministic diagonal matrix $A$ appearing in \eqref{eq: du} (and defined by \eqref{eq: A}) satisfies
\be\label{A bound}
A_{\alpha\alpha}\lesssim \frac{1}{N}\sum_{\abs{\alpha-\beta}\geq N^\mathfrak{a}} \frac{1}{\abs{\alpha-\beta}^2N^{-2}} \lesssim N^{1-\mathfrak{a}}\,.
\ee
As we will show later on, the above bound on $A_{\alpha\alpha}$ leads to the fact that $\norma{U(t)-I}\ll 1$ with high probability, whenever $t\ll N^{-1+\mathfrak{a}}$. We also have a bound for $\hat{Y}$,

\be\label{hatY bound}
	\norm{\hat{Y}} = \max_\alpha \abs{\hat{y}_\alpha} \lesssim \sum_{\abs{\alpha-\beta}\geq N^\mathfrak{a}} \frac{1}{\abs{\alpha-\beta}/N}\lesssim \log N\,.
\ee
To proceed, we introduce the notion of Hermitian Brownian motions:
\begin{defn}\label{def: HBM}
	An $N\times N$ matrix-valued stochastic process $B=(B_{\alpha\beta})_{1\leq \alpha,\beta\leq N}$ is called a Hermitian Brownian motion if 
	\begin{enumerate}
		\item $(B_{\alpha\beta})_{\alpha<\beta}$ are independent standard complex Brownian motions..
		\item $(B_{\alpha\alpha})_{1\leq \alpha\leq N}$ are independent standard real Brownian motions.
		\item $(B_{\alpha\beta})_{\alpha<\beta}$ and  $(B_{\alpha\alpha})_{1\leq \alpha\leq N}$ are independent from each other.
	\end{enumerate}
\end{defn}
Let $B=(B_{\alpha\beta})_{1\leq \alpha\beta \leq N}$ be a Hermitian Brownian motion such that 
\be\label{def: B}
	\tfrac{1}{\sqrt{N}}B_{\alpha\beta}=\ii(y_\alpha-y_\beta)W_{\alpha\beta}\,,\text{ for }\abs{\alpha-\beta}\geq N^{\mathfrak{a}}\,.
\ee
Therefore, \eqref{eq: dH1} becomes
\be\label{eq: dH3}
	\dd H = \frac{1}{\sqrt{N}}\dd \hat{B} + \frac{1}{N}U^*\hat{Y}U\dd t - \frac{1}{\sqrt{N}}U^*(\dd B_{\alpha\beta} \ind{\abs{\alpha-\beta}<N^\mathfrak{a}})U\,.
\ee
Here $\hat{B}(t): = \int _0^tU(s)^*\dd B(s) U(s) $.  It is easy to see that $\hat{B}$ is also a Hermitian Brownian motion.

For technical reasons, the drift term $U^*\hat{Y}U\dd t$  will produce error terms which are difficult to handle at the level of the eigenvalue dynamics.  We therefore consider an alternative process $\tilde{H}(t)$, defined by 
\be
	\tilde{H}(t):= H(t)+ (T-t)\hat{Y}\,.
\ee
It is easy to see that $\tilde{H}(t)$ is a process with initial value
\be
	\tilde{H}(0)= H(0)+T\hat{Y}
\ee
and satisfies the SDE, 
\be\label{eq: dtH}
	\dd \tilde{H}=  \frac{1}{\sqrt{N}}\dd \hat{B} +\frac{1}{N}(U^*\hat{Y}U-\hat{Y})\dd t -\frac{1}{N}U^*(\dd B_{\alpha\beta} \ind{\abs{\alpha-\beta}<N^\mathfrak{a}})U\,.
\ee
Since $H(T) = \tilde{H} (T)$ it will suffice to consider the latter process.  
The advantage of dealing with the $\tilde{H}(t)$ process instead of $H(t)$ is that the $\hat{Y}$ terms can be handled using the matrix estimate $||U(t) -  1|| \ll 1$ which we derive below.
For simplicity denote
\be
	\dd B_k:= \dd \hat{B}_{kk}\,.
\ee
Let $a_k:=(a_{\alpha k})_{1\leq a\leq N}(t)$ be the eigenvector associated to the $k$-th smallest eigenvalue of $\tilde{H}(t)$.  Let
\be
	w_{\beta k}:=\sum_{\alpha} U_{\beta\alpha} a_{\alpha k}\,,\qquad \gamma_{ij}:= \sum_{\abs{\alpha-\beta}<N^\mathfrak{a}} \abs{w_{\alpha i}}^2 \abs{w_{\beta j}}^2 \,.
\ee
We abuse notation and denote the eigenvalues of $\tilde{H}(t)$ by $(\la_1,\cdots,\la_N)$.  Formally applying the It{\^o} lemma we see that
\be\label{eq: dla}
\begin{split}
		\dd \la_i = &\frac{1}{\sqrt{N}}\dd B_i -\frac{1}{\sqrt{N}}\sum_{\abs{\alpha-\beta}<N^\mathfrak{a}}w_{\alpha i}^* \dd B_{\alpha\beta} w_{\beta i}\\ &+\frac{1}{N}\sum_{j\neq i} \frac{(1-\ga_{ij})\dd t}{\la_i-\la_j}+ \langle a_i, (U^*\hat{Y}U-\hat{Y}) a_i \rangle \dd t
\end{split}
\ee
Here $\langle \cdot ,\cdot \rangle$ without subscript $t$ denotes the inner product between vectors in $\C^N$, as opposed to the previous notation $\langle \cdot ,\cdot \rangle_t$ standing for the quadratic covariation process.

Unlike the SDEs for $H$ and $\tilde{H}$, the equation \eqref{eq: dla} is problematic because the drift term 
\be
\frac{1}{N}\sum_{j\neq i} \frac{1-\ga_{ij}}{\la_i-\la_j}\dd t
\ee
is quite singular.  In the usual DBM $\gamma_{ij} = 0$ and the well-posedness for $\beta \geq 1$ can be proven by standard methods, see Proposition 4.3.5 in \cite{AGZ}.   In this case, the effect of $\gamma_{ij} >0$ means that effectively  (at least in terms of eigenvalue collision) we have $\beta <1$. It is therefore non-trivial to justify the well-posedness of the equation.  We remark here that since we are later able to prove that $\gamma_{ij} = o(1)$, the $\beta=2$ case is technically simpler; however the $\beta=1$ case requires the well-posedness.  In Section \ref{sec: wpness}, we will prove the well-posedness of \eqref{eq: dla} and show that the solution of \eqref{eq: dla} gives the eigenvalue process of $\tilde{H}(t)$.  In Section 
\ref{sec: estimates}, we prove some estimates which ensure that the second and fourth terms are negligible and that $\ga_{ij}\ll 1$ with high probability.  This allows us to view the SDE \eqref{eq: dla} as

\ben
	\dd \la_i = \frac{1}{\sqrt{N}}\dd B_i +\frac{1}{N}\sum_{j\neq i} \frac{(1-o(1))\dd t}{\la_i-\la_j}+ \text{error terms}\,,
\ee
which resembles the Dyson Brownian motion.   In Section \ref{sec: DBM} we  compare this process to DBM with initial data $\tilde{H}(0)$, which will in turn prove bulk universality.

\subsection{Well-definedness of coefficients}

Note that in the above derivation, the $k$-th eigenvector $a_k$ of $\tilde{H}(t)$ is not well-defined when the $k$-th eigenvalue of $\tilde{H}(t)$ has multiplicity greater than 1.  To solve this issue, we add a very small Gaussian perturbation, and redefine $X$ as 

\be\label{def: tdX}
	X:= \diag\{x_1,\cdots,x_N\}+ \e^{-N} Q\,,
\ee
where $Q=(Q_{ij})$ is a Hermitian matrix draw from the GUE ensemble, i.e., the probability density of $Q$ equals
\be\label{def: pQ}
	p_Q(q)= \frac{1}{Z_N} \e^{-\frac{N}{2}\sum_{i,j}\abs{q_{ij}}^2}\,.
\ee

We have the following proposition, which indicates that the eigenvalues of $X$ are almost surely distinct, and are well approximated by $x_1,\cdots,x_N$.  In this subsection, an $N$-dependent constant $c_N$ may change from line to line, but only changes for finite many times.

\begin{prop}\label{prop: tdX}
	Let $P$ be an $N\times N$ Hermitian matrix.  Let $\tilde{P}$ be given by
	\be
		\tilde{P} =P+\e^{-N}Q\,,
	\ee
	where $Q$ has the same distribution as in \eqref{def: pQ}, independent from $P$.
	Let ${\alpha_1}\leq\cdots\leq{\alpha_N}$ be the eigenvalues of $\tilde{P}$, and $\ga_1\leq \cdots\leq \ga_N$ be the eigenvalues of $P$.  
	
	Then, ${\alpha_1},\cdots,{\alpha_N}$ are almost surely distinct.  We have the following estimates,
	\be\label{ineq: rep}
		\begin{split}
		\E \sum_{i\neq j} \abs{\alpha_i-\alpha_j}^{-1}&<c_N\psi(N,P)\,,\\
		\quad \P [\min_{i\neq j} \abs{\alpha_i-\alpha_j} \leq \delta] &\leq c_N\psi(N,P)\delta^2\,, \forall \delta\in(0,1)\,.
		\end{split}
	\ee
	where, 
	\be
	\psi(N,P):= \exp \left(\e^{2N}\sum_{kl}\abs{P_{kl}}^2\right).
	\ee
	 Moreover, the following estimate holds:
	\be\label{tdX bd}
			\P[\max_{1\leq k\leq N}\abs{\alpha_k-\ga_k}\geq \e^{-N/2}]\leq \e^{-\e^{N/2}}\,.
	\ee 
\end{prop}
\begin{remark}
	In this proposition, the constant $\psi(N,P)$ is far from optimal and $c_N$ is not explicitly given. However, they are sufficient for the purpose of proving the well-posedness of \eqref{eq: dla}; the $N$-dependence is of no importance and one just needs some weak uniformity of the constants in $P$.  In Section \ref{sec: wpness} we will use Proposition \ref{prop: tdX} for fixed $N$. 
	
	Additionally, the above proposition holds for $P$ a symmetric matrix and $Q$ a GOE matrix.   In fact, in the GUE case we get $\delta^3$ for the second estimate of \eqref{ineq: rep} and $\delta^2$ for the GOE case.
\end{remark}
\begin{proof}
	It is sufficient to prove the proposition for any diagonal matrix $P$, because the law of $Q$ and the quantity $\psi(N,P)$ are invariant under conjugation by unitary matrices.
	
	Let $\mathcal{H}_N$ denotes the space of $N\times N$ Hermitian matrices. Note that  $\mathcal{H}_N$ can be parametrized by $(w_{ij})\in\R^{N\times N}$, such that $(w_{ij})$ represents the Hermitian matrix whose upper triangular  	
	part is  $(w_{ij}+\ii w_{ji}\ind{i\neq j})_{1\leq i\leq j\leq N}$.  Hence $\mathcal{H}_N$ naturally inherits the Lebesgue measure on $\R^{N\times N}$.  
	
	For brevity denote 
	\be
		\sigma_N:= e^{-N}\,.
	\ee
	Since $\tilde{P}$ has a Gaussian component, the probability measure of $\tilde{P}$ has a smooth density with respect to Lebesgue measure on $\mathcal{H}_N$, with the explicit formula
	\be
		p_{\tilde{P}}(w) = \frac{1}{Z_N} \exp\left({-\frac{1}{2\sigma_N^2} \sum_{i,j} \abs{w_{ij}-\delta_{ij}P_{ii}}^2}\right)\,.
	\ee
	Note that $Z_N$ does not depend on $P$.  We want to parametrize $\mathcal{H}_N$ by new coordinates $(\la,u)$ such that \ben
		\la=(\la_1,\cdots,\la_N)\in\R^N_\leq := \{(x_i)_{1\leq i\leq N}\in \R^N: x_1\leq \cdots \leq x_N\}
	\ee
	parametrizes the eigenvalues of any Hermitian matrix. For this purpose , we look at the spectral decomposition for any Hermitian matrix $M$:
	\be
		M= U^*\La U\,,
	\ee
	where  $\La$ is a real diagonal matrix and $U$ is a unitary matrix with non-negative real diagonal entries.  Note that $U$ is uniquely determined by its strict upper triangular part. Therefore, we can define $U(u)$ to be the unitary matrix determined by its strict upper triangular part $u\in\C^{N(N-1)/2}$.  Let $\Sigma\subset \C^{N(N-1)/2}$ be the compact domain where the map $u\mapsto U(u)$ is well defined. In this way, we have defined a map $T:\R^N_\leq \times \Sigma \to \mathcal{H}_N$ through
	\be
		T(\la,u)=U(u)^*\La(\la)U(u)\,,
	\ee
	where $\La(\la):=\diag(\la_1,\cdots,\la_N)$.  By (the proof of) Theorem 2.5.2 in [AGZ], the Jacobian $JT$ of the map $T$ satisfies
	\be
		JT = \prod_{i\neq j}(\la_i-\la_j)^2g(u)\,,
	\ee
	where $g(u)$ is a integrable function on $\Sigma$ . 
	
	Therefore, the probability density of $\tilde{P}$ in the new coordinates $(\la,u)$ equals
	\be
		p_{\tilde{P}}(\la,u)= \frac{1}{Z_N} \exp\left[{\frac{-1}{2\sigma_N^2}\left(\sum_k \la_k^2 -2 \sum_{k,l} \la_k \abs{U(u)_{kl}}^2 P_{ll} +\sum_k P_{kk}^2\right)}\right]\prod_{i\neq j}(\la_i-\la_j)^2g(u) \,.
	\ee
	By the simple observation $2\sum_{k,l} \la_k \abs{U(u)_{kl}}^2 P_{ll}\leq \frac{1}{2}\sum_k \la_k^2+2\sum_kP_{kk}^2$, we have
	\be
		\begin{split}
		p_{\tilde{P}}(\la,u)&\leq c_N \exp\left[{\frac{-1}{2\sigma_N^2}\left(\frac{1}{2}\sum_k \la_k^2-2\sum_kP_{kk}^2\right)}\right]\prod_{i\neq j}(\la_i-\la_j)^2g(u) \\
			& = c_N \exp \left(\frac{1}{\sigma_N^2}\sum_kP_{kk}^2\right) \exp\left(\frac{-1}{4\sigma_N^2}\sum_k \la_k^2\right)\prod_{i\neq j}(\la_i-\la_j)^2g(u)
		\end{split} \,.
	\ee
	Integrating over $u$, we get the following bound for the marginal density of eigenvalues,
	\be
		p(\la)\leq c_N \psi(N,P)\exp\left(\frac{-1}{4\sigma_N^2}\sum_k \la_k^2\right)\prod_{i\neq j}(\la_i-\la_j)^2\,.
	\ee
	Here $\psi(N,P):= \exp \left(\frac{1}{\sigma_N^2}\sum_kP_{kk}^2\right)$.	It immediately follows that the eigenvalues $(\alpha_1,\cdots,\alpha_N)$ of $\tilde{P}$ satisfy 
	\be
	\begin{split}
		\E \sum_{i\neq j} \abs{\alpha_i-\alpha_j}^{-1}&\leq c_N\psi(N,P) \int_{\R^N_\leq}\exp\left(\frac{-1}{5\sigma_N^2}\sum_k \la_k^2\right) \dd \la \leq c_N\psi(N,P)\,. \\
	\end{split}
	\ee
	For the second part of \eqref{ineq: rep}, we use another parametrization $\beta=(\beta_1,\cdots,\beta_N)$ given by
	\be
		\beta_k:=\la_k-\la_{k-1}\,,\text{ for } 1<k\leq N\,;\quad \beta_1:=\la_1\,.
	\ee
	Since $\la\mapsto\beta$ is a linear map, the Jacobian is a constant depending on $N$. The density in terms of $\beta$ satisfies
	\ben
		p(\beta)\leq c_N\psi(N,P) \exp\left[\frac{-1}{4\sigma_N^2}\sum_k \left(\sum_{1\leq l\leq k} \beta_l\right)^2\right]\prod_{i\neq j} \left(\sum_{i<l\leq j} \beta_l\right)^2\,.
	\ee
	Now we fix an $m\geq 2$ and look at the marginal density of $\beta_m$ when $\beta_m<1$.  We use the elementary inequality $(a+b)^2\geq a^2/2-b^2$ to get
	\ben	
		\left(\sum_{1\leq l\leq k} \beta_l\right)^2\geq \frac{1}{2}\left(\sum_{1\leq l\leq k, l\neq m} \beta_l\right)^2-\beta_m^2\geq  \frac{1}{2}\left(\sum_{1\leq l\leq k, l\neq m} \beta_l\right)^2-1\,.
	\ee
	Also,
	\ben
		\sum_{i<l\leq j} \beta_l\leq \sum_{i<l\leq j,l\neq m} \beta_l+1\,.
	\ee
	Therefore, 
	\be
		p(\beta)\leq c_N\psi(N,P) \exp\left[\frac{-1}{4\sigma_N^2}\sum_k \left(\sum_{\substack{1\leq l\leq k,\\l\neq m}} \beta_l\right)^2\right]\prod_{\substack{i\neq j, \\(i,j)\neq (m-1,m)}} \left(\sum_{\substack{1\leq l\leq k,\\l\neq m}} \beta_l+1\right)^2 \beta_m^2\,.
	\ee
	Integrating out all the variables except for $\beta_m$, we have
	\be
		p(\beta_m) \leq  c_N\psi(N,P) \beta_m^2\,.
	\ee
	Therefore, $\P[\la_m-\la_{m-1}\leq \delta] \leq c_N\psi(N,P) \delta^3$.  Summing over $1<m\leq N$ concludes the second part of \eqref{ineq: rep}.

	To prove the bound \eqref{tdX bd}, we denote the Frobenius norm of $Q$ by $\norm{Q}_F:=\sqrt{\sum_{i,j}\abs{Q_{ij}}^2}$ and $\norm{Q}$ the operator norm of $Q$. We have a trivial inequality $\norm{Q}\leq \norm{Q}_F$. By definition of $Q$ we have
	\be
		\E \e^{\norm{Q}^2/4}\leq \E \e^{\norm{Q}_F^2/4} = \int_{\R^{N\times N}} (2\pi)^{-N^2/2}\e^{-\tfrac{1}{4}\sum_{1\leq i,j\leq N}x_{ij}^2}\dd \vec{x} = 2^{N^2/2}\,.
	\ee
	By Chebyshev's ienquality we get a crude bound on $\norm{Q}$:
	\be
		\P[\norm{Q}\geq \e^{N/2}]\leq 2^{N^2/2} \e^{-\e^{N}/4}\leq \e^{-\e^{N/2}}\,,
	\ee
	when $N$ is large enough.  Therefore, $\norm{\tilde{P}-P}\leq \e^{-N/2}$ with probability $1-\e^{-\e^{N/2}}$.  On the event where $\norm{\tilde{P}-P}\leq \e^{-N/2}$, we have by Weyl's inequality, 
	\be
		\abs{\alpha_k-\ga_k}\leq \e^{-N/2}\,,\text{ for all }1\leq k\leq N\,.
	\ee 
\end{proof}

The proposition (in particular, the estimate \eqref{tdX bd}) implies that in order to prove Theorem \ref{thm:bu} where $X=\diag\{x_1,\cdots,x_N\}$, it is sufficient to prove the same result where $X$ is redefined in \eqref{def: tdX}. 

\begin{thm}\label{thm:bu1}
Let $X$ be defined in \eqref{def: tdX} and $Y=\{y_1,\cdots,y_N\}$ where $(x_k)$ and $(y_k)$ satisfy the assumptions in Section \ref{subsec: ass}.  Let $H=V^*XV+U^*YU$ where $V$ and $U$ are independently drawn from the Haar measure on the unitary group $\mathcal{U}_N$ (or the orthogoanl group $\mathcal{O}_N$).  Then, the conclusions of Theorem \ref{thm:bu} hold for $H$.

\end{thm}

\section{Some estimates}\label{sec: estimates}

In this section, we prove some estimates for the coefficients in equation \eqref{eq: dla} and \eqref{eq: dtH}, as well as the initial data $(\la_k(0))_{1\leq k\leq N}$ and $\tilde{H}(0)$.  

\subsection{Estimate of $U(t)-I$}\label{subsec: u-i}
In this subsection we prove that with high probability, $\sup_{0\leq t\leq T}\norm{U(t)-I}\ll 1$. For the reader's convenience, we recall that $0<\mathfrak{b}\leq\frac{\mathfrak{a}}{100}$ and $ \mathfrak{a} \leq \frac{1}{100}$ and  $T$ is defined by
\be\label{def: T 1}
	T=N^{-1+\mathfrak{b}}\,,
\ee
and that $U(t)$ is the unique strong solution to the SDE:
\be
	\dd U(t)=\ii \dd W (t) U(t) -\frac{1}{2}A U(t)\dd t\,;\quad U(0)=I\,.
\ee
The main theorem of this subsection is the following: 
\begin{thm}\label{thm: Ut-I} For  $\mathfrak{b}, \mathfrak{a}$ and $U$ as above, we have the estimate,
	\be	
	\P[\sup_{0\leq t\leq T} \norm{U(t)-I}>N^{-10\mathfrak{b}}]\leq \e^{-N^{10\mathfrak{b}}}\,.
	\ee
\end{thm}
Before proving the theorem, we introduce some notation.  We denote the martingale part of $U(t)$ by
\be	
M(t):=\ii \int_0^t \dd W(t) U(t)\,.
\ee
Therefore we can write $
	U(t)-I=M(t)+ \int_0^t\frac{1}{2}AU(s)\dd s$.
Hence,
\be\label{Ut-I}
	\sup_{0\leq t\leq T}\norm{U(t)-I}\leq \sup_{0\leq t\leq T}\norm{M(t)}+ \frac{T}{2}\norm{A}\,.
\ee
In view of the bound \eqref{A bound} on $A$ and the definition \eqref{def: T 1} of $T$, the second term above is $\mathcal{O}(N^{-\mathfrak{a}+\mathfrak{b}})$. In order to bound the operator norm of $M(t)$, we define 
\ben	
	K_t(\tau):= \exp (\tau M(t)^*M(t))\,,\text{ for } \tau\geq 0,0\leq t \leq T\,.
\ee
For simplicity of notations we omit the dependence of $K$ on $t$.  We shall estimate $\E\tfrac{1}{N} \tr K(\theta)$, which is the exponential moment of the empirical measure of $M(t)^*M(t)$, with parameter $\theta>0$ to be chosen. For matrices $A,B\in\C^{N\times N}$, we define the quadratic forms
\be\label{def: Q}
	\mathcal{Q}(A,B):=  \frac{1}{N}\sum_{i,j} A_{ii}B_{jj}\sigma_{ij}^2\,,\quad \hat {\mathcal{Q}}(A,B):=  \frac{1}{N}\sum_{i,j} A_{ij}B_{ij}\sigma_{ij}^2\,.
\ee
Using It{\^o}'s formula, we find
\be\label{eq: dexp}
\begin{split}
	\dd \frac{1}{N}\tr K(\theta)= & \frac{\theta }{N}\tr(K(\theta)U^* A U)\dd t  + \frac{\theta}{N} \int_0^\theta \left\{\mathcal{Q}(-MK(\tau)U^*, MK(\theta-\tau)U^*)\right.\\
	&+ \mathcal{Q}(UK(\tau)U^*, MK(\theta-\tau)M^*)+ \mathcal{Q}(MK(\tau)M^*, UK(\theta-\tau)U^*) \\
	&+ \left.\mathcal{Q}(-UK(\tau)M^*, UK(\theta-\tau)M^*)\right\}\dd \tau \dd t +\dd R\,.
\end{split}
\ee
Here $\dd R$ is a martingale term, whose quadratic variation process $\langle R\rangle_t$ satisfies
\be
	\dd \langle R\rangle_t= \dfrac{\theta^2}{N^2} \hat{\mathcal{Q}}(-MK(\theta)U^*, MK(\theta)U^*) \dd t\,.
\ee
We require a bound for the quadratic forms $\mathcal{Q}$ and $\hat{\mathcal{Q}}$. The $\ell^r$ norm of an $N$-dimensional vector is defined by,
\be
	\norm{v}_r:=\left(\frac{1}{N}\sum_i \abs{v_i}^r\right)^{\frac{1}{r}}\,, \text{ for } r\in[1,+\infty)\,,v\in\C^N\,.
\ee
For $r=\infty$, we denote $\norm{v}_\infty:= \max_i\abs{v_i}$.  For any matrix $Q$ we define an $N$-dimensional vector
\be
	{s}(Q):=(s_1(Q),\cdots,s_N(Q))\,,
\ee
where $s_1(Q)\geq \cdots\geq s_N(Q)$ are singular values of $Q$.

\begin{lem}
	Let $A$ and $B$ be $N \times N$ square matrices.  Let $s(A)=(s_1(A), \dots, s_N(A))$ and $s(B)=(s_1(B),\dots,s_N(B))$ be the singular values of $A$ and $B$ in decreasing order. Then, for $1\leq r,r' \leq \infty$ satisfying $r^{-1}+r'^{-1}=1$, we have
	
	\be
	\abs{\mathcal{Q}(A,B)}\vee \abs{\hat{\mathcal{Q}}(A,B)}\lesssim N^{2-\mathfrak{a}}\norm{s(A)}_r\norm{s(B)}_{r'}\,.
	\ee
	
\end{lem}

\begin{proof} Let the singular value decompositions of $A$ and $B$ be given by
	\be
		A_{ij}= \sum_k s_k(A)u_{ki}v_{kj}\,,\quad B_{ij}=\sum_k s_k(B)\hat{u}_{ki}\hat{v}_{kj}\,.
	\ee
	Therefore, 
	\be	
	\frac{1}{N}\sum_{i,j} A_{ii}B_{jj}\sigma^2_{ij} =\frac{1}{N} \sum_{i,j,k,l} s_k(A)u_{ki}v_{ki}\sigma^2_{ij} s_l(B)\hat u_{lj} \hat v_{lj}\,.
	\ee
	Denote $\hat \sigma_{kl}^2 :=\sum_{i,j}u_{ki}v_{ki}\sigma^2_{ij} \hat u_{lj} \hat v_{lj}$ and $\hat{\sigma}^2:=(\hat{\sigma}_{kl}^2)_{1\leq k,l\leq N}$. Then, by H{\"o}lder's inequality, 
	\be	
\absa{	\frac{1}{N}\sum_{i,j} A_{ii}B_{jj}\sigma^2_{ij}} =\absa{\frac{1}{N} \sum_{k,l} s_k(A)\hat{\sigma}_{kl}s_l(B)}\leq \norm{\hat{\sigma}^2 s(A)}_r\norm{s(B)}_{r'}\,.
	\ee
	It is sufficient to prove that $\norm{\hat \sigma^2}_{l^r\rightarrow l^r} \lesssim N^{2-\mathfrak{a}}$.  By Riesz-Thorin theorem, it is sufficient to prove for $r=1$ and $r=\infty$.  For $r=1$, we have
	\begin{equation}
	\norm{\hat \sigma^2}_{l^1\rightarrow l^1} = \sup_k \sum_l \absa{\hat{\sigma}_{kl}}^2\leq \sup_{i}\sum_j \sigma_{ij}^2 \lesssim N^{2-\mathfrak{a}}\,.
	\end{equation}
	Obviously this bound also holds for $(\hat{\sigma}^2)^\top$ in place of $\hat{\sigma}^2$, and we therefore get the $\norm{\hat \sigma^2}_{l^\infty\rightarrow l^\infty}$ bound by duality,
	\begin{equation}
	\norm{\hat \sigma^2}_{l^\infty\rightarrow l^\infty} = \norm{(\hat{\sigma^2})^\top}_{l^1\rightarrow l^1}\lesssim N^{2-\mathfrak{a}}\,.
	\end{equation}
	This concludes the proof of $\absa{	\frac{1}{N}\sum_{i,j} A_{ii}B_{jj}\sigma^2_{ij}} \lesssim N^{2-\mathfrak{a}}\norm{s(A)}_r\norm{s(B)}_{r'}$.   
	
	In order to prove a bound for $\abs{\hat{\mathcal{Q}}(A,B)}$, we assume the same spectral decompositions of $A$ and $B$, and write
	\be
		\abs{\hat{\mathcal{Q}}(A,B)}= \frac{1}{N} \sum_{i,j,k,l} s_k(A)u_{ki}v_{kj}\sigma^2_{ij} s_l(B)\hat u_{li} \hat v_{lj}\,.
	\ee
	The question reduces to estimating the $l^r\to l^r$ norm of the matrix $(\tilde{\sigma}^2_{kl})_{k,l}$ given by
	\be
		\tilde{\sigma}^2_{kl}:= \sum_{i,j} u_{ki}v_{kj}\sigma^2_{ij} \hat u_{li} \hat v_{lj}\,.
	\ee
	Again, by Riesz-Thorin Theorem, it is sufficient to prove the $l^1\to l^1$ norm and $l^\infty \to l^\infty$ norm. It is easy to see that
	\be
		\norm{\tilde{\sigma}^2}_{l^1\to l^1} \leq \sup_{l}\sum_{k} \tilde{\sigma}_{lk}^2\leq \norm{\sigma^2_{ij}}_{l^2\to l^2}\,,
	\ee
	as well as $\norm{\tilde{\sigma}^2}_{l^\infty\to l^\infty}\leq  \norm{\sigma^2_{ij}}_{l^2\to l^2}$.  Again, by Riesz-Thorin interpolation, we have
	\ben
		\norm{\sigma^2_{ij}}_{l^2\to l^2}\lesssim \sqrt{\norm{\sigma^2_{ij}}_{l^1\to l^1}\norm{\sigma^2_{ij}}_{l^\infty \to l^\infty}}\lesssim \sup_{i}\sum_j \sigma_{ij}^2 \lesssim N^{2-\mathfrak{a}}\,.
	\ee
	Hence we have
	\ben
		\abs{\hat{\mathcal{Q}}(A,B)}\leq \norm{\tilde{\sigma}^2}_{l^r\to l^r}\norm{s(A)}_r\norm{s(B)}_{r'}\lesssim N^{2-\mathfrak{a}}\norm{s(A)}_r\norm{s(B)}_{r'}\,.
	\ee
	This concludes the proof.
\end{proof}

The lemma enables us to estimate the right hand side of equation \eqref{eq: dexp}, which is the key to proving Theorem \ref{thm: Ut-I}.

\begin{proof}[Proof of Theorem \ref{thm: Ut-I}]
We start with estimating the right hand side of equation \eqref{eq: dexp}.  By the above lemma, the first term $\mathcal{Q}(-MK(s)U^*, MK(\theta-s)U^*)$ in the integrand satisfies
\be\label{ineq: QMK}
	\abs{\mathcal{Q}(-MK(\tau)U^*, MK(\theta-\tau)U^*)}\lesssim N^{2-\mathfrak{a}} \norm{s(MK(\tau))}_r \norm{s(MK(\theta -\tau))}_{r'}\,.
\ee
Here we have used the fact that a matrix's singular values are invariant under multiplication of a unitary matrix.  Assume that the singular decomposition of $M$ is 
\ben
	M= U_1SU_2\,.
\ee
A simple observation is that $MK(\tau)$ has singular decomposition
\ben
	MK(\tau) = U_1 S \e^{\tau S^2} U_2\,.
\ee
Note that $M(t)= U(t)-I-\int_0^t \frac{1}{2}AU(t)\dd t$ implies a crude bound $\norm{M}\leq 3$, and so $\norm{S}\leq 3$.  Therefore, the $k$-th singular value of $MK(\tau)$ satisfies
\ben
	s_k(MK(\tau)) \leq 3\e^{\tau s_k(M)^2}=3s_k(K(\tau))\,.
\ee
Going back to \eqref{ineq: QMK}, we see that
\be
	\abs{\mathcal{Q}(-MK(\tau)U^*, MK(\theta-\tau)U^*)}\lesssim N^{2-\mathfrak{a}} \norm{s(K(\tau))}_r \norm{s(K(\theta -\tau))}_{r'}\,.
\ee
Now we choose $r=\theta/\tau$ and $r'=\theta/(\theta-\tau)$, then $\norm{s(K(\tau))}_r= \norm{s( K(\theta))}_1^{1/r}$ and $\norm{s(K(\theta -\tau))}_{r'}=\norm{s(K(\theta))}_1^{1/r'}$.  The above inequality yields
\be
	\abs{\mathcal{Q}(-MK(\tau)U^*, MK(\theta-\tau)U^*)}\lesssim N^{2-\mathfrak{a}} \norm{s(K(\theta))}_1 = N^{1-\mathfrak{a}}\tr K(\theta) \,.
\ee
By similar arguments, we get the same bound for each term in the integrand on the right hand side of \eqref{eq: dexp}. Therefore, \eqref{eq: dexp} yields
\ben
	\frac{1}{N} \tr K(\theta)\lesssim \int_0^t\frac{\theta^2}{N}\left(\tr(K(\theta)U^*AU)+ N^{1-\mathfrak{a}}\tr K(\theta)\right)\dd s+\int_0^t\dd R\,.
\ee 
By \eqref{A bound} we have $\tr(K(\theta)U^*AU)\leq N^{1-\mathfrak{a}} \tr(K(\theta))$.  Therefore, the above inequality yields
\be\label{ineq: trK}
	\frac{1}{N} \tr K(\theta)\lesssim {\theta^2}N^{1-\mathfrak{a}} \int_0^t \frac{1}{N}\tr K(\theta)\dd s+\int_0^t \dd R\,.
\ee 
If we take the expectation of the above inequality, the martingale term vanishes and we derive,
\ben
\E \frac{1}{N} \tr K(\theta)\lesssim {\theta^2}N^{1-\mathfrak{a}} \int_0^t \E\frac{1}{N}\tr K(\theta)\dd s\,.
\ee 
Hence we obtain by Gronwall's inequality,
\ben
	\E \frac{1}{N} \tr K(\theta) \leq \e^{c\theta^2 N^{1-\mathfrak{a}}t}\,.
\ee
In order to obtain an estimate that holds for all time we return to the martingale term and bound its quadratic variation:
\ben
	\langle R\rangle_t \leq \int_0^t\theta^2 N^{-\mathfrak{a}} \tr K(2\theta) \dd s\,.
\ee
Denote $R^*(t)= \sup_{0\leq s\leq t} \abs{R(s)}$. By the BDG inequality we have, for any $\theta \leq N^{(\mathfrak{a}-\mathfrak{b})/2} $,
\ben
	\E R^*(t) \lesssim \E \sqrt {\int_0^t\theta^2 N^{-\mathfrak{a}} \tr K(2\theta) \dd s}\lesssim  \sqrt{\theta^2N^{1-\mathfrak{a}}\int_0^t\e^{c\theta^2 N^{1-\mathfrak{a}}s} \dd s} \leq \theta N^{-(\mathfrak{a}-\mathfrak{b})/2}\,,
\ee
which is bounded by a constant. Denote 
\ben
Q(\theta, t):= \sup_{0\leq s \leq t} \frac{1}{N} \tr K(\theta,s),.
\ee
Therefore, \eqref{ineq: trK} yields
\be\label{ineq: EQ}
	\E Q(\theta,t)\lesssim\theta^2N^{1-\mathfrak{a}}\int_0^t \E Q(\theta,s)\dd s +1 \,.
\ee
Applying Gronwall's inequality, we have
\ben
\E Q(\theta,T)\leq \e^{c\theta^2 N^{1-\mathfrak{a}} T}=\e^{c\theta^2 N^{\mathfrak{b}-\mathfrak{a}}}\,,
\ee
By Chebyshev's inequality, for any $t\in[0,T]$ and $\delta>0$
\ben
	\P[\sup_{0\leq t\leq T}\norm{M(t)}>\delta]\leq \P \left[ Q(\theta,T)> \frac{1}{N}\e^{\theta\delta^2}\right]\leq N\e^{c\theta^2N^{\mathfrak{b}-\mathfrak{a}}-\theta\delta^2}\,.
\ee
Take $\theta=N^{(\mathfrak{a}-\mathfrak{b})/2}$, $\delta=\theta^{-1/3}$. The above estimate gives
\ben
	\P[\sup_{0\leq t\leq T}\norm{M(t)}>N^{(\mathfrak{b}-\mathfrak{a})/6}]\leq \e^{c-\theta^{1/3}}\leq \e^{-N^{(\mathfrak{a}-\mathfrak{b})/7}}\,,\text{ for $N$ large enough.}
\ee
By the assumption of the theorem, $\mathfrak{b} \leq \mathfrak{a}/10$.  Hence, 
\ben
\P[\sup_{0\leq t\leq T}\norm{M(t)}>N^{-11\mathfrak{b}}]\leq \e^{-N^{10\mathfrak{b}}}\,,\text{ for $N$ large enough.}
\ee
This estimate together with \eqref{Ut-I} concludes the proof of the  estimate in the theorem.

\end{proof}

Note that for any $t_0\in[0,T]$, the process 
\be\label{def: hatU}
	\hat{U}(t):= U(t)U(t_0)^*
\ee
 satisfies the same SDE as $U(t)$ does,
\be
	\dd \hat{U}(t)=\ii \dd W(t) \hat{U}(t) -\frac{1}{2}A\hat{U}(t)\dd t\,.
\ee
Using the same argument, we can actually show a bound for $U(t)-U(t_0)$, for any $t\in[t_0,T]$.

\begin{thm}\label{thm: Us-Ut}
	For $N$ large enough we have the following. For any $0\leq t_0\leq t\leq T$, $\abs{t-t_0}\leq 1/N$, 
	\ben
	\P\left[\sup_{t_0\leq s\leq t} \norm{\hat{U}(s)-\hat{U}(t_0)}\geq ( {N(t-t_0)})^{1/4}\right]\leq \e^{-N^{\mathfrak{a}/3}}\,.
	\ee
	Also, for any $0\leq t_0\leq t\leq T$, $\abs{t-t_0}\leq \vartheta<1/N$, we have
	\ben
	\P\left[\sup_{t_0\leq s\leq t} \norm{\hat{U}(s)-\hat{U}(t_0)}\geq \vartheta^{9/20}\right]\leq \e^{-c_N\vartheta^{-1/10}}\,.
	\ee
\end{thm}

\begin{proof}
	Let $\hat{U}(t)$ be defined as in \eqref{def: hatU}. Denote
	\be
		\hat{M}(t):= \int_{t_0}^t \ii \dd W(t) \hat{U}(t)\,.
	\ee
	 Define $\hat{Q}(\theta,t):=\sup_{t_0\leq s\leq t} \frac{1}{N}\tr \exp\left( \theta \hat{M}^*\hat{M}\right)$. Since $(\hat{U}(t_0+s),\hat{M}(t_0+s))_{s\geq 0}$ has the same distribution as $(U(s),M(s))$ does, the argument in the proof of Theorem \ref{thm: Ut-I} holds for $(\hat{U},\hat{M},\hat{Q})$ in place of $(U,M,Q)$, up to \eqref{ineq: EQ}.  Therefore, according to Gronwall's inequality,
	 \ben
	 \E Q(\theta,t)\leq \e^{c\theta^2 N^{1-\mathfrak{a}} (t-t_0)}\,,
	 \ee
	 In other words,
	 \be\label{ineq: expmoment}
		\E  \exp (\theta \sup_{t_0\leq s\leq t}\norm{\hat{M}}^2) \leq N\e^{c\theta^2N^{1-\mathfrak{a}}(t-t_0)}\,.
	\ee
	Take $\theta=N^{\mathfrak{a}/2}(N(t-t_0))^{-1/2}$, we have
	\ben
		\P\left[\sup_{t_0\leq s\leq t} \norm{\hat{M}}\geq\frac{1}{2}( {N(t-t_0)})^{1/4}\right]\leq N\e^{c-N^{\mathfrak{a}/2}}\,.
	\ee
	On the other hand, note that $\hat{U}(s)-\hat{U}(t_0) = \hat{M}(s) - \int_{t_0}^{s} \frac{1}{2}A \hat{U}(t)\dd t$, and that $\norm{\int_{t_0}^{s} \frac{1}{2}A \hat{U}(t)\dd t}\lesssim \abs{s-t_0} N^{1-\mathfrak{a}}\log N$. Therefore, for $\abs{t-t_0}\leq N^{-2}$, we have
	\ben
		\P\left[\sup_{t_0\leq s\leq t} \norm{\hat{U}(s)-\hat{U}(t_0)}\geq( {N(t-t_0)})^{1/4}\right]\leq N\e^{c-N^{\mathfrak{a}/2}}\,.
	\ee
	Note that the number on the right hand side does not depend on $t_0$ or $t$. 	Therefore, for $N$ large enough, we have, for any $0\leq t_0\leq t\leq T$ such that $\abs{t-t_0}\leq 1/N$, 
	\ben
		\P\left[\sup_{t_0\leq s\leq t} \norm{\hat{U}(s)-\hat{U}(t_0)}\geq ( {N(t-t_0)})^{1/4}\right]\leq \e^{-N^{\mathfrak{a}/3}}\,.
	\ee
	This concludes the first estimate in the theorem.

	 To prove the second estimate, we use the inequality \eqref{ineq: expmoment} with Chebyshev inequality to see
	 \ben
	 \P[\sup_{t_0\leq s\leq t}\norm{\hat{M}} \geq \frac{1}{2}\vartheta^{9/20}] \leq N\e^{c_N\vartheta\theta^2- \theta\vartheta^{9/10}}\,.
	 \ee
	 Optimizing in $\theta$, we find 
	 	\ben
	 	 \P[\sup_{t_0\leq s\leq t}\norm{\hat{M}} \geq \frac{1}{2}\vartheta^{9/20}] \leq N\e^{-c_N\vartheta^{-1/5}}\,.
	 	\ee
	 On the other hand, note that $\hat{U}(s)-\hat{U}(t_0) = \hat{M}(s)- \int_{t_0}^{s} \frac{1}{2}A \hat{U}(t)\dd t$, and that $\norm{\int_{t_0}^{s} \frac{1}{2}A \hat{U}(t)\dd t}\lesssim N^{1-\mathfrak{a}}\log N\vartheta\leq \frac{1}{2}\vartheta^{9/20}$. Therefore
	 	 	\ben
	 	 	\P[\sup_{t_0\leq s\leq t}\norm{\hat{U}(s)-\hat{U}(t_0)} \geq \vartheta^{9/20}] \leq N\e^{-c_N\vartheta^{-1/5}}\,.
	 	 	\ee
	 	 The second estimate in the theorem follows, after absorbing $N$ in the exponential. 
\end{proof}

\subsection{Local law near the edges      }\label{subsec: local law edge}
In this subsection, we prove some estimates on the quantities $(w_{\alpha k})$ and $(\ga_{ij})$ that appear in the coefficients in \eqref{eq: dla}.  
Recall that 
\ben
w_{\beta k}=\sum_{\alpha} U_{\beta\alpha}a_{\alpha k}\,,
\ee
and that $(a_{\alpha k})_{1\leq \alpha \leq N}$ is the $k$-th eigenvector of 
\ben
\tilde{H}(t):=V^* XV+U(t)^* Y U(t)+(T-t)\hat{Y}\,.
\ee
Therefore, $(w_{\beta k})_{1\leq \beta\leq N}$ is the $k$-th eigenvector of
\ben 
\cH(t):=U(t)V^*XVU(t)^* + Y + (T-t) U(t)\hat{Y}U(t)^*\,.
\ee
Denote $\bar{V}(t):= VU(t)^*$.  Then we can write
\ben 
\cH(t)=\bar{V}(t)^*X\bar{V}(t) + Y + (T-t) U(t)\hat{Y}U(t)^*\,.
\ee
In order to get upper bounds for $\abs{w_{\alpha k}}$, we look at the Green's function defined by
\ben
\cG(z,t):= (\cH(t)-z)^{-1}\,,\forall z\in\C^+\,.
\ee
An important observation is that for any $t\geq 0$, the matrix  $\bar{V}(t)$ is Haar-distributed on the unitary group, and independent from $U(t)$. Recall that in the last subsection we proved $\norm{U(t)-I}\ll 1$, therefore, the last term of $\cH(t)$ is approximately $(T-t)\hat{Y}$.  We write
\be\label{eq: cH1}
\cH(t)=\bar{V}(t)X\bar{V}(t)+Y+(T-t)\hat{Y} + (T-t)(U(t)\hat{Y}U(t)^*-\hat{Y})\,.
\ee
In view of \eqref{hatY bound} and Theorem \ref{thm: Ut-I}, the last term above satisfies
\ben
\P\left[\sup_{0\leq t\leq T}\norm{(T-t)(U(t)\hat{Y}U(t)^*-\hat{Y})}\geq N^{-1-8\mathfrak{b}}\right]\leq \e^{-N^{10\mathfrak{b}}}\,.
\ee
Therefore, instead of $\mathcal{H}(t)$, we consider the matrix
\ben
\hat{\cH}(t):=\bar{V}(t)^* X\bar{V}(t)+ Y+ (T-t)\hat{Y}\,,
\ee
where $\bar{V}(t):= V U(t)^*$. We define as before,
\ben
	\cG(z,t):= (\cH(t)-z)^{-1}\,,\qquad \hat{\cG}(z,t):= (\hat{\cH}(t)-z)^{-1}\,.
\ee
Our strategy is to first prove an upper bound for $\hat{\cG}_{ii}(z,t)$, and then derive from this an upper bound for $\cG_{ii}(z,t)$, which will give upper bounds for $w_{\alpha k}$ and $\ga_{ij}$.  Since $\norm{X}\leq K+1$ except for on an event with probability $\e^{-\e^{N/2}}$, we will simply assume $\norm{X}\leq K+1$ in this subsection, and this will not effect the conclusions of any theorem in this subsection.

For brevity we denote
\be\label{def: barY}
	\bar{Y}(t):= Y+(T-t)\hat{Y}\,,
\ee
so that $\hat{\cH}(t)= \bar{V}(t)^* X \bar{V}(t)+\bar{Y}(t)$. For reasons that we will see later on, we also need to bound $\mathsf{G}(z,t):= (\mathsf{H}(t)-z)^{-1}$ where
\ben
	 \mathsf{H}(t):=X+\bar{V}(t)\bar Y(t)\bar{V}(t)^*\,.
\ee
In order to bound $\hat{\cG}_{ii}(z,t)$, we use the following concentration estimate on the unitary group, which is a consequence of the Gromov-Milman theorem.  We consider the metric on the unitary group $\mathcal{U}_N$ that is induced by the Frobenius norm $|| \cdot ||_F$ on $\C^{N\times N}$.  We will consider Lipschitz functions on $\mathcal{U}_N$ where the Lipschitz constant is defined with respect to $|| \cdot ||_F$.

\begin{prop}\label{prop: Gromov}
	Let $g:\mathcal{U}_N\to \C$ be a Lipschitz function with Lipschitz constant $L$. Let $\P$ denote the (normalized) Haar measure on $\mathcal{U}_N$ and $\E$ denote the expectation with respect to the Haar measure.   Assume $\E[g] =0$.  Then, there exist $c>0$, $c_1>0$ that do not depend on $N$ such that
	\ben
		\E \exp(tg)\leq \exp(ct^2L^2/N)\,,\forall t\in\R\,.
	\ee
	As a consequence,
	\ben
		\P[g\geq \delta]\leq \exp(-c_1 N\delta^2/L^2)\,,\forall \delta>0\,.
	\ee
\end{prop}

\begin{cor}\label{cor: Gromov}For any $1\leq i\leq N$, we have the following concentration inequality for $\hat{\cG}_{ii}(z,t)$:
	\ben
	\P\left[\absa{\hat{\cG}_{ii}(z,t)-\E\hat{\cG}_{ii}(z,t)}\geq \frac{N^\nu}{\sqrt{N}\eta^2}\right]\leq \exp(-cN^{\nu})\,.
	\ee
The same estimate holds for $(\hat{\cG}(z,t)\bar{V}^*(t)X\bar{V}(t))_{ii}$, $\mathsf{G}_{ii}(z,t)$ and $(\mathsf{G}(z,t)\bar{V}(t)Y(t)\bar{V}^*(t))_{ii}$.
\end{cor}
\begin{proof}
	Note that the map $\Om\mapsto (\Om^* X \Om-z)^{-1}$ is the composite of two maps: $\Om\mapsto \Om^* X \Om$ and $Q\mapsto(Q-z)^{-1}$.  The former one is Lipschitz with constant $2(K+1)$ by the simple observation
	\ben	
		\norm{\Om_1^*X\Om_1-\Om_2^* X\Om_2}_F\leq 2\norm{X}\norm{\Om_1-\Om_2}_F\,,\forall \Om_1,\Om_2\in\mathcal{U}_N\,.
	\ee
	The latter map's Lipschitz constant can be seen by
	\ben
	\begin{split}
		\norm{(Q_1-z)^{-1}-(Q_2-z)^{-1}}_F&\leq \norm{(Q_1-z)^{-1}}\norm{Q_1-Q_2}_F\norm{(Q_2-z)^{-1}}\\
		&\leq \eta^{-2}\norm{Q_1-Q_2}_F\,.
	\end{split}
	\ee
	for any Hermitian matrices $Q_1$ and $Q_2$.   Therefore, the Lipschitz constant for the map $\Om\mapsto (\Om^* X \Om-z)^{-1}$ is 
	\ben
		L\leq 2(K+1)\eta^{-2}\,.
	\ee
	Taking the $(i,i)$-th component, we see that the map 
	\ben
		\Om\mapsto (\Om^* X \Om-z)^{-1}_{ii}
	\ee
	also has Lipschitz constant $2(K+1)\eta^{-2}$. By Proposition \ref{prop: Gromov} we have
	\ben
		\P\left[\absa{\hat{\cG}_{ii}(z,t)-\E\hat{\cG}_{ii}(z,t)}\geq \frac{N^\nu}{\sqrt{N}\eta^2}\right]\leq \exp(-cN^{\nu})\,.
	\ee
	By similar arguments it is easy to get the same bounds for $(\hat{\cG}(z,t)\bar{V}^*(t)X\bar{V}(t))_{ii}$, $\mathsf{G}_{ii}(z,t)$  
	and $(\mathsf{G}(z,t)\bar{V}(t) \allowbreak Y(t)\bar{V}^*(t))_{ii}$.
\end{proof}

The estimates in Corollary \ref{cor: Gromov} are useful when $\eta\geq N^{-1/4+\nu/2}$. Hence, on a domain where $\eta\geq N^{-1/4+\nu/2}$, the quantities $\hat{\mathcal{G}}_{ii}$ are concentrated around deterministic functions of $z$ and $t$. To figure out what the deterministic functions are, we look at the following system of equations.  Here we denote $X=:\diag(X_1,\cdots,X_N)$, $\bar Y(t)=:\diag(\bar Y_1(t),\cdots,\bar Y_N(t))$.  We omit the dependence of $Y_k(t)$ on $t$ for brevity.
\be\label{eq: sce1}
\left\{
	\begin{split}
		(\hat{\cG}\bar{V}^*X\bar{V})_{ii}+ \hat{\cG}_{ii}\bar{Y}_i-z\hat{\cG}_{ii}=1\\
		\mathsf{G}_{ii}X_{i}+ (\mathsf{G}\bar{V}\bar{Y}\bar{V}^*)_{ii}-z\mathsf{G}_{ii}=1\\
	\tr (\hat{\cG}\bar{V}^*X\bar{V}) +\tr(\hat{\cG}\bar{Y}) -z\tr \hat{\cG}=N
	\end{split}
\right.
\ee
Note that $\mathsf{G}$ is equal to $\hat{\cG}$ up to conjugation by the unitary matrix $\bar{V}(t)$, and hence $\tr (\hat{\cG}\bar{V}^*X\bar{V})=\tr(\mathsf{G}X)$. Rearranging and averaging over $i$, we have
\be\label{eq: sce2}
\left\{
\begin{split}
	\frac{1}{N}\tr \hat{\cG}& = \frac{1}{N}\sum_i\left(-z+\bar{Y}_i +\frac{(\hat{\cG}\bar{V}^*X\bar{V})_{ii}}{\hat{\cG}_{ii}}\right)^{-1}\\
	 \frac{1}{N}\tr {\mathsf{G}}&=  \frac{1}{N}\sum_i\left(-z+X_i+\frac{(\mathsf{G}\bar{V}\bar{Y}\bar{V}^*)_{ii}}{\mathsf{G}_{ii}}\right)^{-1}\\
	\frac{N}{\tr \hat{\cG}}&=\frac{\tr (\mathsf{G}X)}{\tr \hat{\cG}} +\frac{\tr(\hat{\cG}\bar{Y})}{\tr \hat{\cG}} -z
\end{split}
\right.
\ee
Denote $\mathsf{m}:=  \frac{1}{N}\tr \hat{\cG}$, $\mathsf{w}_1:= \frac{\tr (\mathsf{G}X)}{\tr \hat{\cG}}$ and $\mathsf{w}_2:= \frac{\tr(\hat{\cG}\bar{Y})}{\tr \hat{\cG}}$. We want to replace $\tfrac{(\hat{\cG}\bar{V}^*X\bar{V})_{ii}}{\hat{\cG}_{ii}}$ and $\tfrac{(\mathsf{G}\bar{V}\bar{Y}\bar{V}^*)_{ii}}{\mathsf{G}_{ii}}$ by $\mathsf{w}_1$ and $\mathsf{w}_2$, respectively. We need to control
\be\label{eq: w1diff}
	\frac{(\hat{\cG}\bar{V}^*X\bar{V})_{ii}}{\hat{\cG}_{ii}}-\mathsf{w}_1= \frac{\sum_j ((\hat{\cG}\bar{V}^*X\bar{V})_{ii}\hat{\cG}_{jj}- \hat{\cG}_{ii}(\bar{V}^*X\bar{V}\hat{\cG})_{jj}) }{\hat{\cG}_{ii}\tr \hat{\cG}}\,.
\ee
The numerator on the right hand side has $0$ expectation, because by Proposition 3.2 in \cite{PV00}, for any $i$ and $j$,
\ben
	\E[(\hat{\cG}\bar{V}^*X\bar{V})_{ii}\hat{\cG}_{jj}- \hat{\cG}_{ii}(\bar{V}^*X\bar{V}\hat{\cG})_{jj}]=0\,.
\ee
By the simple observation that for any random variables $\xi_1$ and $\xi_2$,
\ben
\begin{split}
	\absa{\xi_1\xi_2-\E[\xi_1\xi_2] }&= \absa{\xi_1(\xi_2-\E\xi_2)+(\xi_1-\E\xi_1)\E\xi_2 -\text{cov}(\xi_1,\xi_2) }\\
	&\leq \norm{\xi_1}_\infty\vee\norm{\xi_2}_\infty (\abs{\xi_1-\E\xi_1}+\abs{\xi_2-\E\xi_2})+\sqrt{\var\xi_1\var\xi_2}\,,
\end{split}
\ee
and in view of Proposition \ref{cor: Gromov}, we have, with probability $1-\exp(-cN^\nu)$,
\be\label{eq: pastur1}
	(\hat{\cG}\bar{V}^*X\bar{V})_{ii}\hat{\cG}_{jj}- \hat{\cG}_{ii}(\bar{V}^*X\bar{V}\hat{\cG})_{jj}= \mathcal{O}\left(\frac{N^\nu}{\sqrt{N}\eta^3}\right)\, ,
\ee
for $\eta \geq N^{-1/4}$.  
We want to divide both sides by $\hat{\cG}_{ii}\tr \hat{\cG}$, and therefore require a lower bound on $\abs{\hat{\cG}_{ii}\tr\hat{\cG}}$.  It is sufficient to get a lower bound on $\im \hat{\cG}_{ii}$.

\begin{prop}\label{prop: imcG}
	For $\abs{z}\leq 4K$, there is a universal constant $c>0$ such that 
	\ben
		\im\hat{\cG}_{ii}\geq c\eta\,.
	\ee	
\end{prop}
\begin{proof}
	Assume that $\hat{\cH}$ has spectral decomposition $\hat{\cH}_{ij}=\om_{ik}^* \ga_k\om_{jk} $. Then,
	\be
		\im\hat{\cG}_{ii}=  \sum_k \frac{\abs{\om_{ik}}^2\eta}{\abs{\ga_k-z}^2}\geq \absa{\norm{\hat{\cH}}+\abs{z}}^{-2}\eta \sum_k \abs{\om_{ik}}^2\geq c\eta\,.
	\ee
\end{proof}
Using this proposition, we divide both sides of \eqref{eq: pastur1} by $\hat{\cG}_{ii}\tr \hat{\cG}$. Recalling \eqref{eq: w1diff}, we have, with probability $1-\exp(-cN^\nu)$,
\be\label{eq: w1diff1}
	\absa{\frac{(\hat{\cG}\bar{V}^*X\bar{V})_{ii}}{\hat{\cG}_{ii}}-\mathsf{w}_1} = \mathcal{O}\left(\frac{N^\nu}{\sqrt{N}\eta^5}\right)\,.
\ee
A similar bound holds for $\mathsf{w}_2$. This enables us to rewrite \eqref{eq: sce2} as
\be\label{eq: sce3}
\left\{
\begin{split}
	 \mathsf{m}& = \frac{1}{N}\sum_i\left(-z+\bar{Y}_i +\mathsf{w}_1\right)^{-1}+\mathcal{R}\\
	 \mathsf{m}&=\frac{1}{N}\sum_i\left(-z+X_i+\mathsf{w}_2\right)^{-1}+\mathsf{R}\\
	\dfrac{1}{\mathsf{m}}&=\mathsf{w}_1+\mathsf{w}_2 -z
\end{split}
\right.
\ee
where $\abs{\mathcal{R}}+\abs{\mathsf{R}} = \mathcal{O}\left(\frac{N^\nu}{\sqrt{N}\eta^7}\right)$ with probability $1-2N\exp(-cN^\nu)$.  Now let $\mathfrak{m}_1$ and $\mathfrak{m}_2$ be the Stieltjes transform of $\frac{1}{N}\sum_i \delta_{X_i}$ and $\frac{1}{N}\sum_i \delta_{\bar{Y}_i}$ respectively. Consider a system of equations
\be\label{eq: sce4}
\left\{
\begin{split}
	\mathfrak{m}& = \mathfrak{m}_2(z-\mathfrak{w}_1)\\
	\mathfrak{m}&=\mathfrak{m}_1(z-\mathfrak{w}_2)\\
	\dfrac{1}{\mathfrak{m}}&=\mathfrak{w}_1+\mathfrak{w}_2 -z
\end{split}
\right.
\ee
By Theorem 4.1 in \cite{BB07}, for any $z\in\C^+$, the system above has a unique solution $(\mathfrak{w}_1,\mathfrak{w_2},\mathfrak{m})\in\C^-\times\C^-\times\C^+$ that depends holomorphically on $z$.

\begin{thm}\label{thm: edge law}
For any $\nu>0$, $t\in[0,T]$, $\eta=\im z\in [N^{-1/40},1]$, $\abs{z}\leq 4K$, the following holds with probability $1-\exp(-cN^\nu)$, when $N$ is large enough.
\ben
	\absa{\hat{\cG}_{ii} - \frac{1}{-z+\bar{Y}_i +\mathfrak{w}_1}} = \mathcal{O}\left(\frac{N^\nu}{\sqrt{N}\eta^{15}}\right)\,.
\ee
\end{thm}

\begin{proof}
Rewrite \eqref{eq: sce4} as
\be\label{eq: sce5}
\left\{
\begin{split}
	-\mathfrak{w}_2& = -(z-\mathfrak{w}_1)-\frac{1}{\mathfrak{m}_2(z-\mathfrak{w}_1)}\\
	-\mathfrak{w}_1& = -(z-\mathfrak{w}_2)-\frac{1}{\mathfrak{m}_1(z-\mathfrak{w}_2)}
\end{split}
\right.
\ee
By Proposition 2.2 in \cite{Ma92}, the map $\zeta\mapsto -\zeta-\frac{1}{\mathfrak{m}_2(\zeta)}$ is the Stieltjes transform of a Borel measure on $\R$ with total mass $\frac{1}{N}\sum_i \bar{Y}_i^2$.  Similarly, the map $\zeta\mapsto -\zeta-\frac{1}{\mathfrak{m}_1(\zeta)}$ is the Stieltjes transform of a Borel measure on $\R$ with total mass $\frac{1}{N}\sum_i X_i^2$.  Denote the Borel measures by $\hat{\mu}_2$ and $\hat{\mu}_1$, respectively.  Denote $\hat{\mathfrak{m}}_2$ and $\hat{\mathfrak{m}}_1$ to be their Stieltjes transform, respectively.  Then, \eqref{eq: sce5} reads
\be\label{eq: sce6}
\left\{
\begin{split}
	-\mathfrak{w}_2& = \hat{\mathfrak{m}}_2(z-\mathfrak{w}_1)\\
	-\mathfrak{w}_1& = \hat {\mathfrak{m}}_1(z-\mathfrak{w}_2)
\end{split}
\right.
\ee
Similarly,  the system \eqref{eq: sce3} can be written as
\be\label{eq: sce7}
\left\{
\begin{split}
	-\mathsf{w}_2& = \hat{\mathfrak{m}}_2(z-\mathsf{w}_1)+r_1\\
	-\mathsf{w}_1& = \hat {\mathfrak{m}}_1(z-\mathsf{w}_2)+r_2
\end{split}
\right.
\ee
where $\abs{r_1}\vee\abs{r_2}= \mathcal{O}\left(\frac{N^\nu}{\sqrt{N} \eta^9}\right)$  with probability $1-2N\exp(-cN^\nu)$. Taking the imaginary part of the first equation in \eqref{eq: sce6} we get,
\ben
	-\im\mathfrak{w}_2 = \int_{\R}\frac{(\eta-\im\mathfrak{w}_1)\hat{\mu}_2(\dd x)}{\absa{x-z+\mathfrak{w}_1}^2},
\ee
which implies
\ben	
	 \int_{\R}\frac{\hat{\mu}_2(\dd x)}{\absa{x-z+\mathfrak{w}_1}^2}= \frac{-\im\mathfrak{w}_2}{\eta-\im\mathfrak{w}_1}\,.
\ee
Similarly, $\int_{\R}\frac{\hat{\mu}_1(\dd x)}{\absa{x-z+\mathfrak{w}_2}^2}=\frac{-\im\mathfrak{w}_1}{\eta-\im\mathfrak{w}_2}$. Similary we also derive from \eqref{eq: sce7} that
\ben
	 \int_{\R}\frac{\hat{\mu}_2(\dd x)}{\absa{x-z+\mathsf{w}_1}^2}=\frac{-\im\mathsf{w}_2-\im r_1 }{\eta-\im\mathsf{w}_1}\,,
\ee
as well as $\int_{\R}\frac{\hat{\mu}_1(\dd x)}{\absa{x-z+\mathsf{w}_2}^2}=\frac{-\im\mathsf{w}_1-\im r_2}{\eta-\im\mathsf{w}_2}$.

Subtracting \eqref{eq: sce7} from \eqref{eq: sce6} yields,
\be\label{eq: sce diff}
\left\{
\begin{split}
	\mathfrak{w}_2-\mathsf{w}_2& = \int_\R \frac{(\mathfrak{w}_1-\mathsf{w}_1)\hat{\mu}_2(\dd x)}{(x-z+\mathfrak{w}_1)(x-z+\mathsf{w}_1)}+r_1\\
	\mathfrak{w}_1-\mathsf{w}_1& = \int_\R \frac{(\mathfrak{w}_2-\mathsf{w}_2)\hat{\mu}_1(\dd x)}{(x-z+\mathfrak{w}_2)(x-z+\mathsf{w}_2)}+r_2
\end{split}
\right.
\ee
The equations are affine equations in $\mathfrak{w}_2-\mathsf{w}_2, \mathfrak{w}_1-\mathsf{w}_1$.  Therefore, we denote 
\ben
	\mathsf{a}:=  \int_\R \frac{\hat{\mu}_2(\dd x)}{(x-z+\mathfrak{w}_1)(x-z+\mathsf{w}_1)}\,,\quad \mathsf{b}:=\int_\R \frac{\hat{\mu}_1(\dd x)}{(x-z+\mathfrak{w}_2)(x-z+\mathsf{w}_2)}\,.
\ee
The identities obtained right before \eqref{eq: sce diff} with Cauchy-Shwartz inequality imply that 
\ben
\absa{	\mathsf{a}\mathsf{b}}\leq \left(\frac{-\im\mathfrak{w}_2}{\eta-\im\mathfrak{w}_1}\frac{-\im\mathfrak{w}_1}{\eta-\im\mathfrak{w}_2}\frac{-\im\mathsf{w}_2-\im r_1 }{\eta-\im\mathsf{w}_1}\frac{-\im\mathsf{w}_1-\im r_2 }{\eta-\im\mathsf{w}_2}\right)^{1/2}\,.
\ee
An elementary but tedious calculation yields $\absa{	\mathsf{a}\mathsf{b}}\leq 1-c\eta^2$ with overwhelming probability.  Also, using the fact that $\im \mathfrak{w}_{1,2}\leq 0$, $\im\mathsf{w}_{1,2}\leq 0$, we have the bound
\ben
	\abs{\mathsf{a}}\vee \abs{\mathsf{b}}\leq c\eta^{-2}\,,\text{ with probability $1-2N\exp(-cN^\nu)$.}
\ee
Equations \eqref{eq: sce diff} can be written as 
\ben
	\begin{bmatrix}
		-\mathsf{a} & 1 \\
		1 & -\mathsf{b}
	\end{bmatrix} 
	\begin{bmatrix}
	\mathfrak{w}_2-\mathsf{w}_2 \\ \mathfrak{w}_1-\mathsf{w}_1
\end{bmatrix} = \begin{bmatrix}
r_1\\r_2
\end{bmatrix}\,.
\ee
According to the inversion formula for $2\times 2$ matrices,
\ben
\begin{bmatrix}
	\mathfrak{w}_2-\mathsf{w}_2 \\ \mathfrak{w}_1-\mathsf{w}_1
\end{bmatrix} = \frac{1}{\mathsf{a}\mathsf{b}-1} \begin{bmatrix}
-\mathsf{b} & -1 \\
-1 & -\mathsf{a}
\end{bmatrix}  \begin{bmatrix}
r_1\\r_2
\end{bmatrix}\,.
\ee
Hence,
\ben
	\absa{\mathfrak{w}_2-\mathsf{w}_2}\vee \absa{\mathfrak{w}_1-\mathsf{w}_1} \leq \frac{\abs{\mathsf{b}}+ \abs{\mathsf{a}}+ 1}{1-\absa{\mathsf{a}\mathsf{b}}} \absa{r_1}\vee \absa{r_2}\leq c\eta^{-4} \absa{r_1}\vee \absa{r_2}\,,
\ee
with probability $1-2N\exp(-cN^\nu)$.  Recalling  \eqref{eq: w1diff1} we obtain,
\ben
	\frac{(\hat{\cG}\bar{V}^*X\bar{V})_{ii}}{\hat{\cG}_{ii}} =\mathsf{w}_1+ \mathcal{O}\left(\frac{N^\nu}{\sqrt{N}\eta^5}\right) = \mathfrak{w}_1+ \mathcal{O}\left(\frac{N^\nu}{\sqrt{N}\eta^{13}}\right)\,,
\ee
with probability $1-2N\exp(-cN^\nu)$.  Plugging this back to the first identity in \eqref{eq: sce1} and recalling Proposition \ref{prop: imcG}, we have, with probability $1-2N\exp(-cN^\nu)$,
\ben
	\absa{\hat{\cG}_{ii} - \frac{1}{-z+\bar{Y}_i +\mathfrak{w}_1}} = \mathcal{O}\left(\frac{N^\nu}{\sqrt{N}\eta^{15}}\right)\,.
\ee
The conclusion of the theorem follows.  
\end{proof}
We define the domain
\begin{align}
\Sigma = \{z=E+\ii \eta: \eta \in[N^{-1/40},1], E\in[-2K,2K] \}.
\end{align}
On this domain we have the following estimate.
\begin{cor} We have with probability at least $1 - \e^{ - c N^{c} }$ for some $c>0$,
\be\label{ineq: cGii}
	\sup_{(z,t)\in \Sigma\times [0,T]}	\absa{{\cG}_{ii} - \frac{1}{-z+\bar{Y}_i +\mathfrak{w}_1}} = \mathcal{O}\left(\frac{N^\nu}{\sqrt{N}\eta^{15}}\right)\,.
\ee	
	
\end{cor}

\begin{proof}
Take a discrete subset $S\subset \Sigma_\nu \times [0,T]$, whose $N^{-\log N}$-neighborhood covers $\Sigma\times [0,T]$. By Theorem \ref{thm: edge law}, we have, with probability $1-\exp(-cN^\nu)$ (with a new universal constant $c>0$),

\ben
	\sup_{(z,t)\in S}	\absa{\hat{\cG}_{ii} - \frac{1}{-z+\bar{Y}_i +\mathfrak{w}_1}} = \mathcal{O}\left(\frac{N^\nu}{\sqrt{N}\eta^{15}}\right)\,.
\ee
By Theorem \ref{thm: Us-Ut}, the quantity $\hat{\cG}(z,t)$ is $C^{1/4}$ in both $t$ and $z$, with a H{\"o}lder constant at most $N^2$.  Therefore, the above supremum can be taken over the entirety of $\Sigma\times [0,T]$, with a small error of $\mathcal{O}(N^{2-\frac{\log N}{4}})$ on the right hand side.  Therefore, we have, with probability $1-\exp(-cN^\nu)$,
\be\label{ineq: hatGii}
	\sup_{(z,t)\in \Sigma\times [0,T]}	\absa{\hat{\cG}_{ii} - \frac{1}{-z+\bar{Y}_i +\mathfrak{w}_1}} = \mathcal{O}\left(\frac{N^\nu}{\sqrt{N}\eta^{15}}\right)\,.
\ee	
Next, we use the resolvent identity $Q_1^{-1}-Q_2^{-1}=Q_1^{-1}(Q_2-Q_1)Q_2^{-1}$ to see
	\ben
	\begin{split}
		\abs{\hat{\cG}_{ii}-\cG_{ii}} &=\abs{\sum_{k,l} \hat{\cG}_{ik}(\cH-\hat{\cH})_{kl}\cG_{li}}\\
		&\leq \norm{\cH-\hat{\cH}} \sqrt{\sum_k \abs{\hat{\cG}_{ik}}^2}\sqrt{\sum_k \abs{\cG_{li}}^2}\,.
	\end{split}
	\ee
	By the Ward Identity, $\sqrt{\sum_k \abs{\hat{\cG}_{ik}}^2}=\sqrt{\im \hat{\cG}_{ii}/\eta}$, $\sqrt{\sum_k \abs{{\cG}_{ik}}^2}=\sqrt{\im {\cG}_{ii}/\eta}$.  Also note that $\eta\geq N^{-1+\nu}$ and that $\norm{\cH-\hat{\cH}}\leq N^{-1-8\mathfrak{b}}$ with probability $1-\e^{-N^{10\mathfrak{b}}}$.  Therefore, with probability $1-\e^{-N^{10\mathfrak{b}}}$, for any $ (z,t)\in\Sigma\times[0,T]$,
	\be
	\begin{split}
		\abs{\hat{\cG}_{ii}-\cG_{ii}} &\leq N^{-1-8\mathfrak{b}}\eta^{-1}\sqrt{\im \hat{\cG}_{ii}\im {\cG}_{ii}}\\
		&\leq N^{-1-8\mathfrak{b}}\eta^{-1} \sqrt{\abs{\hat{\cG}_{ii}} (\abs{\hat{\cG}_{ii}} +\abs{\hat{\cG}_{ii}-\cG_{ii}} )}\,.
	\end{split}
	\ee
	A simple calculation yields, for any $1\leq i\leq N$,
	\be\label{eq: hatG-G1}
	\abs{\hat{\cG}_{ii}-\cG_{ii}} \leq 2N^{-1-8\mathfrak{b}}\eta^{-1} \abs{\hat{\cG}_{ii}}\leq \frac{2}{ N \eta^2} \leq \frac{ N^{\nu}}{ \sqrt{N} \eta^{15}}.
	\ee
	Combining \eqref{ineq: hatGii} and \eqref{eq: hatG-G1} we obtain,
\be
\sup_{(z,t)\in \Sigma\times [0,T]}	\absa{{\cG}_{ii} - \frac{1}{-z+\bar{Y}_i +\mathfrak{w}_1}} = \mathcal{O}\left(\frac{N^\nu}{\sqrt{N}\eta^{15}}\right)
\ee	
with probability $1-\e^{-N^{10\mathfrak{b}}}$.
\end{proof}

\subsection{Deterministic estimates}

In the previous subsection we defined a diagonal matrix $\bar{Y}$ (see \eqref{def: barY}) through
\ben
	\bar{Y}(t):= Y+(T-t)\bar{Y}\,.
\ee
Recalling the assumption \eqref{reg y}, upper bound \eqref{hatY bound} and the definition \eqref{def: T 1} of $T$, we have
\be\label{reg barY}
	\sup_{0\leq k\leq N} \absa{\bar{Y}_k-y_k^\star} \lesssim N^{-1+c}+ N^{-1+\mathfrak{b}}\log N = o\left( N^{-\frac{9}{10}}\right)\,.
\ee
In this subsection we prove some deterministic estimates on the free convolution of $\frac{1}{N}\sum_i \delta_{X_i}$ and $\frac{1}{N}\sum_i \delta_{\bar Y_i}$. The estimates are needed in the next sections.

Recall that $\mu_1$ and $\mu_2$ are limits of empirical measures $\frac{1}{N}\sum_i \delta_{X_i}$ and $\frac{1}{N}\sum_i \delta_{\bar Y_i}$ respectively. Denote by $m_1$ and $m_2$ the stieltjes transforms of $\mu_1$ and $\mu_2$, respectively, i.e.,
\be
m_1(z):= \int_\R \frac{\mu_1(\dd x)}{x-z},\quad m_2(z):=\int_{\R} \frac{\mu_2(\dd y)}{y-z}\,,\quad \text{ for all } z\in\C^+\,.
\ee
Recall that we defined $\mathfrak{m}_1$ and $\mathfrak{m}_2$ to be the Stieltjes transforms of $\frac{1}{N} \sum_i \delta_{X_i}$ and $\frac{1}{N}\sum_i\delta_{\bar{Y}_i}$, respectively. We first prove estimates on differences $m_1-\mathfrak{m}_1$ and $m_2-\mathfrak{m}_2$, which will be used later on.  Throughout this section, the generic constant $p$ may increase from line to line, but will change only for finitely many times, hence will remain finite.

\begin{prop}\label{prop: m1-fm1}
	Denote $z= E+\ii \eta$.  There is a $p>0$ such that if $\eta\geq N^{-\frac{1}{p}}$, then
	\be
	\begin{split}
		\absa{m_1(z)-\mathfrak{m}_1(z)}&\leq N^{-\frac{1}{p}}\,,\\
		\absa{m_2(z)-\mathfrak{m}_2(z)}&\leq N^{-\frac{1}{p}}\,.
	\end{split}
	\ee
\end{prop}
\begin{proof}  The first inequality is by the assumption \eqref{reg x}.  Denote $y_0^\star := \inf \supp \mu_2$. 
	Recall that we defined 
	 $(y_k^\star)_{1\leq k\leq N}$ to be the $N$-quantiles of $\mu_1$ and $\mu_2$, see \eqref{def: quantile}.  Denote $x_0^\star := \inf \supp \mu_1$, $y_0^\star := \inf \supp \mu_2$. By definition,
	\ben
	\begin{split}
		\absa{m_2(z)-\mathfrak{m}_2(z)} 
		&\leq \sum_k\absa{ \int_{y_{k-1}^\star}^{y_k^\star} \frac{\mu_1(\dd y)}{y-z} - \frac{1}{\bar{Y}_k-z)}}\\
		& \leq \frac{1}{N}\sum_k \left[\absa{\frac{1}{y_k^\star-z}-\frac{1}{\bar{Y}_k-z}} + \sup_{y\in[y_{k-1}^\star,y_k^\star]}\absa{\frac{1}{y_k^\star-z}-\frac{1}{y-z}} \right]\\
		& \leq \frac{1}{N}\sum_k \left(\frac{N^{-\frac{9}{10}}}{\eta^2}+\frac{y_k^\star-y_{k-1}^\star}{\eta^2}\right)\,.
	\end{split}
	\ee
	In the last inequality we have used the bound \eqref{reg barY}.  Hence we have
	\be\label{ineq: m2-fm2}
	\absa{m_2(z)-\mathfrak{m}_2(z)}\leq N^{-\frac{9}{10}}\eta^{-2}\,,
	\ee 
	for sufficiently large $N$.  The conclusion follows.
\end{proof}

We consider the deterministic, $N$-independent equations, which is the limiting form of \eqref{eq: sce4}:
\be\label{eq: scelimit}
\left\{\begin{split}
	m&= m_2(z-w_1)\\
	m&=m_1(z-w_2)\\
	\frac{1}{m}&= w_1+w_2-z
\end{split}\right.
\ee
The following Proposition is a summary of Theorem 2.3 in \cite{Be06} and Theorem 3.3 \cite{Be08}.
\begin{prop}\label{prop: w continuity}
	The system \eqref{eq: scelimit} has a unique solution $(m,w_1,w_2):\C^+\to \C^+\times \C^-\times \C^-$ which depends holomorphically on $z$. $m(z)$ can be extended analytically to $I\subset\R$ if $\mu=\mu_1\boxplus \mu_2$ has a density on $I$ 
	 bounded away from zero. Moreover, $(w_1,w_2)$ can be continuously extended to $\R$. In particular, $\abs{w_1}\vee \absa{w_2}$ is uniformly bounded on compact subsets of $\C^+\cup\R$.
\end{prop}

We view \eqref{eq: sce4} as a perturbed version of \eqref{eq: scelimit}. The following theorem proves that the solution $(\fm,\fw_1,\fw_2)$ is not far from the solution to \eqref{eq: scelimit}.
\begin{thm}\label{thm: w bound}
	Let $(\fm,\fw_1,\fw_2)$ be the solution to \eqref{eq: sce4}. There is a universal constant $p>0$ such that the following holds. For $z\in \{z=E+\ii \eta: E\in[-2K,2K], \eta\in [N^{-p^{-2}},1]\}$, we uniformly have 
	\ben
		\absa{w_1-\fw_1}\vee \absa{w_2-\fw_2}\leq N^{-\frac{1}{p}}\eta^{-p}\,, \quad \absa{\fw_1}\vee\absa{\fw_2}\lesssim 1\,.
	\ee
	
\end{thm}

\begin{proof}
	The strategy is similar to what we did in the proof of Theorem \ref{thm: edge law}.  We write equations \eqref{eq: scelimit} as 
	\be\label{eq: scelimit2}
	\left\{\begin{split}
		-w_2&= \hat{m}_2(z-w_1)\\
		-w_1&=\hat{m}_1(z-w_2)
	\end{split}\right.
	\ee
	Here $\hat{m}_1(\zeta):= -\zeta-\frac{1}{m_1(\zeta)}$, $\hat{m}_2(\zeta):=-\zeta-\frac{1}{m_2(\zeta)}$.  According to Proposition 2.2 in \cite{Ma92}, $\hat{m}_1$ and $\hat{m}_2$ are Stieltjes transforms of Borel measures $\hat{\mu}_1$ and $\hat{\mu}_2$ on $\R$ that have  total mass $\int_{\R} x^2\mu_1(\dd x)$ and $\int_ \R y^2\mu_2(\dd y)$, respectively.  We write \eqref{eq: sce4} as 
	\be\label{eq: scelimit3}
	\left\{\begin{split}
		-\fw_2&= \hat{m}_2(z-\fw_1)+r_1\\
		-\fw_1&=\hat{m}_1(z-\fw_2)+r_2
	\end{split}\right.\,.
	\ee
	Above, the error terms $r_1$ and $r_2$ are given by
	\be\label{def: r1r2}
	r_1:= \frac{1}{m_2(z-\fw_1)}- \frac{1}{\fm_2(z-\fw_1)}\,,\quad r_2:= \frac{1}{m_1(z-\fw_2)} -\frac{1}{\fm_1(z-\fw_2)}\,.
	\ee

First, we derive upper bounds on $\absa{w_1-\fw_1}$ and $\absa{w_2-\fw_2}$ in terms of $r_1$ and $r_2$.  Subtracting \eqref{eq: scelimit2} from \eqref{eq: scelimit3}, we obtain
	\ben
		\left\{\begin{split}
			w_2-\fw_2&= \int_{\R} \frac{(w_1-\fw_1)\hat{\mu}_2(\dd x)}{(x-z+w_1)(x-z+\fw_1)}+r_1\\
			w_1-\fw_1&=\int_{\R} \frac{(w_2-\fw_2)\hat{\mu}_1(\dd x)}{(x-z+w_2)(x-z+\fw_2)}+r_2
		\end{split}\right.\,.
	\ee
Noting that the equations are linear in $w_2-\fw_2$ and $w_1-\fw_1$, we are lead to denote the coefficients by
	
\be\label{def: ab}
	\mathrm{a}:= \int_{\R} \frac{\hat{\mu}_2(\dd x)}{(x-z+w_1)(x-z+\fw_1)}\,,\quad \mathrm{b}:= \int_{\R} \frac{\hat{\mu}_1(\dd x)}{(x-z+w_2)(x-z+\fw_2)}\,.
\ee
The notations enable us to write the above equations as
	\ben
		\begin{bmatrix}
			-\mathrm{a} & 1\\ 1 & -\mathrm{b} 
		\end{bmatrix} 
		\begin{bmatrix}
			w_1-\fw_1\\w_2-\fw_2
		\end{bmatrix} = \begin{bmatrix} 
		r_1\\r_2
	\end{bmatrix}\,.
\ee
Inverting the 2 by 2 matrix on the left hand side, we get
\ben
	\begin{bmatrix}
		w_1-\fw_1\\w_2-\fw_2 
	\end{bmatrix} = \frac{1}{1-\mathrm{a}\mathrm{b}} \begin{bmatrix}
	\mathrm{b} & 1\\1 & \mathrm{a} 
\end{bmatrix}\begin{bmatrix}
r_1\\r_2
\end{bmatrix}\,.
\ee
A simple calculation yields the following upper bound
\be\label{ineq: w1-fw1}
	\absa{w_1-\fw_1}\vee \absa{w_2-\fw_2} \leq \frac{\absa{\mathrm{b}}+\absa{\mathrm{a}}+1}{\absa{1-\mathrm{a}\mathrm{b}}}(\absa{r_1}+\absa{r_2})\,.
\ee

In order to bound the right hand side of \eqref{ineq: w1-fw1}, we need estimates on $r_1$ and $r_2$ as well as $ \frac{\absa{\mathrm{b}}+\absa{\mathrm{a}}+1}{\absa{1-\mathrm{a}\mathrm{b}}}$. We first derive upper bounds for $\absa{r_1}$ and $\absa{r_2}$. According to Proposition \ref{prop: m1-fm1}, we have
	\ben
	\absa{r_1}\leq \frac{N^{-\frac{4}{5}}}{\absa{m_2(z-\fw_1)\fm_2(z-\fw_1)}}\,,\quad 	\absa{r_2}\leq \frac{N^{-\frac{1}{p}}}{\absa{m_1(z-\fw_2)\fm_1(z-\fw_2)}}
	\ee
	From \eqref{eq: sce6} we have a crude upper bound $\abs{\fw_1}\vee\abs{\fw_2}\leq c\eta^{-1}$.  Noting that for the Stieltjes transform $\zeta\mapsto s(\zeta)$ of any probability measure $\nu$ on $\R$,  we have a simple lower bound $\im s(\zeta)\geq \im \zeta (\absa{\zeta}+ \sup_{x\in\supp\nu }\abs{x})^{-2}$.  Therefore, the inequalities above yield upper bounds
	\be\label{ineq: r1r2}
	\absa{r_1}+\absa{r_2}\leq N^{-\frac{1}{p}}\eta^{-3}\,.
	\ee
	Letting $\eta\geq N^{-\frac{1}{5p}}$, we have
	\be\label{ineq: r1r2eta}
	\absa{r_1}+\absa{r_2}\leq \eta/2\,.
	\ee

Next, we derive an upper bound for $ \frac{\absa{\mathrm{b}}+\absa{\mathrm{a}}+1}{\absa{1-\mathrm{a}\mathrm{b}}}$. We use Cauchy-Schwartz to see
\be
	\absa{\mathrm{a}}\leq \left(\int_{\R} \frac{\hat{\mu}_2(\dd x)}{\absa{x-z+w_1}^2} \int_ \R \frac{\hat{\mu}_2(\dd x)}{\abs{x-z+\fw_1}^2}\right)^\frac{1}{2}\,,\absa{\mathrm{b}}\leq \left(\int_{\R} \frac{\hat{\mu}_1(\dd x)}{\absa{x-z+w_2}^2} \int_ \R \frac{\hat{\mu}_1(\dd x)}{\abs{x-z+\fw_2}^2}\right)^\frac{1}{2}\,,
\ee
We immediately have a crude bound
\ben
	\absa{\mathrm{a}}+\absa{\mathrm{b}}+1\leq c \eta^{-2}\,.
\ee
To get a lower bound for $\absa{1-\mathrm{a}\mathrm{b}}$, one takes the imaginary part of the first equation in \eqref{eq: scelimit3}, then divide both sides by $\eta-\im\fw_1$ to see
\ben
	\int_ \R \frac{\hat{\mu}_2(\dd x)}{\absa{x-z+\fw_1}^2} = \frac{-\im\fw_2-\im r_1}{\eta-\im\fw_1}\,.
\ee
Taking the imaginary part of the second equation in \eqref{eq: scelimit3} yields $$\int_ \R \frac{\hat{\mu}_1(\dd x)}{\absa{x-z+\fw_2}^2} = \frac{-\im\fw_1-\im r_2}{\eta-\im\fw_2}.$$ Similarly taking the imaginary part of both equations in \eqref{eq: scelimit2}, we have
\ben
	\int_ \R \frac{\hat{\mu}_2(\dd x)}{\absa{x-z+w_1}^2} = \frac{-\im w_2}{\eta-\im w_1}\,,\quad \int_ \R \frac{\hat{\mu}_1(\dd x)}{\absa{x-z+w_2}^2} = \frac{-\im w_1}{\eta-\im w_2}\,.
\ee
Hence we have a lower bound for $\abs{1-\mathrm{ab}}$,
\ben
	\abs{1-\mathrm{ab}}\geq 1- \sqrt{\frac{-\im w_1}{\eta-\im w_1}\frac{-\im w_2}{\eta-\im w_2} \frac{-\im \fw_1}{\eta-\im\fw_1}\frac{-\im\fw_2}{\eta-\im\fw_2}}\geq \eta^p\,.
\ee
Hence we have the following estimate for some new universal constant $p$,
\ben
	 \frac{\absa{\mathrm{b}}+\absa{\mathrm{a}}+1}{\absa{1-\mathrm{a}\mathrm{b}}}\leq \eta^{-p}\,.
\ee
This estimate with inequality \eqref{ineq: w1-fw1} and upper bounds \eqref{ineq: r1r2} yields,
\ben
	\abs{w_1-\fw_1}\vee \abs{w_2-\fw_2}\leq N^{-\frac{1}{p}}\eta^{-p}\,.
\ee
The conclusion follows from Proposition \ref{prop: w continuity} and letting $\eta\geq N^{-1/ p^2}$.
\end{proof}

\begin{cor}\label{cor: fm bd}
	Under the same assumptions as Theorem \ref{thm: w bound}, we have
	\ben
	\absa{\fm(z)}\lesssim 1\,.
	\ee
\end{cor}

\begin{proof}
	Recall that in Section \ref{subsec: ass} we assumed that either $\mu_1$ or $\mu_2$ has a bounded Stieltjes transform. Let the bound be $\alpha>0$. By Proposition \ref{prop: m1-fm1} we know for $\eta\geq N^{-\frac{1}{p}}$,
	\ben
		\absa{\fm_1(z-\fw_2)}\wedge\absa{\fm_2(z-\fw_1)}\leq \absa{m_1(z-\fw_2)}\wedge\absa{m_2(z-\fw_1)}+N^{-\frac{1}{p}}\leq  2\alpha\,.
	\ee
	By \eqref{eq: sce4}, we see that this actually means
	\ben
		\absa{\fm(z)}\leq 2\alpha\,.
	\ee
\end{proof}

In the bulk of the spectrum, the stability becomes better, in the sense that the imaginary part of $z$ can be as small as possible, as is shown by the following theorem.
\begin{thm}\label{thm: w1-fw1 bulk}
	Let $I$ be an interval on which the probability measure $\mu=\mu_1\boxplus \mu_2$ has a density bounded away from $0$ and bounded above. Let $\Sigma=\{ z=E+\ii \eta: E \in I, \eta\in[0,1]\}$.  Then, there is a constant $q>0$ such that
	\ben
		\sup_{z\in \Sigma} \absa{w_1-\fw_1}\vee \absa{w_2-\fw_2} \leq  N^{-\frac{1}{q}}\,,
	\ee
	\ben 
		\sup_{z\in \Sigma} \absa{m(z)-\fm(z)}\leq  N^{-\frac{1}{q}}\,.
	\ee
\end{thm}

\begin{proof}
	First, we proceed exactly as in the proof of Theorem \ref{thm: w bound}, up to \eqref{ineq: w1-fw1}.  In the rest of this proof, the $I$-dependent constant $c_I$ may change from line to line, but only changes for finitely many times. Let $q>0$ be a positive constant to be chosen. Define a subset $\Sigma_1$ of $\Sigma$ by
	\ben
		\Sigma_1 = \{z\in\Sigma: \absa{w_1-\fw_1}\vee \absa{w_2-\fw_2} \leq  N^{-\frac{1}{q}}\}\,.
	\ee
	The set $\Sigma_1$ is not empty, thanks to Theorem \ref{thm: w bound}. It is a closed subset of $\Sigma$, because all the functions $w_1$, $\fw_1$, $w_2$ and $\fw_2$ are continuous in $\Sigma$.  We shall show that $\Sigma_1$ is an open subset of $\Sigma$, which would imply $\Sigma_1=\Sigma$.
	
	Take any point $z=E+\ii \eta\in\Sigma_1$. According to the assumption that $\mu=\mu_1\boxplus \mu_2$ has a density bounded above and bounded away from zero, we have $\im m \in [c_I,c_I^{-1}]$. According to Proposition \ref{prop: w continuity}, the Stieltjes transform $m$ of $\mu$ can be analytically extended to $I$, which implies that $\absa{m}<c_I$.  From \eqref{eq: scelimit2} we know
	\ben
		-\im w_2 \geq -c_I \im w_1\,,\quad -\im w_1 \geq -c_I \im w_2\,.
	\ee
	Taking the imaginary part of the third equation of  \eqref{eq: scelimit} we obtain 
	\ben
		-\im w_1-\im w_2 \geq c_I>0\,.
	\ee
	The above three inequalities give
	\ben
	(	-\im w_1 )\wedge(-\im w_2)\geq c_I>0\,.
	\ee
	Therefore, on the set $\Sigma_1$ we have
	\ben
		(-\im\fw_1)\wedge(-\im\fw_2)\geq c_I>0\,,\quad \absa{\fw_1}\vee \absa{\fw_2}\leq c_I^{-1}\,.
	\ee
	Substituting $z$ in Proposition \ref{prop: m1-fm1} with $z-\fw_2$ and $z-\fw_1$, we see that 
	\be\label{ineq: m-fm}
		\absa{m_1(z-\fw_2)-\fm_1(z-\fw_2)}\vee \absa{m_2(z-\fw_1)-\fm_2(z-\fw_1)}\leq 2 N^{-\frac{1}{p}}\,.
	\ee
	Recalling the definition of $r_1$ and $r_2$ (see \eqref{def: r1r2}), the above estimate yields
	\be\label{ineq: r1r2 1}
		\absa{r_1}\vee \absa{r_2}\leq c_I N^{-\frac{1}{p}}\,.
	\ee
	In order to make use of the bound \eqref{ineq: w1-fw1}, we need an estimate on $\frac{\absa{\mathrm{b}}+\absa{\mathrm{a}}+1}{\absa{1-\mathrm{ab}}}$.  By definition \eqref{def: ab} of $\mathrm{a}$ and $\mathrm{b}$ and the lower bounds on $-\im\fw_1$ and $-\im\fw_2$, we easily get
	\be\label{ineq: a b}
		\absa{a}\vee \absa{b}\leq c_I>0\,,
	\ee
	which is enough for bounding the numerator of $\frac{\absa{\mathrm{b}}+\absa{\mathrm{a}}+1}{\absa{1-\mathrm{ab}}}$.  To bound the denominator, we consider 
	\ben
		\mathfrak{a}:=\int_ \R \frac{\hat{\mu}_2(\dd x)}{(x-z+w_1)^2} \,,\quad \mathfrak{b}:=\int_ \R \frac{\hat{\mu}_2(\dd x)}{(x-z+w_2)^2}\,.
	\ee
	Note that by assumption, on $\Sigma_1$ we have
	\be\label{ineq: fa-a}
		\absa{\mathfrak{a}-\mathrm{a}}\vee \absa{\mathfrak{b}-\mathrm{b}}\leq c_I N^{-\frac{1}{p}}\,.
	\ee
	Taking the imaginary part of the first and second equations in \eqref{eq: scelimit2}, we have
	\ben\left\{
		\begin{split}
		-\im w_2 &= (-\im w_1 +\eta) \int_ \R \frac{\hat{\mu}_2(\dd x)}{\absa{x-z+w_1}^2}\\
		-\im w_1 & = (-\im w_2+\eta)\int_ \R \frac{\hat{\mu}_1(\dd x)}{\absa{x-z+w_2}^2}
		\end{split}\right.\,,
	\ee
	which yields,
	\ben
		 \int_ \R \frac{\hat{\mu}_2(\dd x)}{\absa{x-z+w_1}^2}\int_ \R \frac{\hat{\mu}_1(\dd x)}{\absa{x-z+w_2}^2}\leq \frac{-\im w_2}{-\im w_2+\eta} \frac{-\im w_1}{-\im w_1+\eta}\leq 1\,.
	\ee
	The estimate holds on  $\Sigma$ as well as on $I$. Since $w_1$ and $w_2$ has non-zero imaginary part, and that  $\hat{\mu}_2$ is not a point-mass, we have a strict inequality for $\mathfrak{a}$ and $\mathfrak{b}$,
	\ben
		\absa{\mathfrak{ab}}<\int_ \R \frac{\hat{\mu}_2(\dd x)}{\absa{x-z+w_1}^2}\int_ \R \frac{\hat{\mu}_1(\dd x)}{\absa{x-z+w_2}^2}\leq 1\,.
	\ee
	By the continuity of $\mathfrak{a}$ and $\mathfrak{b}$ on $I$, we have
	\ben
		\absa{1-\mathfrak{ab}}\geq c_I>0\,.
	\ee
	This estimate with \eqref{ineq: fa-a} gives a lower bound for the numerator of $\frac{\absa{\mathrm{b}}+\absa{\mathrm{a}}+1}{\absa{1-\mathrm{ab}}}$,
	\ben
		\absa{1-\mathrm{ab}}\geq c_I>0\,.
	\ee
	We plug this estimate and \eqref{ineq: r1r2 1} and \eqref{ineq: a b} into \eqref{ineq: w1-fw1} to get, for any $z\in \Sigma_1$, 
	\ben
		\absa{w_1-\fw_1}\vee \absa{w_2-\fw_2} \leq c_I N^{-\frac{1}{p}}\,.
	\ee
	Therefore, for any $z\in\Sigma_1$, there is a neighborhood of $z$ such that
	\ben
		\absa{w_1-\fw_1}\vee \absa{w_2-\fw_2} \leq N^{-\frac{1}{p}}\,,
	\ee
	which implies that $\Sigma_1$ is an open subset of $\Sigma$, as long as $q>p$.  Therefore, $\Sigma_1$ is an open and closed non-empty subset of $\Sigma$, so $\Sigma_1=\Sigma$. 
	
	We already obtained \eqref{ineq: m-fm} on $\Sigma_1$. Therefore,\eqref{ineq: m-fm} holds on the entire $\Sigma$. By equations  \eqref{eq: sce4} and \eqref{eq: scelimit}, it follows that on $\Sigma$ we have
	\ben	
		\absa{m(z)-\fm(z)}\leq 2N^{-\frac{1}{10}} + \absa{m_1(z-\fw_2)-m_1(z-w_2)}\leq N^{-\frac{1}{q}}\,.
	\ee
\end{proof}

\subsection{Global upper bounds for $w_{\alpha k}$ and $\ga_{ij}$}

We now derive an estimate which holds for all the coefficients appearing in \eqref{eq: dla}.   In the next subsection we derive stronger estimates which hold for terms corresponding to the bulk.

\begin{thm}\label{thm: edge deloc}
There is a universal constant $p>0$ such that the following holds. With probability $1-\e^{-N^{9\mathfrak{b}}}$, we have
\ben
	\sup_{t\in[0,T]} \sup_{1\leq \alpha,k\leq N}\abs{w_{\alpha k}}^2 \leq N^{-1/p}\,,\quad
\sup_{t\in[0,T]} \sup_{1\leq i<j\leq N}\ga_{ij}\leq N^{-2/p+\mathfrak{a}}\,.
\ee
\end{thm}

\begin{proof}
	Fix $\eta=N^{-1/p}$ and  $1\leq \alpha,k\leq N$, where $p>0$ is a large constant to be chosen. First, we bound $\abs{w_{\alpha k}}^2$ in terms of $G_{\alpha\alpha}$. For each $(z=E+\ii\eta,t)\in\Sigma\times [0,T]$, note that $\im\cG_{\alpha\alpha}(E+\ii \eta) = \sum_k \frac{\eta\abs{w_{\alpha k}}^2}{\abs{\la_k-E}^2+\eta^2}$, therefore, 
	\ben
		\abs{w_{\alpha k}}^2 \leq \eta \im\cG_{\alpha\alpha}(\la_k+\ii \eta) \,.
	\ee
	In the following, we denote $z_k:=\la_k+\ii\eta$. Then, we want to use the estimate \eqref{ineq: cGii} obtained above, which says that $\cG_{\alpha\alpha}(z_k)$ is approximately $\left(-z_k+\bar{Y}_\alpha +\mathfrak{w}_1\right)^{-1}$. Therefore, we need an upper bound on $\absa{-z_k+\bar{Y}_\alpha +\mathfrak{w}_1}^{-1}$.  It is sufficient to get lower bounds on $\absa{-\la_k+\bar{Y}_\alpha+\re\mathfrak{w}_1}\vee(-\im\mathfrak{w}_1)$, because 
	\be\label{wakimG}
		\absa{\frac{1}{-z_k+\bar{Y}_\alpha +\mathfrak{w}_1}}\leq \frac{1}{\abs{-\la_k+\bar{Y}_\alpha+\re\mathfrak{w}_1}\vee(-\im\mathfrak{w}_1)}\,.
	\ee
	
	We claim that 
	\be\label{imw1}
		\absa{-\la_k+\bar{Y}_\alpha+\re\mathfrak{w}_1} \vee(-\im\mathfrak{w}_1)\geq \eta^{1-\mathfrak{c}/4}\,.
	\ee
	We will prove this by contradiction and assume that $\absa{-\la_k+\bar{Y}_\alpha+\re\mathfrak{w}_1} \vee(-\im\mathfrak{w}_1)< \eta^{1-\mathfrak{c}/4}$.  We take the imaginary part of \eqref{eq: sce5} to see
	\be\label{-imw1-imw2}
		-\im\mathfrak{w}_1-\im\mathfrak{w}_2=-\eta+\frac{\im \mathfrak{m}_2(z_k-\mathfrak{w}_1)}{\absa{\mathfrak{m}_2(z_k-\mathfrak{w}_1)}^2}\geq -\eta+c\im\mathfrak{m}_2(z_k-\mathfrak{w}_1)\,.
	\ee
	In the last inequality we have used Corollary \ref{cor: fm bd}.  Note that by definition of $\mathfrak{m}_2$,
	\be\label{Naeta}
		\begin{split}
		\im\mathfrak{m}_2(z_k-\mathfrak{w}_1) &= \frac{1}{N}\sum_{\beta=1}^N \frac{\eta-\im\mathfrak{w}_1}{\absa{-\la_k+\bar{Y}_\beta+\re\mathfrak{w}_1}^2+\absa{\eta-\im\mathfrak{w}_1}^2}\\
		&\geq \frac{1}{N} \sum_{\absa{\bar{Y}_\beta-\bar{Y}_\alpha}\leq \eta}  \frac{\eta}{\absa{-\la_k+\bar{Y}_\beta+\re\mathfrak{w}_1}^2+\absa{\eta-\im\mathfrak{w}_1}^2}\\
		&\geq \frac{N_{\alpha,\eta}}{N}\frac{\eta}{4\eta^{2-\mathfrak{c}/2}} = \frac{N_{\alpha,\eta}\eta^{\mathfrak{c}/2-1}}{4N}
		\end{split}
	\ee
	In the last line above, $N_{\alpha,\eta}:=\# \{\beta:\absa{\bar{Y}_\beta-\bar{Y}_\alpha}\leq \eta\}$.  
	Recalling \eqref{reg barY} we have, 
	\ben
		N_{\alpha,\eta}\geq \# \{\beta:\absa{y_\beta-y_\alpha}\leq 2\eta/3\}\geq N\mu_2([y_\alpha-\eta/2,y_\alpha+\eta/2])\,.
	\ee
	According to \eqref{edge behavior}, we obtain
	\ben	
		N_{\alpha,\eta}\geq cN\eta^{2-\mathfrak{c}}\,.
	\ee
	Therefore, \eqref{Naeta} yields
	\ben
		\im\mathfrak{m}_2(z_k-\mathfrak{w}_1)\geq c\eta^{1-\mathfrak{c}/2}\,.
	\ee
	Going back to \eqref{-imw1-imw2} and absorb the constants by small power of $\eta$, we have
	\ben
		-\im\mathfrak{w}_1-\im\mathfrak{w}_2 \geq 2\eta^{1-\mathfrak{c}/3}\,.
	\ee
	However, we have assumed that $-\im\mathfrak{w}_1<\eta^{1-\mathfrak{c}/4}$. Hence, the above inequality gives
	\ben
		-\im\mathfrak{w}_2\geq \eta^{1-\mathfrak{c}/3}\,.
	\ee
	Then we take the imaginary part of the second equation of \eqref{eq: sce6},
	\ben
		-\im \mathfrak{w}_1 =\int_{\R} \frac{(\eta-\im\mathfrak{w}_2)\hat{\mu}_2(\dd x)}{\absa{x-z_k+\mathfrak{w}_2}^2}\geq \eta^{1-\mathfrak{c}/3} \int_{\R} \frac{\hat{\mu}_2(\dd x)}{\absa{x-z_k+\mathfrak{w}_2}^2}\,.
	\ee
	By Theorem \ref{thm: w bound}, we know that
	\ben
		\int_{\R} \frac{\hat{\mu}_2(\dd x)}{\absa{x-z_k+\mathfrak{w}_2}^2}\geq \int_{\R} \frac{\hat{\mu}_2(\dd x)}{c(\absa{x}^2+1)}\geq c_1\im\hat{\mathfrak{m}}_2(\ii )\,.
	\ee
	By definition of $\hat{\mathfrak{m}}_2$ and Proposition \ref{prop: m1-fm1}, the quantity $\im\hat{\mathfrak{m}}_2(\ii )$ is bounded below by
	\ben
		\im\hat{\mathfrak{m}}_2(\ii )=\im \hat{m}_2(i)+\mathcal{O}(N^{-\frac{1}{p}})\geq c\,,
	\ee
	where $c>0$ is a universal constant. Therefore, we get, for large enough $N$,
	\ben
		-\im \mathfrak{w}_1\geq c\eta^{1-\mathfrak{c}/3}\geq \eta^{1-\mathfrak{c}/4}\,,
	\ee
	which leads to a contradiction. Thus, we have proved the claim \eqref{imw1}.  In view of \eqref{wakimG} and \eqref{ineq: cGii}, we  immediately have
	\ben
		\abs{w_{\alpha k}}^2\leq \eta\im\cG_{\alpha\alpha}(\la_k+\ii \eta) \leq \eta^{\mathfrak{c}/4}+\mathcal{O}\left(N^{\nu-1/2}\eta^{-15}\right)\,.
	\ee
	The bound for $\gamma_{ij}$ follows from its definition and taking $p$ large enough.
\end{proof}

\subsection{Upper bounds for $w_{\alpha k}$ and $\ga_{ij}$ in the bulk}
In this subsection, we prove some estimates on the quantities $(w_{\alpha k})$ and $(\ga_{ij})$ that appear in the coefficients in \eqref{eq: dla} which have indices corresponding to bulk eigenvalues of  $\tilde{H}(t)$. Recall that $(w_{\beta k})_{1\leq \beta\leq N}$ is the $k$-th eigenvector of
\ben 
\cH(t)=\bar{V}(t)^*X\bar{V}(t) + Y + (T-t) U(t)\hat{Y}U(t)^*\,,
\ee
where $\bar{V}(t):= VU(t)^*$. We defined the Green's function by
\ben
\cG(z,t):= (\cH(t)-z)^{-1}\,,\forall z\in\C^+\,.
\ee
The probability distribution of $\bar{V}(t)$ is Haar measure, independent from $U(t)$. Recall that in Section \ref{subsec: u-i} we proved $\norm{U(t)-I}\ll 1$.  We therefore write 
\be\label{eq: cH11}
\cH(t)=\bar{V}(t)X\bar{V}(t)+Y+(T-t)\hat{Y} + (T-t)(U(t)\hat{Y}U(t)^*-\hat{Y})\,.
\ee
In view of \eqref{hatY bound} and Theorem \ref{thm: Ut-I}, the last term above satisfies
\ben
\P\left[\sup_{0\leq t\leq T}\norm{(T-t)(U(t)\hat{Y}U(t)^*-\hat{Y})}\geq N^{-1-8\mathfrak{b}}\right]\leq \e^{-N^{10\mathfrak{b}}}\,.
\ee
For the first three terms in \eqref{eq: cH1}, we will apply Theorem 2.5 of \cite{BES16} to get a bound in the bulk of the spectrum. In sum, we are able to prove Theorem \ref{thm: bulk law} below for $\cG_{kk}(z,t)$ for $z$ near the bulk.  Before stating the theorem, we introduce the following notion of overwhelming probability.
\begin{defn}
	A sequence of events $(\mathcal{A}_N)_{N\geq 0}$ is said to hold with overwhelming probability, if for any $L\geq 0$, we have
	\be
	\P[\mathcal{A}_N] \geq 1- N^{-L}\,,\text{ for } N\geq N(L)\,,
	\ee
	for some $N(L)$ depending only on $L$ and universal constants.	
\end{defn}

\begin{thm}\label{thm: bulk law}
	Let $I$ be an interval such that the measure $\mu=\mu_1\boxplus \mu_2$ restricted to $I$ has a strictly positive density. Denote 
	\be
	\mathcal{D}_{I,\nu}:= \{ z=E+\ii \eta: E\in I,N^{-1+\nu}\leq \eta \leq 1\}\,.
	\ee 
	Let $(\fw_1,\fw_2,\fm)$ be the solution to the system \eqref{eq: sce4}. Then we have 
	\be\label{eqn:w1stab}
		\inf_{z\in\mathcal{D}_{I,\nu}}(-\im \fw_1(z))\geq c\,,
	\ee
	and that for any $\nu>0$, the following holds with overwhelming probability,
	\be
	\sup_{0\leq t\leq T}\sup_{z\in\mathcal{D}_{I,\nu}}\max_{1\leq i\leq N}\absa{\cG_{ii}(z,t)- \frac{1}{-z+y_i+(T-t)\hat{y}_i+\fw_1(z) }}\leq \frac{N^\nu}{\sqrt{N\eta}}\,.
	\ee
\end{thm}

\begin{proof}
	Denote $\hat{\cH}(t):= \bar{V}(t)^* X\bar{V}(t)+Y+(T-t)\hat{Y}$ and 
	\ben
	\hat{\cG}(z,t):=(\hat{\cH}(t)-z)^{-1}\,.
	\ee
	Note that by assumption $X$ has a decomposition $X=X_0+ \e^{-N}Q$, where the empirical measure of eigenvalues of $X_0$ converges weakly to $\mu_1$ and $Q$ is drawn from the Gaussian Unitary ensemble.  Therefore, the empirical measure of eigenvalues of $X$ converges weakly to $\mu_1$ almost surely. Recall that $T=N^{-1+\mathfrak{b}}$ and the bound \eqref{hatY bound}, we have
	\ben
	\norm{(T-t)\hat{Y}}\leq N^{-1+\mathfrak{b}}\log N\,.
	\ee
	It follows that the empirical measure of eigenvalues of $Y+(T-t)\hat{Y}$ converges to $\mu_2$ weakly.  Therefore, the conditions of Theorem 2.4 in \cite{BES16} are satisfied. Therefore we have, for any fixed $t\in[0,T]$, and $\nu>0$,
	\be\label{ineq: hatcG}
	\sup_{z\in\mathcal{D}_{I,\nu}}\max_{1\leq i\leq N}\absa{\hat\cG_{ii}(z,t)- \frac{1}{y_i+(T-t)\hat{y}_i+\fw_1(z) }}\leq \frac{N^\nu}{\sqrt{N\eta}}\,,
	\ee
	with overwhelming probability, with $\fw_1$ satisfying \eqref{eqn:w1stab}.
	
	Next, we estimate the difference $\abs{\hat{\cG}_{ii}(z,t)-\cG_{ii}(z,t)}$.  By the resolvent identity $Q_1^{-1}-Q_2^{-1}=Q_1^{-1}(Q_2-Q_1)Q_2^{-1}$, we have
	\ben
	\begin{split}
		\abs{\hat{\cG}_{ii}-\cG_{ii}} &=\abs{\sum_{k,l} \hat{\cG}_{ik}(\cH-\hat{\cH})_{kl}\cG_{li}}\\
		&\leq \norm{\cH-\hat{\cH}} \sqrt{\sum_k \abs{\hat{\cG}_{ik}}^2}\sqrt{\sum_k \abs{\cG_{li}}^2}\,.
	\end{split}
	\ee
	By the Ward Identity, $\sqrt{\sum_k \abs{\hat{\cG}_{ik}}^2}=\sqrt{\im \hat{\cG}_{ii}/\eta}$, $\sqrt{\sum_k \abs{{\cG}_{ik}}^2}=\sqrt{\im {\cG}_{ii}/\eta}$.  
	We have that $\norm{\cH-\hat{\cH}}\leq N^{-1-8\mathfrak{b}}$ with probability $1-\e^{-N^{10\mathfrak{b}}}$.  Therefore, with probability $1-\e^{-N^{10\mathfrak{b}}}$, for $\forall z\in\mathcal{D}_{I,\nu},t\in[0,T]$,
	\ben
	\begin{split}
		\abs{\hat{\cG}_{ii}-\cG_{ii}} &\leq  \frac{ N^{ - 8 \mfb} }{ N \eta} \sqrt{\im \hat{\cG}_{ii}\im {\cG}_{ii}}\\
		&\leq \frac{ N^{ - 8 \mfb} }{ N \eta} \sqrt{\abs{\hat{\cG}_{ii}} (\abs{\hat{\cG}_{ii}} +\abs{\hat{\cG}_{ii}-\cG_{ii}} )}\,.
	\end{split}
	\ee
	A simple calculation yields, for any $1\leq i\leq N$,
	\be\label{eq: hatG-G}
	\abs{\hat{\cG}_{ii}-\cG_{ii}} \leq 2 \frac{ N^{ - 8 \mfb} }{ N \eta}  \abs{\hat{\cG}_{ii}} \,.
	\ee
	This estimate together with \eqref{ineq: hatcG} and \eqref{eqn:w1stab} gives, for any $t\in[0,T]$,
	\be\label{ineq: cG}
	\sup_{z\in\mathcal{D}_{I,\nu}}\max_{1\leq i\leq N}\absa{\cG_{ii}(z,t)- \frac{1}{y_i+(T-t)\hat{y}_i+\fw_1(z) }}\leq \frac{N^\nu}{\sqrt{N\eta}}\,,
	\ee
	with overwhelming probability.
	
	To conclude the proof, we need to look at the continuity of $\mathcal G(z,t)$ with respect to $t$.   We divide the time interval $[0,T]$ into $N^{100}$ parts by $t_l:= lT/N^{100}$. For each $t_l$, we set $t=t_l$ in \eqref{ineq: cG}, and so by a union bound we have the estimate
	\be\label{ineq: cGtl}
	\sup_{0\leq l\leq N^{100}}\sup_{z\in\mathcal{D}_{I,\nu}}\max_{1\leq i\leq N}\absa{\cG_{ii}(z,t_l)- \frac{1}{y_i+(T-t_l)\hat{y}_i+\fw_1(z) }}\leq \frac{N^\nu}{\sqrt{N\eta}}\,,
	\ee 
	with overwhelming probability.  Again, by the resolvent identity,
	\be
	\norm{\cG(z,t)-\cG(z,s)}\leq \norm{\cG (z,t)}\norm{\cG(z,s)} \norm{\cH(s)-\cH(t)}  \leq N^2 \norm{\cH(s)-\cH(t)}  \,,\quad \forall 0\leq t,s\leq T\,.
	\ee
	By definition, $\norm{\cH(s)-\cH(t)}\leq 2\norm{U(s)-U(t)}(\abs{\norm{X}+\norma{Y}})$.  Recall that $X$ has a small Gaussian component of size $\e^{-N}$, $\norm{X}$ is bounded by $K+1$ with probability $1-\e^{-\e^{N/2}}$.  By Theorem \ref{thm: Us-Ut} we then derive the estimate
	\be
	\P\left[\sup_{1\leq l\leq N^{100} }\sup_{s\in[t_{l-1},t_{l+1}]}  \norm{\cG(z,t_l)-\cG(z,s)}\geq N^{-10}\right] \geq N^{100}\e^{-N^{\mathfrak{a}/3}}\,.
	\ee
	This estimate together with \eqref{ineq: cGtl} yields
	\ben
	\sup_{0\leq t\leq T}\sup_{z\in\mathcal{D}_{I,\nu}}\max_{1\leq i\leq N}\absa{\cG_{ii}(z,t)- \frac{1}{y_i+(T-t)\hat{y}_i+\fw_1(z) }}\leq \frac{N^\nu}{\sqrt{N\eta}}\,,
	\ee 
	with overwhelming probability.			
\end{proof}

\begin{cor}\label{cor: wbd}
	Let $I$ be an closed interval on which the probability measure $\mu:= \mu_1 \boxplus \mu_2$ has a  strictly positive density. Then,  for any $\nu>0$, the following estimates hold with overwhelming probability.
	\be
	\sup_{0\leq t\leq T}\max_{\la_k \in I}\max_{1\leq \alpha \leq N} \abs{w_{\alpha k}} \leq \frac{N^\nu}{\sqrt{N}}\,,\quad
	\sup_{0\leq t\leq T}\max_{\la_i \in I}\max_{j\neq i} \ga_{ij} + \ga_{ji} \leq \frac{N^{\mathfrak{a} +\nu}}{N}\,.
	\ee
\end{cor}
\begin{proof}
	For any $\la_k\in I$, we set  $z_k:=\la_k+\ii N^{-1+\nu}$. Denote $\eta=\im z$. Then,
	\ben
	\im \cG_{\alpha\alpha}(z_k,t) = \sum_{l}\frac{\im z_k\abs{w_{\alpha l}}^2}{({\la_l-\la_k})^2+\im z_k^2}\geq \frac{\im z_k\abs{w_{\alpha k}}^2}{(\la_k-\la_k)^2+\im z_k^2}=N^{1-\nu}\abs{w_{\alpha k}}^2.
	\ee
	It follows that
	\ben
	\abs{w_{\alpha k}}^2 \leq N^{-1+\nu} \im  \cG_{\alpha\alpha}(z_k,t) \,.
	\ee
	Taking the maximum over $k$, $\alpha$ and $t$, we have 
	\ben
	\sup_{0\leq t\leq T}\max_{\la_k \in I}\max_{1\leq \alpha \leq N} \abs{w_{\alpha k}} ^2 \leq N^{-1+\nu} \sup_{0\leq t\leq T}\max_{\la_k \in I}\max_{1\leq \alpha \leq N}  \abs{\cG_{\alpha\alpha}(z_k,t)}\,.
	\ee
	Theorem \ref{thm: bulk law} implies that the right hand side is bounded by $N^{-1+2\nu}$ with overwhelming probability. This gives the first estimate in the corollary.  The second follows from the definition of $\gamma_{ij}$ and the normalization $\sum_\alpha |w_{\alpha k } |^2 = 1$.
	%
\end{proof}

\begin{cor}[Estimate on the initial data $\tilde{H}(0)$] \label{cor:initlaw}
Under the same assumptions as Theorem \ref{thm: bulk law}, for some constant $p>0$ we have
\ben
	\sup_{z\in \mathcal{D}_{I,\nu}} \absa{\frac{1}{N}\tr \left(\frac{1}{\tilde{H}(0)-z}\right)- m(z)} \leq N^{-\frac{1}{p}}+\frac{N^\nu}{\sqrt{N\eta}}\,
\ee
with overwhelming probability.
\end{cor}

\begin{proof}
Note that $\tilde{H}(0)=\mathcal{H}(0)$. Theorem \ref{thm: bulk law} implies that with overwhelming probability,
\ben
	\sup_{z\in \mathcal{D}_{I,\nu}} \absa{\frac{1}{N}\tr \left(\frac{1}{\tilde{H}(0)-z}\right)- \frac{1}{N}\sum_i \frac{1}{-z+y_i+T\hat{y}_i+\fw_1(z)} }\leq \frac{N^\nu}{\sqrt{N\eta}}\, .
\ee
By definition of $\fm_2$ (see the paragraph before \eqref{eq: sce4}) and \eqref{eq: sce4}, the above inequality reads,
\ben
\sup_{z\in \mathcal{D}_{I,\nu}} \absa{\frac{1}{N}\tr \left(\frac{1}{\tilde{H}(0)-z}\right)- \fm(z)}\leq \frac{N^\nu}{\sqrt{N\eta}}\,.
\ee
The conclusion follows from this estimate and Theorem \ref{thm: w1-fw1 bulk}.
\end{proof}

\section{Analysis of the SDE}\label{sec: DBM}

Our starting point is the SDE 
\beq
\d \lambda_i = \frac{ \d B_i }{ \sqrt{N}}+ \frac{1}{N} \sum_j \frac{ 1 - \gamma_{ij}}{ \lambda_i - \lambda_j } + \d M_i + Z_i \d t ,
\eeq
for $0 \leq t \leq T$ with
\beq
T := \frac{ N^{\fb}}{N}
\eeq
The martingale $M_i$ is given by
\beq
\d M_i := - \frac{1}{ \sqrt{N}} \sum_{ |a-b| \leq N^{\fae}} w_{ai}^* \d B_{ab} w_{bi},
\eeq
and the drift term $Z_i$ is given by
\beq
Z_i = \langle a_i, ( U^* \hat{Y} U - \hat{Y} ) a_i \rangle.
\eeq
Thanks to Theorems \ref{thm: Ut-I} and \ref{thm: edge deloc} we have the estimates
\beq
\sup_{0 \leq t \leq T } |Z_i (t) | \leq N^{ - 9 \fb}, \quad \sup_{ 0 \leq t \leq T } |w_{ab} | + |\gamma_{ij} |  \leq N^{ - \frc_1} 
\eeq
for some $\frc_1 >0$ with overwhelming probability.  We assume that 
\beq
\fb < \fae/100, \qquad \fae < \frc_1 / 10.
\eeq
We will compare $\lambda_i$ to the process $\mu_i$ defined by
\beq
\d \mu_i = \frac{ \d B_i }{ \sqrt{N}} + \frac{1}{N} \sum_j \frac{1}{ \mu_i - \mu_j } , \qquad \mu_i (0) = \lambda_i (0).
\eeq
Let $I = (a, b)$ be a interval on which the limiting law $\mu_1 \boxplus \mu_2$ has a density bounded away from $0$ and above.  We use the notation
\beq
I_\kappa := (a + \kappa , b - \kappa ).
\eeq
We will also make use of the following index set.  Let $\gXY_i$ be the $i$th classical eigenvalue location of $\mu_1 \boxplus \mu_2$.  Define the index set $\Gkap$ by
\beq
\Gkap := \{ i : \gXY_i \in I_\kappa \}.
\eeq
Note that Corollary \ref{cor:initlaw} implies that
\beq \label{eqn:initrig}
| \gXY_i - \lambda_i (0) | \leq N^{ - c}
\eeq
for some $ c>0$ for $i \in \Gkap$ with overwhelming probability.

The main result of this section is the following Theorem.
\bet \label{thm:dbmcompare}  Fix $\kappa >0$.  Assume that $\fb < \fae / 100$ and $\fae < \frc_1 / 10$.  For every time $t$ with $0 \leq t \leq T$ we have with overwhelming probability we have for every index $i \in \Gkap$,
\beq
| \lambda_i  (t) - \mu_i (t) | \leq \frac{1}{N} \left( N^{ - \frc_1/5} + N^{-5 \fb} + N^{ -1/4 } \right).
\eeq
\eet

\subsection{Removal of error terms and regularization}
We fix a small $c_2 >0$ which satisfies
\beq
0 < c_2 < \frc_1.
\eeq
We introduce an auxilliary process $z_i (t, \alpha )$ for $0 \leq \alpha \leq 1$ by
\beq
\d z_i ( t, \alpha) = \sqrt{ \frac{1}{ 1 + N^{ - c_2 }}} \d B_i + \frac{1}{N} \sum_j \frac{ 1 - \alpha \hatgam_{ij}}{ z_i ( t, \alpha) - z_j (  t, \alpha) } \d t, \qquad z_i (0, \alpha ) = \lambda_i (0)
\eeq
where
\beq
\hatgam_{ij} = \gamma_{ij} \wedge N^{ - \frc_1}.
\eeq
The reason for the introduction of the $N^{-c_2}$ term is technical and only necessary in the case $\beta =1$.  In this case since the coefficient infront of the Brownian motion term and numerator of the drift term satisfies
\beq
(1 + N^{ c_2 } ) ( 1 - \alpha \gamhat_{ij} ) > 1
\eeq
the process $z_i (t, \alpha )$ is well-defined and satisfies $z_i (t, \alpha ) < z_{i+1} (t, \alpha )$ for every $t$.  This can be proven via the methods of \cite{AGZ}.  This implies that $z_i (t, \alpha)$ is a differentiable function of $\alpha$, which we will use later.  Without the regularizing $N^{-c_2}$ term, this would not be true in the $\beta=1$ case due to possible eigenvalue collisions.

The following compares the processes $z_i (t, 1)$ to $\lambda_i (t)$ and $z_i (t, 0)$ to $\mu_i (t)$.  The proof is essentially a regularized parabolic maximum principle, the regularization being needed to apply the It{\^o} formula and deal with the error term $M_i$.

\bel \label{lem:lamz}
With overwhelming probability we have
\beq \label{eqn:lamz}
\sup_i \sup_{0 \leq s \leq T} | z_i (s, 1) - \lambda_i (s) | \leq \frac{N^{\eps}}{N} \left( N^{ - c_4} + N^{- 8 \fb} + N^{\frb/2} ( N^{\fae/2-\frc_1} + N^{ - c_2 } ) + N^{ \frb + c_4} ( N^{ \fae-2 \frc_1} + N^{ - 2 c_2 } )   \right)
\eeq
for any $\eps >0$ and $c_4 >0$.  Similarly,
\beq \label{eqn:muz}
\sup_i \sup_{0 \leq s \leq T } | z_i (s, 0) - \mu_i (s) | \leq \frac{N^{\eps}}{N} \left( N^{ - c_4 } + N^{\fb/2} N^{ - c_2} + N^{ \fb + c_4 - 2 c_2} \right).
\eeq
\eel
\proof We only prove \eqref{eqn:lamz}.  The proof of \eqref{eqn:muz} is the same but easier.  Define a stopping time $\tau$ by
\beq
\tau := \tau_1 \wedge \tau_2 \wedge  \tau_3 \wedge \tau_4 \wedge T
\eeq
where
\beq
\tau_1 := \inf \{ t : \exists (i, j) : \gamma_{ij} > N^{ - \frc_1} \}
\eeq
and
\beq
\tau_2 := \inf \{ t: \exists (i, j) : |w_{ij} | > N^{ - \frc_1 } \}
\eeq
and
\beq
\tau_3 := \inf \{ t: \exists i : |z_i (t, 1) | + | \lambda_i (t) | > R \}
\eeq
for some large $R>0$.  Finally define
\beq
\tau_4 := \inf \{ t : \exists i : |Z_i | > N^{ - 9 \fb} \}.
\eeq
We know that $\tau_1 \wedge \tau_2 \wedge \tau_4 \geq T$ with overwhelming probability and that $|\lambda_i (t) | \leq R$ with overwhelming probability.  We will see later (see Lemma \ref{lem:weakglobal}) that $|z_i (t, 1)  - z_i (t, 0) | \leq C$ with overwhelming probability.  By \cite{dbmrigid} the process $z_i (t, 0)$ stays bounded with overwhelming probability, and so 
\beq
\tau = T
\eeq
with overwhelming probability.  Note that for $ t \leq \tau$ we have $\gamma_{ij} = \gamhat_{ij}$.

For the rest of the proof we denote $z_i = z_i (t, 1)$.  Define
\beq
u_i := \lambda_i - z_i
\eeq
For $t \leq \tau$, this satisfies the equation
\beq
\d u_i = \sum_j B_{ij} (u_j - u_i ) \d t + \d M_i + Z_i \d t + \frac{A_N}{\sqrt{N}} \d B_i
\eeq
where 
\beq
B_{ij} = \frac{ 1 - \gamhat_{ij}}{( \lambda_i  - \lambda_j ) ( z_i - z_j ) } > 0
\eeq
and
\beq
A_N = \frac{1}{ \sqrt{ 1 + N^{ - c_2}}} - 1 = \O ( N^{ - c_2 } ).
\eeq
We fix a $c_4 >0$ to be chosen and let $\lambda := N^{1 + c_4 }$.  Define
\beq
F (t) := \frac{1}{ \lambda } \log \left( \sum_i \e^{ \lambda u_i } \right).
\eeq
By the It{\^o} lemma for $t \leq \tau$ we may calculate
\begin{align}
\d F (t) &= \frac{1}{ \sum_i \e^{ \lambda u_i } } \sum_i \e^{ \lambda u_i } \sum_j B_{ij} (u_j - u_i ) \label{eqn:mdiff} \\
&+ \frac{1}{ \sum_i \e^{ \lambda u_i } } \sum_i \e^{ \lambda u_i } \left(  \d M_i + Z_i \d t + N^{-1/2}  A_N \d B_i \right) \label{eqn:mmart} \\
&+ \lambda \frac{1}{ \sum_i \e^{ \lambda u_i } } \sum_i \e^{ \lambda u_i } \d \langle M_i +N^{-1/2} A_N B_i , M_i +N^{-1/2} A_N B_i \rangle \label{eqn:mito1} \\
&+ \lambda \frac{1}{ \left( \sum_i \e^{ \lambda u_i } \right)^2 } \sum_{i,j} \e^{ \lambda u_i } \e^{ \lambda u_j } \d \langle M_i + N^{-1/2} A_N B_i, M_j  + N^{-1/2} A_N B_j \rangle \label{eqn:mito2}
\end{align}
The first observation is that the term \eqref{eqn:mdiff} is negative,
\beq
\sum_i \e^{ \lambda u_i } \sum_j B_{ij} (u_j - u_i ) = \frac{1}{2} \sum_{ij} B_{ij} (u_j - u_i ) ( \e^{ \lambda u_i} - \e^{ \lambda u_j } ) \leq 0
\eeq
because $x \to \e^{ \lambda x}$ is an increasing function.  

We first bound \eqref{eqn:mmart}.  By definition of $\tau$,
\beq
\sup_{0 \leq t \leq \tau } \int_0^{t} \sum_i \frac{ \e^{ \lambda u_i }}{ \sum_k \e^{ \lambda u_k } } |Z_i | \d s \leq N^{ - 9 \fb } \int_0^{T} \sum_i \frac{ \e^{ \lambda u_i }}{ \sum_k \e^{ \lambda u_k } }  \leq T N^{ - 9 \fb}.
\eeq
We next calculate some quadratic variations.  We have for $t \leq \tau$,
\beq
d \langle M_i, M_j \rangle = \frac{1}{N} \sum_{ |a-b| \leq N^{ \fae} } |w_{ia} w_{bi}  w_{ja} w_{bj} | \leq \frac{N^{ \fae- 2 \frc_1}}{N}
\eeq
and
\beq
\frac{1}{N} \d \langle A_N B_i, A_N B_i \rangle \leq \frac{N^{ -2 c_2}}{N}.
\eeq
Hence we have by the BDG inequality,
\beq
\sup_{0 \leq t \leq \tau } \int_0^t \sum_i \frac{ \e^{ \lambda u_i } }{ \sum_k \e^{ \lambda u_k } } \left( \d M_i + N^{-1/2} A_N \d B_i \right) \leq \frac{ N^{\eps} T^{1/2}}{N^{1/2} } \left( N^{ \fae/2 - \frc_1 } + N^{ -c_2} \right)
\eeq
for any $\eps >0$ with overwhelming probability.  For the term \eqref{eqn:mito1} we expand out the covariation.  For the diagonal terms we obtain
\beq
\sup_{ 0 \leq t \leq \tau } \lambda \int_0^t \sum_i \frac{ \e^{ \lambda u_i }}{\sum_k \e^{ \lambda u_k }} \left( \d \langle M_i , M_i \rangle + \d \langle N^{-1/2} A_N B_i , N^{-1/2} A_N B_i \rangle \right) \leq \frac{ C \lambda T}{N} \left( N^{ \fae - 2 \frc_1 } + N^{ -2 c_2 } \right).
\eeq
For the off-diagonal terms we apply the Kunita-Watanabe inequality and obtain
\begin{align}
& \sup_{ 0 \leq t \leq \tau } \left| \lambda \int_0^t \sum_i \frac{ \e^{ \lambda u_i }}{\sum_k \e^{ \lambda u_k }}  \d \langle M_i , N^{-1/2} A_N B_i \rangle  \right| \notag\\
\leq &   \lambda \sum_i \left( \int_0^\tau \frac{ \e^{ \lambda u_i }}{ \sum_i \e^{ \lambda u_i } } \d \langle M_i, M_i \rangle \right)^{1/2} \left( \int_0^\tau \frac{ \e^{ \lambda u_i }}{ \sum_i \e^{ \lambda u_i } }  N^{-2 c_2 - 1}\d \langle B_i, B_i \rangle \right)^{1/2} \notag\\
\leq &\frac{ \lambda N^{ - c_2 + \fae/2-\frc_1 }}{N} \sum_i \left( \int_0^T  \frac{ \e^{ \lambda u_i }}{ \sum_i \e^{ \lambda u_i } } \d t \right) \leq \frac{ \lambda T N^{ - c_2 + \fae/2-\frc_1 }}{N}.
\end{align}
Applying the same argument for \eqref{eqn:mito2} we obtain for the diagonal terms,
\beq
\sup_{ 0 \leq t \leq \tau} \lambda \int_0^t \sum_{i, j} \frac{ \e^{ \lambda u_i } \e^{ \lambda u_j } }{ \left( \sum_k \e^{ \lambda u_k } \right)^2 } \left( \d \langle M_i , M_j \rangle + \d \langle B_i , B_j \rangle \right) \leq \frac{ C T \lambda}{N} \left( N^{ \fae-2 \frc_1 } + N^{ -2 c_2 } \right).
\eeq
Applying the Kunita-Watanabe inequality for the off-diagonal terms as above we obtain,
\begin{align}
& \sup_{ 0 \leq t \leq \tau }\left| \lambda  \int_0^t \sum_{i, j} \frac{ \e^{ \lambda u_i } \e^{ \lambda u_j } }{ \left( \sum_k \e^{ \lambda u_k } \right)^2} \d \langle M_i , A_N N^{-1/2} B_j \rangle \right| \notag\\
\leq & \lambda \sum_{i, j} \left( \int_0^{\tau} \frac{ \e^{ \lambda u_i } \e^{ \lambda u_j }}{ \left( \sum_k \e^{ \lambda u_k } \right)^2} \d \langle M_i , M_i \rangle  \right)^{1/2}\left( \int_0^{\tau} \frac{ \e^{ \lambda u_i } \e^{ \lambda u_j }}{ \left( \sum_k \e^{ \lambda u_k } \right)^2} A_N^2 N^{-1} \d \langle B_j , B_j \rangle  \right)^{1/2} \notag\\
\leq & \frac{ \lambda }{N} N^{ \fae/2 - \frc_1 - c_2 } \int_0^t \sum_{i, j} \frac{ \e^{ \lambda u_i } \e^{ \lambda u_j }}{ \left( \sum_k \e^{ \lambda u_k } \right)^2} \d s  \leq \frac{ \lambda T N^{ \fae/2 - \frc_1 - c_2 } }{N}.
\end{align}
We have therefore derived that
\beq
\sup_{ 0 \leq s \leq \tau} F (s) \leq F(0) + T N^{ -9 \fb} +  \frac{ N^{\eps} T^{1/2}}{N^{1/2} } \left( N^{ \fae/2 - \frc_1 } + N^{ -c_2} \right) + \frac{ C T \lambda}{N} \left( N^{ \fae-2 \frc_1 } + N^{ -2 c_2 } \right)
\eeq
Note that
\beq
F(0) = \frac{ \log(N) } { N^{1 + c_4 } }
\eeq
and
\beq
F(s) \geq \sup_{i} u_i (s).
\eeq
Hence we obtain the upper bound of \eqref{eqn:lamz}.  The same argument applies to $-u_i$ and so we obtain \eqref{eqn:lamz}. \qed

\subsection{Interpolating processes}
Consider the processes $z_i (t, \alpha)$ defined above.  It is not hard to see that the map $\alpha \to z_i (t, \alpha )$ is a Lipschitz function and so one can demand that almost surely, a solution exists for all $\alpha \in [0, 1]$.  Once this has been established it is easy to check that $z_i (t, \alpha )$ is in fact a differentiable function of $\alpha$.  The derivative $u_i := \del_\alpha z_i (t, \alpha )$ satisfies the equation
\beq
\del_t u_i = \sum_{j} B_{ij} (u_j - u_i ) + \sum_j \frac{ \hatgam_{ij} }{ z_i (t, \alpha ) - z_j (t, \alpha ) }:= - ( B u  )_i + \xi_i,
\eeq
where
\beq
B_{ij} := \frac{1}{N} \frac{ 1 - \hatgam_{ij}}{ ( z_i (t, \alpha ) - z_j (t, \alpha ) )^2}.
\eeq

We also pause here to introduce some notation.  The inner product on $\ell^2$ is (we will only have to consider real sequences)
\beq
\langle u, v \rangle = \frac{1}{N} \sum_{i=1}^N u_i v_i.
\eeq
This notation clashes with the covariation of martingales, but we will not need to calculate any more covariations in Section \ref{sec: DBM}. 
The $\ell^p$ norms are
\beq
||u||_p^p := \frac{1}{N} \sum_i |u_i|^p, \qquad || u||_\infty := \sup_{1 \leq i \leq N } |u_i|.
\eeq

\subsection{Weak global a-priori estimate}
We first derive a weak global estimate on the processes $z_i (t, \alpha)$.
\bel \label{lem:weakglobal}
With overwhelming probability we have,
\beq \label{eqn:weakglobal}
\sup_{ 0 \leq s \leq T } \sup_i \sup_{ 0 \leq \alpha \leq 1 } | z_i (0, \alpha ) - z_i (s, 0 ) | \leq  C N^{ \fb/2 - \frc_1}.
\eeq
\eel
\proof  We differentiate the $\ell^2$ norm and obtain
\begin{align}
\del_t \frac{1}{N} \sum_i u_i^2 &= \frac{1}{N} \sum_i  u_i B_{ij} (u_j - u_i )  + \frac{1}{N^2} \sum_{i, j} \frac{ u_i \hatgam_{ij}}{ z_i - z_j }  \notag\\
&= - \frac{1}{ 2 N^2} \sum_{ij} \frac{ (1 - \hatgam_{ij} ) (u_i - u_j ) }{ ( z_i -z_j )^2 } + \frac{1}{ 2 N^2 } \sum_{i, j} \frac{ (u_i - u_j ) \gamma_{ij}}{ z_i - z_j }
\end{align}
Above we used the symmetry $\hatgam_{ij} = \hatgam_{ji}$.  By Cauchy-Schwartz and the fact that $\hatgam_{ij} \ll 1$, we can bound this by
\beq
\del_t ||u||_2^2 \leq - \frac{1}{4} \langle u, B u \rangle + \frac{10}{ N^2} \sum_{i, j} | \hatgam_{ij } |^2 \leq C N^{ - 2 \frc_1}.
\eeq
Since $||u (0) ||_2 = 0$ this yields the claim. \qed

\subsection{Local law for $\alpha=0$ process}
Let $m_0 (z)$ be
\beq
m_0 (z) := \frac{1}{N} \sum_i \frac{1}{ \lambda_i (0) - z }.
\eeq
Define $m_t (z)$ to be the free convolution of $m_0 (z)$ with the semicircle law at time $t$, i.e., $m_t(z)$ is the unique solution to
\beq
m_t (z) = m_0 (z + t m_t (z) ) 
\eeq
vanishing as $|z| \to \infty$.  Then $m_t(z)$ is the Stieltjes transform of a measure with density $\rho_t (E)$.  By Theorem \ref{thm: bulk law} we have
\beq
c \leq \Im [ m_0 (E + \i \eta ) ] \leq C, \qquad E \in I_\kappa , \quad N^{\nu}/N \leq \eta \leq 10
\eeq
for any $\nu >0$ and $\kappa >0$.  Since $|m_t (z) |^2 \leq 1/t $ (see \cite{schnelli2}) we see that
\beq
c \leq \Im [ m_t (E + \i \eta ) ] \leq C , \qquad E \in I_\kappa , \quad N^{\nu}/N \leq \eta \leq 10, \quad 0 \leq t \leq T
\eeq
for any $\nu >0$ and $\kappa >0$.  Define the classical eigenvalue locations of the free convolution $\rho_t$ by $\gamma_i (t)$.  They satisfy
\beq
\del_t \gamma_i (t) = - \Re [ m_t ( \gamma_i (t) ) ]
\eeq
and since $|m_t| \leq t^{-1/2}$ we see that with overwhelming probability (also using \eqref{eqn:initrig})
\beq
i \in \Gkap \implies \gamma_i (t) \in I_{\kappa/2}, \qquad 0 \leq t \leq T.
\eeq
Therefore, for any $\nu >0$ and $i, j \in \Gkap$ we have
\beq \label{eqn:gamijspace}
c \frac{ |i-j|}{N} \leq | \gamma_i (t) - \gamma_j (t) | \leq C \frac{ |i-j|}{N}, \qquad N^{\nu} \leq |i-j|.
\eeq

For $\alpha =0$ the process $z_i (t, 0)$ satisfies the equation
\beq
\d z_i (t, 0) = \frac{ \d B_i}{ \sqrt{ 1 + N^{ - c_2 } } } + \frac{1}{N} \sum_j \frac{1}{ z_i (t, 0) - z_j (t, 0) } \d t.
\eeq
Therefore, Corollary 3.2 of \cite{dbmrigid} implies that
\beq \label{eqn:alpha0r}
\sup_{ 0 \leq t \leq T } \sup_{ i \in \Gkap} | z_i (0, t) - \gamma_i (t) | \leq \frac{ N^{\nu}}{N}
\eeq
with overwhelming probability for any $\nu >0$.

We remark that while the $\gamma_i (t)$ are random, we will essentially only be using the deterministic property \eqref{eqn:gamijspace} of the classical eigenvalue locations to derive the same property of the $z_i (\alpha, t)$.

\subsection{Local law for interpolating processes}

We define the empirical Stieltjes transform of the interpolating processes by
\beq
m_N (z, t, \alpha ) := \frac{1}{N} \sum_i \frac{1}{ z_i (t, \alpha ) - z }.
\eeq
It satisfies the equation
\begin{align}
\d m_N &= ( m_N  ) \del_z m_N \d t+ \d M_t +\frac{2-\beta_N}{N^2 \beta_N} \sum_i \frac{1}{ (z_i - z)^3} \d t  \notag\\
&+ \frac{1}{N^2} \sum_{i \neq j} \frac{  \alpha \gamhat_{ij}}{ (z_i - z )^2 (z_j - z)} \d t
\end{align}
where $M$ is a martingale term and $\beta_N = 2 ( 1 + N^{ -\frc_1} )$ (the value of $\beta_N$ is of no real importanc here).   The only difference between this and the equation appearing in \cite{dbmrigid} is the error term on the last line.   By Corollary \ref{cor: wbd} we have that for every pair of indices $(i, j)$ with either $i \in \G_{\kappa}$ or $j \in \G_{\kappa}$, the estimate
\beq
| \gamhat_{ij} | \leq \frac{ N^{ \fae+\eps}}{N}
\eeq
for any $\eps>0$ holds with overwhelming probability.  By the weak global estimate \eqref{eqn:weakglobal} and the rigidity estimate \eqref{eqn:alpha0r} we see that for every energy $E \in I_\kappa$ we have for any $\eps >0$ and $0 < t < T$ that with overwhelming probability,
\beq
\left| \frac{1}{N^2} \sum_{i, j} \frac{  \alpha \gamma_{ij}}{ (z_i - z )^2 (z_j - z)} \right| \leq C N^{- \frc_1} + \frac{ N^{ \fae+\eps}}{ N \eta^2 } \Im [ m_N (E + \i \eta ) ].
\eeq
Using this estimate, one can modify, in a straightforward fashion, the methods of \cite{dbmrigid} to derive the estimate (as long as $\fb < \frc_1$),
\beq
| m_N (z, t, \alpha ) - m_{ t } ( z) | \leq \frac{ N^{\fae+\eps}}{ N \eta}
\eeq
with overwhelming probability in the region
\beq
\{ E + \i \eta : E \in I_{\kappa} ,  10 \geq \eta \geq N^{\delta+ \fae}/N \} \cup \{ E + \i \eta : |E| \leq C, 1/2 \leq \eta \leq 10 \},
\eeq
for any $C>0$ and $\delta >0$.  Here, $m_t$ is as in the last subsection.  Standard methods then give us the rigidity estimate
\beq \label{eqn:optrigid}
| z_i (t, \alpha ) - \gamma_i (t) | \leq \frac{ N^{ 5 \fae }}{N}
\eeq
with overwhelming probability, for $i \in \G_{\kappa}$.

\subsection{Short-range approximation}
In this section we introduce the short-range approximation $\hatz_i (t, \alpha)$.  
Fix a $\kappa_* >0$ and denote
\beq
\G_{\kappa_*} = [[g_-, g_+ ]].
\eeq
The parameter $\kappa_*$ will be fixed for the rest of Section \ref{sec: DBM}.
 Fix an $\ell = N^{\om_\ell}$.   We choose
 \beq
 \om_{\ell} > 5 \fae.
 \eeq
  Define the index set
\beq
\A := \{ (i, j) : |i - j | \leq \ell \} \cup \{ (i, j) : i > g_+ , j > g_+ \} \cup \{ (i, j) : i < g_-, j < g_- \}.
\eeq
Define the short-range approximation
\beq
\d \hatz_i (\alpha, t)= \sqrt{ \frac{2}{ 1 + N^{ - c_2 } } } \frac{ \d B_i }{ \sqrt{N}} + \frac{1}{N} \sum_{ j : (i, j) \in \A } \frac{ 1 - \alpha \hatgam_{ij} }{ \hatz_i  (\alpha, t)- \hatz_j  ( \alpha, t) } \d t + \frac{1}{N} \sum_{ j : (i, j) \notin \A } \frac{1}{ z_i (0, t) - z_j (0, t) } \d t.
\eeq
By the strong rigidity esimates \eqref{eqn:optrigid} and \eqref{eqn:alpha0r} and the weak global estimate \eqref{eqn:weakglobal} we can bound for every $i$,
\beq
 \left| \frac{1}{N} \sum_{ j : (i, j) \notin \A } \frac{1}{ z_i (0, t) - z_j (0, t) }  -  \frac{1}{N} \sum_{ j : (i, j) \notin \A } \frac{1- \hatgam_{ij}}{ z_i (\alpha, t) - z_j (\alpha, t) }  \right| \leq N^{\eps} N^{- \frc_1} + N^{\frb/2-\frc_1 } + N^{5 \fa - \om_\ell}.
\eeq
Hence, by the proof of Lemma 3.7 of \cite{fixed} we obtain the following estimate.
\bel \label{lem:sr}
WIth overwhelming probability we have
\beq
\sup_i \sup_{ 0 \leq s \leq T } |z_i (s, \alpha ) - \hatz_i (s, \alpha ) | \leq \frac{N^{\eps}}{N}\left(  N^{ 3 \frb/2- \frc_1} + N^{ \fb + 5 \fae - \om_\ell} \right).
\eeq
for any $\eps >0$.
\eel
Note that this implies that the weak global estimate \eqref{eqn:weakglobal} and the rigidity estimate \eqref{eqn:optrigid} hold with $z_i (t, \alpha)$ replaced by $\hatz_i (t, \alpha)$, as long as $3 \fb/2 < \frc_1$ and $\frb + 5 \fae < \om_\ell$.

\subsection{Weak level repulsion estimates}
We will require the following weak level repulsion estimate which will allow us to make a cut-off later.
\bel \label{lem:lr}
With overwhelming probability we have,
\beq
\sup_{ i \neq j } \int_{0}^{T} \frac{1}{ |\hatz_i (t, \alpha ) - \hatz_{j} (t, \alpha ) | } \leq N^{10}.
\eeq
\eel
\proof For any $k$ we calculate
\begin{align}
\d \left( \sum_{i=1}^k \hatz_i  (t, \alpha ) \right) &= \sum_{i=1}^k \sqrt{ \frac{1}{1 + N^{-c_2} } } \frac{ \d B_i}{ \sqrt{N}} + \sum_{i=1}^k \sum_{ j : (i, j) \notin \A } \frac{1}{ z_i (t, 0) - z_j (t, 0) } \notag\\
& +\sum_{i=1}^k \sum_{j=k+1 , (i, j) \in \A }^N \frac{1 - \alpha \gamma_{ij}}{ \hatz_{i} (t, \alpha ) - \hatz_{j} (t, \alpha )}. \label{eqn:samesign} 
\end{align}
Note that every term in the second line \eqref{eqn:samesign} has the same sign and $ 1 - \alpha \hatgam_{ij} \geq 1/2$.  Hence we get the inequality
\begin{align}
 & \frac{1}{2} \int_{0}^{T} \frac{1}{ \hatz_{k+1} (t, \alpha ) - \hatz_{k} (t, \alpha ) }  \notag\\
\leq & \left| \sum_{i=1}^k \int_{0}^{t_0} \hatz_{i} (s, \alpha ) \d s - \sum_{j : (i, j) \notin \A} \frac{1}{z_i (s, 0) - z_j (t, 0 ) } \d s - \sqrt{ \frac{1}{1 + N^{-c_2} } } \frac{ \d B_i}{ \sqrt{N}}  \right| \notag\\
\leq & N^{10}
\end{align}
with overwhelming probability.
 \qed

\subsection{Cut-off of long-range terms}

Define now $\hatu := \del_\alpha \hatz$.  Then $\hatu$ satisfies the equation
\beq
\del_t \hatu_i = \sum_{j :(i, j) \in \A } \hatB_{ij} (\hatu_j - \hatu_i ) + \hatxi_i 
\eeq
where 
\beq
\hatB_{ij} = \frac{1}{N} \frac{1 - \alpha \hatgam_{ij}}{ ( \hatz_i (t, \alpha ) - \hatz_j (t, \alpha ))^2}
\eeq
and
\beq
\hatxi_i := \sum_{ j : (i, j ) \in \A } \frac{ \hatgam_{ij}}{\hatz_i (t, \alpha ) - \hatz_j (t, \alpha ) } .
\eeq
Define $v$ by $v_i (0) = 0$ and
\beq
\del_t v = - (\hatB v)_i + \1_{ \{ i \in \G_{ 2 \kappa_* }  \}} \zeta_i.
\eeq
where
\beq
\zeta_i = \frac{1}{N} \sum_{ j : (i, j ) \in \A } \1_{ \{ j \in \G_{ 2 \kappa_* } \} } \frac{ \hatgam_{ij}}{ \hatz_i - \hatz_j }.
\eeq
Note that by the choice of $\A$ we have $\zeta_i = \hatxi_i$ for $i \in \G_{3 \kappa_*}$.  The purpose of $v$ is to cut off error terms from the $\hatxi$ for which we do not have good estimates on $\hatgam_{ij}$.  The choice of $\zeta_i$ is motivated by a symmeterization in the summation indices $i, j$ later. 
By the Duhamel formula the difference satisfies
\beq
u_i (t) - v_i  (t)= \int_0^t \sum_{ j \notin \G_{ 3 \kappa_*  } } \UB_{ij} (s, t )  ( \hatxi_j (s)  - \zeta_j (s) )\d s,
\eeq
where we used that $\hatxi_i = \zeta_i$  for $i \in \G_{3 \kappa_*}$.
We assume $\fb < \om_\ell$.  The proof of Theorem 4.1 of \cite{fixed} implies that for each fixed $\alpha$ we have with overwhelming probability,
\beq
\sup_{ 0 \leq s \leq t \leq T } \UB_{ij} (s, t) \leq N^{-D}
\eeq
for any $D >0$ as long as $i \in \G_{4 \kappa_*}$ and $ j \notin \G_{ 3 \kappa_*}$, and as long as $5 \fae < 1/10$.  This estimate together with Lemma \ref{lem:lr} implies the following.
\bel
For every $\alpha $ there is an event that holds with overwhelming probability on which
\beq
\sup_{ i \in \G_{4 \kappa_*}} \sup_{ 0 \leq s \leq t_0 } | v_i (s) - \hatu_i (s) | \leq N^{-10}.
\eeq
\eel
Due to the fact that we will later apply Markov's inequality, we also require the following auxilliary bound on $\hatu$.
\bel \label{lem:uhatbd}
We have that almost surely,
\beq
\sup_{i} \sup_{ 0 \leq t \leq T } \sup_{ 0 \leq \alpha \leq 1 } | \hatu_i (t, \alpha ) | \leq C
\eeq
\eel
\proof This is a simple $\ell^2$ calculation. We have
\beq
\del_t  \frac{1}{N} \sum_i \hatu_i = \frac{-1}{2} \langle \hatu, B \hatu \rangle + \frac{1}{N^2} \sum_{ (i, j) \in \A } \frac{ \hatu_i \hatgam_{ij} }{ \hatz_i - \hatz_j }.
\eeq
The second term is bounded by
\begin{align}
\frac{1}{N^2} \sum_{ (i, j) \in \A } \frac{ \hatu_i \hatgam_{ij} }{ \hatz_i - \hatz_j } &= \frac{1}{2 N^2} \sum_{ (i, j) \in \A } \frac{ (\hatu_i - \hatu_j ) \hatgam_{ij} }{ \hatz_i - \hatz_j } \notag\\
&\leq \frac{1}{10} \langle \hatu, B \hatu \rangle + \frac{C}{ N^2} \sum_{i, j}  | \hatgam_{ij} |^2  \notag\\
&\leq \frac{1}{10} \langle \hatu, B \hatu \rangle + C N^{ - 2 \frc_1 },
\end{align}
and therefore by Gronwall,
\beq
\sup_{ 0 \leq t \leq T } || \hatu (t) ||_2^2 \leq  C T N^{ - 2 \frc_1}.
\eeq
The claim follows. \qed

\subsection{Energy method}
Finally we estimate $v$ using the energy method.  This is the main calculation of the current section.
\bel \label{lem:vbd} 
We have with overwhelming probability,
\beq
\sup_i \sup_{0 \leq s \leq T } |v_i (s) | \leq \frac{N^{\eps}}{N} \left( N^{ \fb /2 + 7/2\fae - 1/2 } + N^{ \fb +  \fae  - 1/2 } \right).
\eeq
for any $\eps >0$.
\eel
\proof Define
\beq
\A_1 := \A \cap \{ (i, j) : i \in \G_{2 k_*} \mbox{ and } j \in \G_{2 k_* } \}.
\eeq
We differentiate the $\ell^2$ norm and find
\beq
\del_t \frac{1}{N} \sum_{i} v_i^2  =  - \frac{1}{2}  \langle v , \hatB v \rangle +  \frac{1}{N^2} \sum_{(i, j ) \in \A_1 } \frac{ v_i  \hatgam_{ij}}{ \hatz_i - \hatz_j }.
\eeq
For $\eps >0$ we write the second term as
\begin{align}
\frac{1}{N^2} \sum_{(i, j ) \in \A_1 } \frac{ v_i  \hatgam_{ij}}{ \hatz_i - \hatz_j } &= \frac{1}{N^2} \sum_{(i, j ) \in \A_1 , |i-j| \leq N^{\eps + 5 \fae} } \frac{ v_i  \hatgam_{ij}}{ \hatz_i - \hatz_j }  \notag\\
&+\frac{1}{N^2} \sum_{(i, j ) \in \A_1, |i-j| > N^{\eps + 5 \fae} } \frac{ v_i  \hatgam_{ij}}{ \hatz_i - \hatz_j }
\end{align}
We can use rigidity and the estimate $\hatgam_{ij} \leq N^{\fa+\eps}/N$ to bound the second term by
\begin{align}
\left| \frac{1}{N^2} \sum_{(i, j ) \in \A_1, |i-j| > N^{\eps + 5 \fae} } \frac{ v_i  \gamma_{ij}}{ \hatz_i - \hatz_j }   \right| \leq &\frac{N^{ \fa +\eps}}{ N^2} \sum_i v_i \sum_j \frac{1}{ |i-j| } \notag\\
\leq & N^{\eps} \frac{ N^{ \fa + 2 \eps }}{N^2} \sum_i v_i \notag\\
\leq & \frac{1}{T} ||v||_2^2 + C N^{4 \eps} T  \frac{ N^{ 2 \fa}}{N^2}.
\end{align}
We absorb the first term into the negative energy term.
\begin{align}
&\frac{1}{N^2} \sum_{(i, j ) \in \A_1 , |i-j| \leq N^{\eps + 5 \fae} } \frac{ v_i  \hatgam_{ij}}{ \hatz_i - \hatz_j }  \notag\\
 = &\frac{1}{N^2} \sum_{(i, j ) \in \A_1 , |i-j| \leq N^{\eps + 5 \fae} }   \frac{ \hatgam_{ij} (v_i -v_j )}{ (z_i -z_j )} \notag\\
 \leq & \frac{1}{10}   \langle v, \hatB v \rangle + \frac{C}{N^2} \sum_{ (i, j) \in \A_1 , |i-j| \leq N^{ \eps + 5 \fae} } |\hatgam_{ij} |^2 \leq \frac{1}{10} \langle v, \hatB v \rangle + C \frac{ N^{ \eps + 7 \fae}}{N^3}.
\end{align}
By Gronwall's inequality we deduce
\beq
\sup_{0 \leq s \leq T } || v (s) ||_2^2 \leq \frac{N^{4 \eps}}{N^4} \left( N^{ \fb + 7\fa } +N^{ 2 \fb + 2 \fa } \right).
\eeq
Hence, after redefining $\eps$, we get
\beq
\sup_{ 0 \leq s \leq T} || v (s) ||_\infty \leq \frac{N^{\eps}}{N} \left( N^{ \fb /2 + 7/2\fae - 1/2 } + N^{ \fb +  \fae  - 1/2 } \right).
\eeq \qed

\subsection{Proof of Theorem \ref{thm:dbmcompare}}
We write
\begin{align}
\lambda_i (t) - \mu_i (t) &= ( \lambda_i (t) - z_i (t, 1) ) + ( z_i (t, 0) - \mu_i (t) )  + (z_i (t, 1) - z_i (t, 0) ).
\end{align}
By Lemma \ref{lem:lamz}, we have for $0 \leq t \leq T$,
\beq
| \lambda_i (t) - z_i (t, 1) | + | z_i (t, 0) - \mu_i (t) | \leq \frac{N^{\eps}}{N} \left( N^{ - \frc_1/4} + N^{ - 8\fb} \right)
\eeq
after choosing $c_2 = \frc_1 /2$ and $c_4 = \frc_1/4$.  

Applying Lemma \ref{lem:sr} we have
\beq
z_i (t, 1) - z_i (t, 0) = \hatz_i (t, 1) - \hatz_i (t, 0) + \frac{N^{\eps}}{N}  \left( N^{ - \frc_1/4} + N^{ - \fae} \right)
\eeq
after taking, say, $\om_\ell = 10 \fae$.  

We now write, recalling the notation $\hatu$ and $v$ of the  previous section
\begin{align}
\hatz_i (t, 1) - \hatz_i (t, 0) &= \int_0^1 \del_{\alpha} \hatz_i (t, \alpha ) \d \alpha \notag\\ 
&= \int_0^1 \hatu_i (t, \alpha ) \d \alpha \notag\\
&= \int_0^1 \hatu_i (t, \alpha ) \1_{ \F ( \alpha ) } \d \alpha + \int_0^1 \hatu_i (t, \alpha ) \1_{ \F^c (\alpha ) } \d \alpha
\end{align}
where $\F ( \alpha)$ is the event of Lemma \ref{lem:sr}.  Since $\hatu_i (t, \alpha )$ is bounded a.s. we have
\beq
\left| \int_0^1 \hatu_i (t, \alpha ) \1_{ \F^c (\alpha ) } \d \alpha \right| \leq N^{ - 100}
\eeq
with overwhelming probability, by Markov's inequality.

By the definition of $\F ( \alpha )$, we have for $i \in \G_{4 \kappa_* }$,
\begin{align}
 \int_0^1 \hatu_i (t, \alpha ) \1_{ \F ( \alpha ) } \d \alpha &=  \int_0^1 v_i (t, \alpha ) \1_{ \F ( \alpha ) } \d \alpha +  \O \left( N^{-10} \right)
\end{align}
Finally, by Lemma \ref{lem:vbd} we have
\beq
\left|  \int_0^1 v_i (t, \alpha ) \1_{ \F ( \alpha ) } \d \alpha \right| \leq N^{-5/4}
\eeq
with overwhelming probability.  This completes the proof. \qed

\subsection{Proof of Theorem \ref{thm:bu1}}
In this section we prove universality of the local statistics of $H$.  For simplicity we only do the $2$-point function.  Higher $k$-point functions and gap statistics are similar.  Let $\lambda_i (t)$ be as defined above.  We need to calculate the quantity
\beq
\sum_{i, j} \mathbb{E} [ O (  N (\lambda_i (T) - E),  N (\lambda_j (T) - E )) ]
\eeq
for smooth compactly supported $O$ and $E \in I_{\kappa}$ for some fixed $\kappa >0$.  We will eventually apply Theorem \ref{thm:dbmcompare} to replace this with an expectation over $\mu_i (T)$.  We first do some preliminary calculations.  The main result of \cite{fixed} implies that with overwhelming probability over the initial data $\lambda_i (0)$ we have 
\begin{align}
&\bigg| \sum_{i, j} \mathbb{E} [ O (  \rho_T (E) N (\mu_i (T) - E),  \rho_T (E) N (\mu_j (T) - E ))  | \{ \lambda_k (0) \}_k] \notag\\
&- \sum_{i, j} \bbE^{(GOE)} [ O ( \rhosc (E') N ( \mu_i - E' ), \rhosc (E') N ( \mu_i - E' ) ) ] \bigg| \leq N^{ - c}
\end{align}
for some $c> 0$ and any $E' \in (-2, 2)$.  The quantity $\rho_T (E)$ is random and we want to replace it by the deterministic quantity $\rho (E)$ which denotes the density of the free convolution  $\mu_1 \boxplus \mu_2$.   First note that since $ | \del_z m_T (z) | \leq C / T$ we have
\beq
| \rho_T (E) - \Im [ m_{T} (E + \i N^{\fb/2-1} ) ] | \leq  C N^{ - \fb/2}.
\eeq
By Corollary \ref{cor:initlaw} we have the estimate for any $\nu >0$,
\beq
\sup_{ E \in I_{\kappa} , N^{ \nu-1} \eta \leq 10 }  \left| m_0 (z ) - m (z) \right| \leq N^{ - c_3}
\eeq
for some $c_3  >0$, with overwhelming probability.   With this estimate, the proof of Lemma 3.6  in \cite{schnelli2} implies that
\beq
\sup_{ E \in I_{\kappa} , N^{ \nu-1} \eta \leq 10 }  \left| m_T (z ) - m_3 (z) \right| \leq N^{ - c_3/2}
\eeq
where $m_3(z)$ is the free convolution of $m (z)$ with the semicircle distribution at time $T$.   Since the density of $m(z)$ is analytic on its support which contains $I_{\kappa/2}$ it is not hard to see that
\beq
\sup_{ E \in I_{\kappa} , 0 \leq \eta \leq 10 } | m_3 (z) - m (z) | \leq N^{ - c_5}
\eeq
and then that
\beq
| \Im [ m (E + \i N^{ \fb/2 -1 } ) - \rho (E) | \leq N^{ -c_5}
\eeq
for some $c_5 >0$.  Hence,
\beq \label{eqn:rhodif}
| \rho_T (E) - \rho (E) | \leq N^{ - c_6}
\eeq
for some $c_6 >0$.   

Now let $i_0$ be the (random) index s.t. $\gamma_{i_0} (T)$ is closest to $E$.  Note that this is measureable wrt $\{\lambda_k (0) \}_k$.  The estimate \eqref{eqn:alpha0r} also holds for $\mu_i$ due to Lemma \ref{lem:lamz}.   We have with overwhelming probability over the $\lambda_i (0)$ that
\begin{align}
 & \sum_{i, j} \mathbb{E} [ O (  \rho_T (E) N (\mu_i (T) - E),  \rho_T (E) N (\mu_j (T) - E ))  | \{ \lambda_k (0) \}_k]  \notag\\
 = & \sum_{|i-i_0| + |j-i_0| \leq N^{\nu} } \mathbb{E} [ O (  \rho_T (E) N (\mu_i (T) - E),  \rho_T (E) N (\mu_j (T) - E ))  | \{ \lambda_k (0) \}_k] + N^{ -10}
\end{align}
for any $\nu >0$.

By \eqref{eqn:rhodif} we have
\begin{align}
 &\sum_{|i-i_0| + |j-i_0| \leq N^{\nu} } \mathbb{E} [ O (  \rho_T (E) N (\mu_i (T) - E),  \rho_T (E) N (\mu_j (T) - E ))  | \{ \lambda_k (0) \}_k] \notag\\
 = & \sum_{|i-i_0| + |j-i_0| \leq N^{\nu} } \mathbb{E} [ O (  \rho (E) N (\mu_i (T) - E),  \rho(E) N (\mu_j (T) - E ))  | \{ \lambda_k (0) \}_k]  + \O ( N^{  - c_6/2} ) 
\end{align}
provided $\nu < c_6 / 4$.  Applying Theorem \ref{thm:dbmcompare} we have,
\begin{align}
  & \sum_{|i-i_0| + |j-i_0| \leq N^{\nu} } \mathbb{E} [ O (  \rho (E) N (\mu_i (T) - E),  \rho (E) N (\mu_j (T) - E ))  | \{ \lambda_k (0) \}_k]  \notag\\
 = &  \sum_{|i-i_0| + |j-i_0| \leq N^{\nu} } \mathbb{E} [ O (  \rho (E) N (\lambda_i (T) - E),  \rho (E) N (\lambda_j (T) - E ))  | \{ \lambda_k (0) \}_k] + \O ( N^{ - c_7 } )
\end{align}
for some $c_7 >0$ as long as we take $\nu$ small enough.  Lastly since Theorem \ref{thm:dbmcompare} implies that the estimate \eqref{eqn:alpha0r} also holds for $\lambda_i (T)$ we see that
\begin{align}
& \sum_{|i-i_0| + |j-i_0| \leq N^{\nu} } \mathbb{E} [ O (  \rho (E) N (\lambda_i (T) - E),  \rho(E) N (\lambda_j (T) - E ))  | \{ \lambda_k (0) \}_k]  \notag\\
 = & \sum_{i, j } \mathbb{E} [ O (  \rho_T (E) N (\lambda_i (T) - E),  \rho_T (E) N (\lambda_j (T) - E ))  | \{ \lambda_k (0) \}_k] +\O ( N^{-100} )
\end{align}
with overwhelming probability.  Bulk universality follows after taking the expectation over $\lambda_i (0)$. \qed

\section{Well-posedness}\label{sec: wpness}


In this section, we  prove that the eigenvalues $\la(t)=(\la_1(t),\cdots,\la_N(t))$ of $\tilde{H}(t)$ are a strong solution to \eqref{eq: dla}. Thoughout this section, we fix a positive integer $N\in\N$.   We are only trying to establish well-posedness of the SDE \eqref{eq: dla} and so $N$-dependence of constants will play no role in our calculations; we are not trying to establish asymptoptic results.  As a result, all the constants appearing in this section may depend on $N$, but we omit the dependence for simplicity of notation.  The generic constant $c$ may change from line to line, but will only change finitely many times and therefore will remain finite.

Recall that in Section \ref{subsec: formal derivation}, we defined $U(t)$ to be the unique strong solution to the following SDE:
\be\label{eq: du a}
\dd U(t) = \ii \dd W(t) U(t) -\tfrac{1}{2}A U(t)\dd t\,,\quad U(0)=I\,.
\ee
Here $W(t)=(W_{\alpha\beta}(t)) = \left(\tfrac{\ind{\abs{\alpha-\beta}< N^\mathfrak{a}}}{\abs{y_\alpha-y_\beta}} B_{\alpha\beta}(t)\right)$, where $(B_{\alpha\beta}(t))$ is a Hermitian Brownian motion in the sense of Definition \ref{def: HBM}. For any constant $0<\mathfrak{b}<\mathfrak{a}$ and $T=N^{-1+\mathfrak{b}}$ we defined 
\ben\begin{split}
	H(t)&:=V^*XV+U(t)^*Y U(t)\,,\\
	\tilde{H}(t)&:= H(t)+ (T-t)\hat{Y}\,.
	\end{split}
\ee 
See \eqref{def: haty} for the definition of $\hat{Y}$.  For the reader's convenience, we restate equation \eqref{eq: dla} as follows
\be\label{eq: dla a}
\begin{split}
	\dd \la_i = &\frac{1}{\sqrt{N}}\dd B_i -\frac{1}{\sqrt{N}}\sum_{\abs{\alpha-\beta}<N^\mathfrak{a}}w_{\alpha i}^* \dd B_{\alpha\beta} w_{\beta i}\\ &+\frac{1}{N}\sum_{j\neq i} \frac{(1-\ga_{ij})\dd t}{\la_i-\la_j}+ \langle a_i, (U^*\hat{Y}U-\hat{Y}) a_i \rangle \dd t.
\end{split}
\ee
Here $\dd B_i := \sum_{\alpha,\beta}w_{\alpha i}^* \dd B_{\alpha\beta} w_{\beta i}$.  The vector $a_i(t)=(a_{\alpha i}(t))_{1\leq \alpha \leq N}$ is defined to be the eigenvector associated to the $i$-th smallest eigenvalue of $\tilde{H}(t)$, such that $a_k$'s first non-zero component is non-negative. We formally defined
\be
	w_i=(w_{\beta i})_{1\leq \beta \leq N} := \left(\sum_\alpha U_{\beta\alpha} a_{\alpha i}\right)_{1\leq \beta\leq N}\,,\quad \ga_{ij}:= \sum_{ \abs{\alpha-\beta}<N^\mathfrak{a}} \abs{w_{\alpha i}}^2 \abs{w_{\beta j}}^2\,.
\ee
Note that $(a_k)_{1\leq k\leq N}$, $(w_{\beta k})_{1\leq \beta,k\leq N}$  and $(\ga_{ij})_{1\leq i,j\leq N}$ are not well-defined if $\tilde{H}(t)$ has repeated eigenvalues. However, this does not cause any trouble because Proposition \ref{prop: tdX} says that for any fixed $t\geq 0$, the eigenvalues of $\tilde{H}(t)$ are almost surely distinct.  Hence, by Fubini's theorem we know that almost surely, $(w_{\beta k}(t))_{1\leq \beta,k\leq N}$, $(a_k(t))_{1\leq k\leq N}$ and $(\ga_{ij}(t))_{1\leq i,j\leq N}$ are well-defined for almost all $t\geq 0$.  Therefore, without loss of generality, we simply set them equal to $0$ where they are not well-defined.

For brevity of notation we define a semi-martingale $M(t)=(M_i)_{1\leq i\leq N}$ by $M(0)=0$ and 
\be
	\dd M_i(t) = \frac{1}{\sqrt{N}}\dd B_i -\frac{1}{\sqrt{N}}\sum_{\abs{\alpha-\beta}<N^\mathfrak{a}}w_{\alpha i}^* \dd B_{\alpha\beta} w_{\beta i}+\langle a_i, (U^*\hat{Y}U-\hat{Y}) a_i \rangle \dd t\,.
\ee
We define a potential function on the domain $\R_{<}^N:=\{ x=(x_1,\cdots,x_N)\in\R^N: x_1<\cdots<x_N\}$ through
\be
	V(x):= -\frac{1}{N} \sum_{i<j} (1-\ga_{ij})\log(x_j-x_i)\,.
\ee
We define $V(x)=+\infty$ for those $x$ on the boundary of $\R_<^N$. Hence equation \eqref{eq: dla a} can be written as
\be
	\dd \la = \dd M-\nabla V(\la)\dd t\,.
\ee
Before verifying that $\la$ is the solution to the above equation, we prove the following proposition on the continuity of $\la(t)$ and $M(t)$ with respect to $t$.  
\begin{prop}\label{prop: lacont}
	For any $\ga\in(0,1/2)$ and $\mathfrak{t}>0$, the processes $\la(t)$ and $M(t)$ are almost surely $\ga$-H{\"o}lder continuous on $[0,\mathfrak{t}]$ .	
\end{prop}
\begin{proof}
For any matrix $Q$, denote $\norm{Q}_\infty:= \max_{\alpha,\beta}\abs{Q_{\alpha\beta}}$.  By Weyl's inequality, for any $s,t\in[0,\mathfrak{t}]$, we have
	\ben
		\abs{\la(s)-\la(t)}_\infty \lesssim  \norm{\tilde{H}(s)-\tilde{H}(t)}_\infty\,.
	\ee
By definition, $\tilde{H}(t)$ is a Lipschitz function  of $U(t)$, and therefore
\be\label{eq: lip}
	\abs{\la(s)-\la(t)} \lesssim  \norm{U(s)-U(t)}_\infty\,.
\ee
Consider the $(\alpha,\beta)$-th entry of $U(t)$, which can be written as 
\ben	
	U_{\alpha\beta}(t)=\delta_{\alpha\beta}+\int_0^t \sigma(r)\dd w(r) +\int_0^t b(r)\dd r\,,
\ee
for some bounded, continuous adapted processes $\sigma(t)$, $b(t)$ and a standard complex Brownian motion $w(t)$.  The third term $\int_0^t b(r)\dd r$ is a Lipschitz function in $t$. By the time change formula (see Theorem 8.5.7 in \cite{Ok}) the martingale part, $\int_0^t \sigma(r)\dd w(r) $, can be written as a Brownian motion stopped at a certain stopping time.   To be more precise, let

\be	
	\theta(t):= \int_0^t \abs{\sigma(r)}^2 \dd r \leq C t \,,
\ee
which is a Lipschitz function in $t$. There exists a Brownian motion $\hat{W}$ with respect to a filtration $(\mathcal{G}_t)_{t\geq 0}$ such that, for each $t\geq 0$, $\theta(t)$ is a $\mathcal{G}$-stopping time and 
\ben
	\int_0^t\sigma(r)\dd w(r) = \hat{W}({\theta(t)}).
\ee
 According to L{\`e}vy's modulus of continuity theorem, 
 the Brownian motion $\hat{W}(t)$ satisfies
 \be
	 \lim_{h\to0}\sup_{0\leq t\leq T-h} \frac{\abs{\hat{W}_{t+h}-\hat{W}_{t}}}{\sqrt{2h\abs{\log h}}}=1\,,\text{ a.s.}
\ee
The martingale part of $U_{\alpha\beta}$, as the composite of $\hat{W}(\cdot)$ with a Lipschitz function $\theta(\cdot)$, satisfies
\be\label{ineq: levy}
	 \limsup_{h\to0}\sup_{0\leq t\leq T-h} \frac{\norm{U_{t+h}-U_{t}}_\infty}{\sqrt{2h\abs{\log h}}}\leq c\,,\text{ a.s.}
\ee 
Hence $U(t)$ is almost surely $\ga$-H{\"o}lder continuous for $\gamma < 1/2$.  By \eqref{eq: lip}, it follows that $\la(t)$ is almost surely $\ga$-H{\"o}lder continuous, for any $\ga\in(0,1/2)$.
 
The continuity of $M(t)$ follows from the same time-change argument above. We omit the details here.  
\end{proof}

We state and prove the main theorem of this section.  In the following, for any $\mathfrak{t}>0$, we denote the filtration generated by $\tilde{H}(0)$ and $B(t)$ by
\ben
	(\mathcal{F}_t)_{0\leq t \leq\mathfrak{t}}:=(\sigma(\tilde{H}(0),(B_s)_{0\leq s\leq t}))_{0\leq t\leq \mathfrak{t}}\,,
\ee
\begin{thm}\label{thm: strong}
	For any $\mathfrak{t}>0$, the eigenvalues $\la(t)=(\la_1(t),\cdots,\la_N(t))$ are the unique strong solution to equation \eqref{eq: dla a} on $[0,\mathfrak{t}]$, such that:
	\begin{enumerate}
		\item $\la(t)$ is adapted to the filtration $(\mathcal{F}_t)_{0\leq t \leq\mathfrak{t}}$\,.
		\item $\P\left[\la_1(t)<\cdots< \la_N(t)\,, \text{ for almost all }t\in[0,\mathfrak{t}]\,\right]=1$\,.
	\end{enumerate} 
\end{thm}

\begin{proof}
For simplicity of notation, we assume $\mathfrak{t}=1$.  The same argument applies for any $\mathfrak{t}>0$. Fixing a large number $n\geq N^2$, we divide the interval $[0,1]$ into $n$ parts of equal size:
\be
	[0,1]= \cup_{l=1}^n I_l\,,
\ee
where $I_l:= [(l-1)/n, l/n]$.  Let $\delta >0$ be an $n$-dependent parameter to be chosen. For each index $l$ and parameter $t$ we define the stopping time 
\be
	\tau_{l}(t):= \inf\{ t\in I_l: \exists 1\leq i< N,\text{ s.t. } \abs{\la_{i+1}-\la_{i}}\leq \delta\}\wedge t\,.
\ee
On the interval $[(l-1)/n,\tau_{l})$, the neighboring eigenvalues are separated from each other by a lower bound $\delta$, hence the eigenvalues are smooth functions of the matrix entries of $\tilde{H}(t)$ with bounded second derivatives.  Therefore, it is safe to apply It{\^o}'s formula to the eigenvalues and all  formal calculations are rigorous.  Therefore, we have
\be
	\la (\tau_{m})-\la(\tfrac{m-1}{n})= \int^{\tau_{m} }_{\frac{m-1}{n}} \dd M -\nabla V(\la)\dd t\,.
\ee
We add $\la (\tfrac{m}{n})-\la (\tau_{m})$ to both sides and set the time $t$ in the definition of each $\tau_m$ to be $m/n$,
\be\label{eq: diffla}
	\la (\tfrac{m}{n})-\la(\tfrac{m-1}{n})= \int^{\tau_{m}(\frac{m}{n}) }_{\frac{m-1}{n}} (\dd M -\nabla V(\la)\dd t )+\left(\la (\tfrac{m}{n})-\la(\tau_m(\tfrac{m}{n})) \right)\,.
\ee
Note that Proposition \ref{prop: tdX} implies the absolute integrability of $\nabla V(\la)$, i.e., 
\be
	\E \int_0^1 \abs{\nabla V(\la(t))} \dd t<+\infty\,.
\ee
 Also, $M$ is a semi-martingale with bounded quadratic variation. Therefore, the integral $\int_{\frac{m-1}{n}}^{\frac{m}{n}}(\dd M-\nabla V(\la(t))\dd t)$ is a.s. well-defined. We can subtract the integral $\int_{\frac{m-1}{n}}^{\frac{m}{n}}\left(\dd M-\nabla V(\la)\dd t\right)$ from both sides of \eqref{eq: diffla} to get,
\be\label{eq: diffla1}
	\begin{split}
&\la (\tfrac{m}{n})-\la(\tfrac{m-1}{n})-\int_{\frac{m-1}{n}}^{\frac{m}{n}}\left(\dd M-\nabla V(\la)\dd t\right)\\&= -\int_{\tau_m(\frac{m}{n}) }^{\frac{m}{n}} (\dd M -\nabla V(\la)\dd t) +\left(\la (\tfrac{m}{n})-\la(\tau_m(\tfrac{m}{n})) \right)\,.
\end{split}
\ee
Note that the right hand side vanishes if $\tau_m(\tfrac{m}{n}) = \tfrac{m}{n}$. Intuitively, because of the repulsive nature of eigenvalues, $\tau_m(\tfrac{m}{n}) = \tfrac{m}{n}$ holds with high probability, which should give some a-priori control over the amount of time that the eigenvalues spend close together. 
 In order to rigorously justify this heuristic,  we define a stochastic index set $\mathcal{I}\subset \{1,\cdots,n\}$ by
\be\label{def: I}
	\mathcal{I}:= \{ m: \tau_m(m/n)< m/n\}\,.
\ee
The meaning of $\mathcal{I}$ is that, for each $m\in \mathcal{I}$, the smallest gap between neighbouring eigenvalues hits $\delta$ 
on the time interval $[\tfrac{m-1}{n}, \tfrac{m}{n})$.  If we fix any $1\leq l\leq N$ and sum \eqref{eq: diffla1} over $m\leq l$, we obtain 
\be\label{eq: lal}
	\begin{split}
	&\la(\tfrac{l}{n})-\la(0)-\int_0^{\tfrac{l}{n}}(\dd M -\nabla V(\la)\dd t )\\=& -\sum_{m\in\mathcal{I},m\leq l} \int^{\frac{m}{n}}_{\tau_m(\tfrac{m}{n})}\dd M -\nabla V(\la)\dd t +\sum_{m\in\mathcal{I},m\leq l}\left(\la_i (\tfrac{m}{n})-\la_i(\tau_m(\tfrac{m}{n}))\right)\,.
	\end{split}
\ee
We want to show that the left hand side is always $0$, so denote
\be\label{def: mut}
	\mu(t):= \la(t)-\la(0)-\int_0^{t}(\dd M -\nabla V(\la)\dd s)\,.
\ee
 We shall prove that $\mu\equiv 0$ almost surely. Equation \eqref{eq: lal} yields
\be
	 \mu(\tfrac{l}{n})=\sum_{m\in\mathcal{I},m\leq l} \int^{\frac{m}{n}}_{\tau_m(\tfrac{m}{n})}\dd M -\nabla V(\la)\dd t +\sum_{m\in\mathcal{I},m\leq l}\left(\la_i (\tfrac{m}{n})-\la_i(\tau_l(\tfrac{m}{n}))\right)\,.
\ee
Therefore, taking the supremum over all $1\leq l\leq n$, 
\be
	\begin{split}
	\sup_{1\leq l\leq n}\abs{\mu(\tfrac{l}{n}) }\leq& \abs{\mathcal{I}}\sup_{\abs{s-t}\leq n^{-1}} \left[\abs{M(s)-M(t)} +\abs{\la(s)-\la(t)}\right] \\ &+ \sum_{m\in\mathcal{I}} \int_{I_m}\abs{\nabla V(\la)}\dd t\,.
	\end{split}
\ee
In view of the fact that $M$ and $\la$ are almost surely $\ga$-H{\"o}lder continuous for all $\ga<1/2$ (see Proposition \ref{prop: lacont}), we have, almost surely,
\be\label{eq: diffla2}
\begin{split}
	\sup_{1\leq l\leq n}\abs{\mu(\tfrac{l}{n}) }\leq& \abs{\mathcal{I}}  n^{-9/20}+ \sum_{m\in\mathcal{I}} \int_{I_m}\abs{\nabla V(\la)}\dd t\,,
\end{split}
\ee
when $n$ is large enough (here the `large enough' is random but this will not play a role in what follows). We now make the choice $$\delta=n^{-9/10}\,.$$ For the first term on the right hand side of \eqref{eq: diffla2}, we have by Lemma \ref{lem: cardI} below that $\E \abs{\mathcal{I}}n^{-9/20}\to0$ as $n\to \infty$. Therefore, $\abs{\mathcal{I}}n^{-9/20}\to0$ almost surely (along a subsequence).  For the second term on the right hand side of \eqref{eq: diffla2}, note that
\be
	 \sum_{m\in\mathcal{I}} \int_{I_m}\abs{\nabla V(\la)}\dd t = \int_{A_n} \abs{\nabla V(\la)}\dd t \,.
\ee
Here $A_n$ is a random subset of $[0,1]$ whose Lebesgue measure equals $\abs{\mathcal{I}}/n$.  Again, Lemma \ref{lem: cardI} implies that the Lebesgue measure  of $A_n$ goes to $0$ almost surely along (the same as above) subsequence. It follows from the absolute integrability of $\absa{\nabla V(\la)}$ that, 
along this subsequence, the second term on the right hand side of \eqref{eq: diffla2} goes to $0$, almost surely. Therefore, there is a subsequence $(n_k)\subset{\N}$ such that
\be
	\lim_{k\to \infty} \sup_{1\leq l \leq n_k} \abs{\mu(\tfrac{l}{n_k})} =0\,,\quad \text{a.s.}
\ee
In view of the H{\"o}lder continuity of $\mu(t)$, we have almost surely
\be
	\mu(t)=0\,, \forall t\in[0,1]\,.
\ee
By definition \eqref{def: mut} of $\mu$, we know that
\be
	\la(t)= \la(0)+\int_0^t (\dd M -\nabla V(\la)\dd t)\,.
\ee
This implies that $\la(t)$ is a strong solution to \eqref{eq: dla a}.  The fact that 
\be	
	\P\{ \la_1(t)<\cdots<\la_N(t)\,,\text{ for almost all } t\in[0,1]\}=1
\ee
follows immediately from the following estimate (see Proposition \ref{prop: tdX}),
\be
	\E \int_0^t \sum_{i\neq j}\frac{1}{\abs{\la_i-\la_j}}\leq c_N\,.
\ee

Finally it remains to prove uniqueness of the strong solution.  Suppose there is another strong solution $\hat{\la}(t)$ almost surely satisfying $\hat{\la}_1(t)\leq \cdots \leq \hat{\la}_N(t)$ for all $t\in[0,1]$ .  We have an equation for the difference $\la(t)-\hat{\la}(t)$:
\be
	\dd (\la(t)-\hat{\la}(t)) =\left( -\nabla V(\la(t))+\nabla V(\hat{\la}(t))\right) \dd t\,.
\ee
The square of Euclidean norm of $\la(t)-\hat{\la}(t)$ satisfies
\be
	\frac{\dd}{\dd t} \abs{\la(t)-\hat{\la}(t)}^2 = -\left\langle \nabla V(\la(t))-\nabla V(\hat{\la}(t)), \la(t)-\hat{\la}(t)\right\rangle \,.
\ee
The right hand side is non-positive, because $V$ is a convex function.  Hence, $\abs{\la(t)-\hat{\la}(t)}^2\leq \abs{\la(0)-\hat{\la}(0)}^2=0$ for all $t\in[0,1]$, which implies that $\la(t)=\hat\la(t)$ for all $t\geq 0$.
 
\end{proof}

In the proof above, we used the following lemma that gives a bound on the cardinality of $\mathcal{I}$ (see \eqref{def: I}). For the reader's convenience we restate the definition of $\mathcal{I}$ in the lemma.

\begin{lem}\label{lem: cardI}
	Let $\delta= n^{-9/10}$, $n\geq N^2$.  Let
	\ben
		\mathcal{I}:= \{i: \exists t\in [(i-1)/n,i/n)\,,\exists 1\leq k< N, \text{ s.t. } \abs{\la_k(t)-\la_{k+1}(t)} \leq \delta\}\,.
	\ee
	Then, for $n$ large enough,
	\ben
		\E \abs{\mathcal{I}} \leq cn^{1/5}\,.
	\ee
\end{lem}

\begin{proof}
	Observe that $\mathcal{I}=\sum_{l=1}^{n}\ind{l\in\mathcal{I}}$, so
	\be\label{eq: EI}
	\E \abs{\mathcal{I}} =\sum_{l=1}^n \E \ind{l\in\mathcal{I}} = \sum_{l=1}^n \P[{l\in\mathcal{I}}] \,.
	\ee
	For every fixed $1\leq l\leq n$, we compute the probability of the event $\{l\in\mathcal{I}\}$. We choose $n$ equally spaced points $t_0<\cdots<t_{n-1}$ on the interval $[(l-1)/n,l/n)$, given by
	\ben
	  t_j:= \frac{l-1}{n}+\frac{j}{n^2 }\,.
	\ee
	For each sample path $\om\in\{l\in\mathcal{I}\}$, either the smallest gap between $\la_1(t_j),\cdots,\la_N(t_j)$ gets smaller than $3\delta$ for some $j$, or the path vibrates dramatically near some $t_j$. To be precise, we have
	\ben
		\begin{split}
		\{l\in&\mathcal{I}\}\subset \left(\cup_j \{ \exists 1\leq k< N, \text{ s.t. } \abs{\la_k(t_j)-\la_{k+1}(t_j)} \leq 3\delta  \} \right)\\
		&\cup (\cup_j\{ \exists t \in [t_{j-1},t_j) \,, \exists 1\leq k\leq N, \text{ s.t. } \,, \abs{\la_k(t)-\la_k(t_{j-1})} \geq \delta \})\\
		&=: (\cup_j \mathcal{A}_{lj}) \cup (\cup_j\mathcal{B}_{lj})\,.
		\end{split}
	\ee
	According to Proposition \ref{prop: tdX}, we have $\P[\mathcal{A}_{lj}] \leq c\delta^2$, therefore
	\be\label{ineq: Pl}
		\P[l\in\mathcal{I}] \leq \sum_{ 1\leq j\leq n} \P[\mathcal{A}_{lj}]+\sum_{ 1\leq j\leq n}\P[\mathcal{B}_{lj}]\leq cn\delta^2+\sum_{ 1\leq j\leq n}\P[\mathcal{B}_{lj}]\,.
	\ee
	Now we look at $\P[\mathcal{B}_{lj}]$.  Since the eigenvalues of a Hermitian matrix are Lipschitz functions of the matrix entries,  and according to the definition of $\tilde{H}(t)$, we have
	\ben
		\mathcal{B}_{lj} \subset \{ \exists t \in  [t_{j-1},t_j)  \text{, s.t. }  \norm{U(t)-U(t_{j-1})} \geq c\delta \}\,.
	\ee
	Setting $\vartheta=1/n^2$ in the second estimate of Theorem \ref{thm: Us-Ut}, we see that $\P[\mathcal{B}_{lj}]\leq \e^{-cn^{1/5}}$.
	Therefore, \eqref{ineq: Pl} yields
	\ben
		\P[l\in\mathcal{I}] \leq cn^{-4/5}+n\e^{-cn^{1/20}}\leq c_1n^{-4/5}\,.
	\ee
	Then \eqref{eq: EI} implies
	\ben
		\E \abs{\mathcal{I}}\leq n\cdot c_1n^{-4/5} \leq c_1n^{1/5}\,.
	\ee
\end{proof}

\bibliography{FreeConv.bib}
\bibliographystyle{abbrv}
\end{document}